\documentclass[a4paper]{amsart}
\usepackage[T1]{fontenc}
\usepackage[utf8]{inputenc}
\usepackage[english]{babel}
\usepackage{amsmath}
\usepackage{amssymb}
\usepackage{amsfonts}
\usepackage{amsrefs}
\usepackage{MnSymbol}
\usepackage{stmaryrd}
\usepackage{mathrsfs} 
\usepackage{natbib}
\usepackage{graphicx,graphics}
\usepackage{slashed}
\usepackage{esint}
\usepackage{color}

\allowdisplaybreaks

\newcommand{\dcup}{\mathop{{\declareslashed{}{\boldsymbol\cdot}{0}{0.1}{\cup}\slashed{\cup}}}}

\DeclareMathSymbol{\eth} {\mathord}{AMSb}{"67}

\renewcommand{\phi}{\varphi}

\renewcommand{\rho}{\varrho}

\renewcommand{\theta}{\vartheta}

\renewcommand{\d}{\partial}
\newcommand{\vol}{\ensuremath{\mathrm{vol}}}
\newcommand{\open}{\ensuremath{\mathrm{open}}}

\renewcommand{\Re}{\ensuremath{\mathfrak{Re}}}

\newcommand{\eps}{\ensuremath{\varepsilon}}
\newcommand{\ex}{\exists}
\newcommand{\fa}{\forall}
\newcommand{\lnorm}{\left\lVert}
\newcommand{\rnorm}{\right\lVert}
\newcommand{\lbetr}{\left\lvert}
\newcommand{\rbetr}{\right\lvert}
\newcommand{\norm}[1]{\lnorm {#1}\rnorm}
\newcommand{\betr}[1]{\lbetr {#1}\rbetr}

\renewcommand{\l}{\ensuremath{\left}}
\renewcommand{\r}{\ensuremath{\right}}
\newcommand{\sse}{\ensuremath{\subseteq}}
\newcommand{\spt}{\opn{spt}}
\newcommand{\loc}{\opn{loc}}

\renewcommand{\div}{\opn{div}}

\newcommand{\tr}{\ensuremath{\opn{tr}}}
\newcommand{\grad}{\ensuremath{\opn{grad}}}

\newcommand{\connected}{\ensuremath{\mathrm{connected}}}

\newcommand{\ilc}{\ensuremath{\mathrm{ilc}}}
\newcommand{\oilc}{\ensuremath{\mathrm{oilc}}}
\newcommand{\diag}{{\ensuremath{\opn{diag}}}}

\newcommand{\res}{\ensuremath{\opn{res}}}

\newcommand{\gf}{\ensuremath{\mathfrak{g}}}

\renewcommand{\det}{\ensuremath{\opn{det}}}

\newcommand{\fp}{\ensuremath{\mathfrak{fp}}}
\newcommand{\dR}{{\ensuremath{\mathrm{dR}}}}

\newcommand{\ubr}{\underbrace}

\newcommand{\opn}{\operatorname}
\newcommand{\sgn}{\operatorname{sgn}}
\renewcommand{\iff}{\ensuremath{\Leftrightarrow}}
\newcommand{\then}{\ensuremath{\Rightarrow}}
\newcommand{\neht}{\ensuremath{\Leftarrow}}
\newcommand{\id}{\opn{id}}
\newcommand{\nn}[1][{}]{\ensuremath{\mathbb{N}^{#1}}}
\newcommand{\rn}[1][{}]{\ensuremath{\mathbb{R}^{#1}}}

\newcommand{\Hn}[1][{}]{\ensuremath{\mathbb{H}^{#1}}}
\newcommand{\cn}[1][{}]{\ensuremath{\mathbb{C}^{#1}}}

\newcommand{\zn}[1][{}]{\ensuremath{\mathbb{Z}^{#1}}}
\newcommand{\Tn}[1][{}]{\ensuremath{\mathbb{T}^{#1}}}
\newcommand{\Kn}[1][{}]{\ensuremath{\mathbb{K}^{#1}}}
\newcommand{\Zn}[1][{}]{\ensuremath{\mathbb{Z}^{#1}}}

\newcommand{\Mp}[1][{}]{\ensuremath{\mathcal{M}^{#1}}}
\newcommand{\Ap}[1][{}]{\ensuremath{\mathcal{A}^{#1}}}

\newcommand{\Dp}[1][{}]{\ensuremath{\mathcal{D}^{#1}}}
\newcommand{\Fp}[1][{}]{\ensuremath{\mathcal{F}^{#1}}}
\newcommand{\Lp}[1][{}]{\ensuremath{\mathcal{L}^{#1}}}
\newcommand{\Sp}[1][{}]{\ensuremath{\mathcal{S}^{#1}}}

\newtheorem{theorem}{Theorem}[section]
\newtheorem{lemma}[theorem]{Lemma}
\newtheorem{prop}[theorem]{Proposition}
\newtheorem{obs}[theorem]{Observation}
\newtheorem{corollary}[theorem]{Corollary}
\newtheorem{definition}[theorem]{Definition}

\newenvironment{theorem*}[1][Theorem]{\begin{trivlist}
\item[\hskip \labelsep {\bfseries #1}]\itshape}{\normalfont\end{trivlist}}
\newenvironment{corollary*}[1][Corollary]{\begin{trivlist}
\item[\hskip \labelsep {\bfseries #1}]\itshape}{\normalfont\end{trivlist}}
\newenvironment{example*}[1][Example]{\begin{trivlist}
\item[\hskip \labelsep {\bfseries #1}]}{\nopagebreak\flushright$\filledmedsquare$\end{trivlist}}
\newenvironment{remark*}[1][Remark]{\begin{trivlist}
\item[\hskip \labelsep {\bfseries #1}]}{\nopagebreak\flushright$\filledmedsquare$\end{trivlist}}

\renewcommand{\cite}[1]{{[\cites{#1}}}

\begin{document}

\title[A generalized KV trace for FIOs]{A generalized Kontsevich-Vishik trace for Fourier Integral Operators and the Laurent expansion of $\zeta$-functions}

\author{Tobias Hartung}
\address{Department of Mathematics, King's College London, Strand, London WC2R 2LS, United Kingdom}
\email{tobias.hartung@kcl.ac.uk}

\author{Simon Scott}
\address{Department of Mathematics, King's College London, Strand, London WC2R 2LS, United Kingdom}
\email{simon.scott@kcl.ac.uk}

\keywords{Kontsevich-Vishik trace, Fourier Integral Operators, $\zeta$-functions, Laurent expansion, gauged poly-$\log$-homogeneous distributions}
\subjclass[2010]{35S30, 58J40, 46F10}

\date{\today}

\begin{abstract}
  Based on Guillemin's work on gauged Lagrangian distributions, we will introduce the notion of a poly-$\log$-homogeneous distribution as an approach to $\zeta$-functions for a class of Fourier Integral Operators which includes cases of amplitudes with asymptotic expansion $\sum_{k\in\nn}a_{m_k}$ where each $a_{m_k}$ is $\log$-homogeneous with degree of homogeneity $m_k$ but violating $\Re(m_k)\to-\infty$. We will compute the Laurent expansion for the $\zeta$-function and give formulae for the coefficients in terms of the phase function and amplitude as well as investigate generalizations to the Kontsevich-Vishik quasi-trace. Using stationary phase approximation, series representations for the Laurent coefficients and values of $\zeta$-functions will be stated explicitly. Additionally, we will introduce an approximation method (mollification) for $\zeta$-functions of Fourier Integral Operators whose symbols have singularities at zero by $\zeta$-functions of Fourier Integral Operators with regular symbols.
\end{abstract}

\maketitle

\tableofcontents

\section{Introduction}
In \cite{guillemin-lagrangian}, Guillemin showed the existence of $\zeta$-functions of gauged Lagrangian distributions, investigated their residues, and used the residues to study the commutator structure of certain algebras of Fourier Integral Operators. Guillemin also extended the residue trace (cf. \cite{wodzicki}) to Fourier Integral Operators which allowed for many special cases to be studied; e.g. the class of Toeplitz operators (cf. \cite{boutet-de-monvel-guillemin}), wave traces (cf. e.g. \cite{chazarain,guillemin-wave,guillope-zworski,zelditch}), and operators with $\log$-terms (cf. e.g. \cite{lesch}). However, many questions about $\zeta$-functions are still to be answered. For instance, whether there is a natural extension of the Kontsevich-Vishik trace (cf. \cite{kontsevich-vishik}). In particular, wave traces are a prime example of a Kontzevich-Vishik trace for Fourier Integral Operators and thus the primary motivation to study extensions of the Kontsevich-Vishik trace. Other questions may revolve around $\zeta$-determinants or other traces induced by the $\zeta$-function. 

For such questions, knowing the Laurent expansion would be very helpful. Furthermore, it would be interesting to know in itself how the Laurent expansion of $\zeta$-functions of Fourier Integral Operators relates to the special case of pseudo-differential operators (cf. \cite{paycha-scott}). In order to compute the Laurent coefficients, taking derivatives, i.e. being able to handle $\log$-terms, will be crucial. We will, therefore, assume a generalized approach and define the notion of a gauged poly-$\log$-homogeneous distribution which is based on Guillemin's approach in \cite{guillemin-lagrangian}. It is interesting to note that all the cases above are covered and some other cases (including some relaxations which might be advantageous in explicit calculations) can be considered, as well.

Operator $\zeta$-functions were introduced by Ray and Singer \cite{ray,ray-singer} using Seeley's work on complex powers of elliptic pseudo-differential operators \cite{seeley}. In order to regularize the trace of an operator $A$, Ray and Singer considered the map $\cn\ni z\mapsto\tr A^{-z}\in\cn$. Its meromorphic extension $\zeta_\sigma(A)$ is called the spectral $\zeta$-function of $A$. Since raising an operator to a complex power is not always possible (even if the algebra has the holomorphic functional calculus), one often considers (generalized) $\zeta$-functions $\zeta(A)$ which are meromorphic extensions of $\cn\ni z\mapsto\tr A(z)\in\cn$ for a suitable operator family $A$. In particular, $A(z)=A_0Q^z$ with a suitably chosen $Q$ are well-studied. 

If each $A(z)$ has a polyhomogeneous amplitude $a\sim\sum_{j\in\nn_0}a_{m-j+z}$, then $\zeta(A)$ exists as a meromorphic function on $\cn$ and has only isolated simple poles. The residue at zero is called the Wodzicki residue of $A_0$ \cite{wodzicki} and, in general, it is the (projectively) unique continuous trace on an algebra of Fourier Integral Operators \cite{guillemin-lagrangian,guillemin-residue-traces}. Furthermore, the evaluation $\zeta(A)(0)$ defines the regularized trace of $A(0)$ (provided $\zeta(A)$ does not have a pole in zero). This (unbounded) trace was studied by Kontsevich and Vishik \cite{kontsevich-vishik} and is, thus, called the Kontsevich-Vishik trace. The Kontsevich-Vishik trace is particularly interesting since it is the only trace on the algebra of classical pseudo-differential operators which restricts to the canonical trace in the space of bounded operators $L(L_2(X))$ on $L_2$-functions on a closed manifold $X$ \cite{maniccia-schrohe-seiler}.

More precisely, let $A_0$ be a pseudo-differential operator with amplitude $a\sim\sum_{j\in\nn_0}a_{d-j}$ such that each $a_{d-j}$ is homogeneous of degree $d-j$ for a given $d\in\cn\setminus\zn_{\ge-\dim X}$ and $N\in\nn_{0,>\Re(d)+1}$. The Kontsevich-Vishik trace is then given by
\begin{align*}
  \tr_{KV}A_0=\int_X\int_{\rn[\dim X]}a(x,x,\xi)-\sum_{j=0}^Na_{d-j}(x,x,\xi)\ d\xi\ d\vol_X(x)
\end{align*}
and independent of $N$. 

While Guillemin \cite{guillemin-lagrangian} showed that $\zeta$-functions for Fourier Integral Operators with classical amplitudes exist as meromorphic functions with isolated simple poles and, then, studied the generalized residue trace, the focus of the present paper will be the constant Laurent coefficient and the generalization of the Kontsevich-Vishik trace. Examples of the generalized Kontsevich-Vishik trace have been studied in form of wave traces $t\mapsto\tr\exp\l(it\sqrt{\betr{\Delta_M}}\r)$ where $\Delta_M$ is the Laplacian on a manifold $M$ (cf. e.g. \cite{chazarain,guillemin-wave,guillope-zworski,zelditch}) or, more generally, $t\mapsto\tr\exp\l(-it\sqrt[m]P\r)$ for a positive elliptic self-adjoint pseudo-differential operator $P$ of order $m>0$ \cite{duistermaat-guillemin}. A particularly interesting result \cite{duistermaat-guillemin}*{equation (0.2)} is the residue formula if all periodic solution curves of the Hamiltonian vector field are isolated and non-degenerate:
\begin{align*}
  \lim_{t\to T}(t-T)\tr\exp\l(-it\sqrt[m]P\r) = \sum_\gamma\frac{T_{0\gamma}}{2\pi}i^{\sigma_\gamma}\betr{1-P_\gamma}^{-\frac{1}{2}}
\end{align*}
where the sum is taken over all integral curves $\gamma$ of period $T$, $T_{0\gamma}$ is the smallest positive period of $\gamma$, $\sigma_\gamma$ is a Maslov factor, and $P_\gamma$ the Poincar\'e map around $\gamma$. Furthermore, Guillop\'e and Zworski \cite{guillope-zworski} studied wave traces on Riemann surfaces and proved a Selberg trace formula for the wave group by introducing the $0$-integral which regularizes in geometric terms and is based on the $0$-calculus of Mazzeo and Melrose \cite{mazzeo-melrose}.

In this paper, we aim to study $\zeta$-regularization of Fourier Integral Operator traces in general. In particular, we will compute the Laurent expansion of Fourier Integral Operator $\zeta$-functions and study the generalized Kontsevich-Vishik trace since it is essentially the constant Lauren coefficient. 

We will consider the (at first quite restrictive looking) notion of gauged poly-$\log$-homogeneous distributions which only contain holomorphic families $A$ such that the degrees of homogeneity $d$ in the expansion are of the form 
\begin{align*}
  \fa z\in\cn:\ d(z)=d(0)+z.
\end{align*}
As it turns out, this will be sufficient as the most general families we can consider (these are holomorphic families $A$ in an open, connected subset of $\cn$ where the degrees of homogeneity are non-constant holomorphic functions) are germ equivalent to this special form and, hence, all local properties are shared, that is, in particular, the Laurent expansion.

In sections \ref{gauged-poly-log-homogeneous-distributions}-\ref{lagrangian} we will compute the Laurent expansion and apply it to Fourier Integral Operators whose amplitudes have no singularities. This will also yield the generalized Kontsevich-Vishik density and trace by removing the critical degree of homogeneity terms. In fact, this is the unique extension of the Kontsevich-Vishik density that is globally defined (if there are no critical degrees of homogeneity) and it is the only extension for which the trace coincides with the value of the $\zeta$-function (if there are no critical degrees of homogeneity). However, in the generalized case of Fourier Integral Operators, splitting off $\sum_{j=0}^Na_{d-j}$ is not possible anymore since these terms do not regularize to zero (as is the case for pseudo-differential operators). Instead, the generalized Kontsevich-Vishik trace will contain all terms that do not contribute to poles of the $\zeta$-function.

Using the Laurent expansion, we can reproduce many well-known facts about $\zeta$-functions of pseudo-differential operators and Fourier Integral Operators like \cite{kontsevich-vishik}*{equation (2.21)}, \cite{paycha}*{equation (9)}, \cite{paycha-scott}*{equations (0.12), (0.14), (0.17), (0.18), and (2.20)}. 

In section \ref{sec:mollification}, we will introduce an approximation method, which we call mollification, to extend the results to Fourier Integral Operators with asymptotic expansions which have singularities at zero and allow classical amplitudes. 

Furthermore, we will have a closer look at the coefficients in section \ref{sec:structural singularities}. For polyhomogeneous amplitudes, we will obtain the residue trace (as Guillemin has shown to exist). For poly-$\log$-homogeneous amplitudes we will find a generalization of the Kontsevich-Vishik trace and we can generalize Lesch's main statements about the residue trace and the Kontsevich-Vishik trace for pseudo-differential operators in \cite{lesch} to Fourier Integral Operators. We will show that both (the residue trace and the generalized Kontsevich-Vishik trace) induce globally well-defined densities on the underlying manifold (provided that we started with globally defined kernels). We will see that the Laurent coefficients vanish if and only if the corresponding term $e^{i\theta}a$ in the Schwartz kernel is a divergence on $X\times\d B_{\rn[N]}$

Finally, in section \ref{sec:stationary-phase}, we will use stationary phase approximation to treat the integrals 
\begin{align*}
  I(x,y,r)=\int_{\d B_{\rn[N]}}e^{ir\theta(x,y,\eta)}a(x,y,\eta)d\vol_{\d B_{\rn[N]}}(\eta)
\end{align*}
which appear as coefficients in the Laurent expansion for $r=1$. The stationary phase approximation also allows us to compute the kernel singularity structure of certain Fourier Integral Operators by integrating $I(x,y,r)$ over $r\in\rn_{>0}$. This yields many ``exotic'' algebras of Fourier Integral Operators which happen to be subsets of the Hilbert-Schmidt operators and $\zeta$-functions in such algebras have no poles (independent of the Hörmander class of the amplitude). Although it is a peculiar property of certain classes of Fourier Integral Operators that cannot occur with pseudo-differential operators, these algebras are still very natural; e.g. they appear as terms after pushing down a pseudo-differential operator onto a quotient manifold. 

The kernel singularity structure also allows us to produce analogues of Boutet de Monvel's result that the residue trace is the trace of the logarithmic coefficient for a certain class of Fourier Integral Operators \cite{boutet-de-monvel}*{equations (3) and (4)}. 

In addition to Boutet de Monvel's result, we can also compute the Kontsevich-Vishik trace. In the case of \cite{boutet-de-monvel} (one dimensional Fourier integrals on the half-line bundle with phase function satisfying $\theta(x,x,r)=0$), we will see that the generalized Kontsevich-Vishik trace reduces to the pseudo-differential form. More precisely, let $A$ have the amplitude $a\sim\sum_{j\in\nn_0}a_{d-j}$, each $a_{d-j}$ homogeneous of degree $d-j$, $d\in\cn\setminus\zn_{\ge-1}$, and $N\in\nn_{0,>\Re(d)+1}$. Then,
\begin{align*}
  \tr_{KV}A=\int_X\int_{\rn_{>0}}a(x,x,r)-\sum_{j=0}^Na_{d-j}(x,x,r)\ dr\ d\vol_X(x)
\end{align*}
independent of $N$. This is still true for Fourier Integral Operators whose phase function $\theta$ satisfies $\fa x\in X\ \fa\xi\in\rn[N]:\ \theta(x,x,\xi)=0$. 

However, reduction to the pseudo-differential form is highly non-trivial and false in general. Consider, for instance,
\begin{align*}
  \int_X\int_{\rn}e^{i\Theta(x,x)r}r^{-n}dr d\vol_X(x)=&\int_X\frac{-i\pi(-2\pi i\Theta(x,x))^{n-1}\sgn(\Theta(x,x))}{(n-1)!}d\vol_X(x).
\end{align*}
If $\Theta(x,x)=1$ and $n=4$, then this term reduces to $\frac{4\pi^4\vol(X)}{3}$. In other words, such a term would violate independence of $N$.

In short, section \ref{gauged-poly-log-homogeneous-distributions} computes the Laurent expansion in terms of the abstract notion of gauged poly-$\log$-homogeneous distributions. The Laurent expansion will be applied to Fourier Integral Operator traces in section \ref{lagrangian} and the method of mollification in section \ref{sec:mollification} extends the Laurent expansion to amplitudes that are homogeneous in $\rn[N]\setminus\{0\}$, i.e. allowing for classical amplitudes. In section \ref{sec:structural singularities}, we will identify the generalized Kontsevich-Vishik trace and study general characteristics of vanishing Laurent coefficients. Section \ref{sec:stationary-phase} will focus on the stationary phase approximation which allows us to compute the Laurent coefficients and the Kontsevich-Vishik trace. Thus, most interesting examples, like the Kontsevich-Vishik trace of Fourier Integral Operators considered by Boutet de Monvel in \cite{boutet-de-monvel} and a class of Fourier Integral Operators that contains (rather surprisingly) only Hilbert Schmidt operators, are found at the end of section \ref{sec:stationary-phase}. Appendix \ref{examples} contains some examples, as well; however these are intended to be direct applications of the Laurent expansion which can be computed and independently/easily checked by hand. 

Lastly, we would like to note that this is a reduced version of the work laid out in this article. Many proofs that use standard techniques, i.e. proofs that are mutatis mutandis compared to the pseudo-differential equivalent, have been omitted in the interest of brevity. Similarly, a number of additional results that are well-known from the pseudo-differential theory and not essential for the present paper have been excluded. For these results, we would like to refer to the longer version of this article \cite{hartung-scott} or TH's Ph.D. thesis \cite{hartung-phd}.

\section{Gauged poly-$\log$-homogeneous distributions}\label{gauged-poly-log-homogeneous-distributions}
In this section, we will introduce the notion of gauged poly-$\log$-homogeneous distributions and their $\zeta$-functions. These distributions and $\zeta$-functions generalize the Fourier Integral Operator $\zeta$-functions while still maintaining most of the analytical structure. Furthermore, $\zeta$-functions of gauged poly-$\log$-homogeneous distributions are more accessible using  methods developed for $\zeta$-functions of pseudo-differential operators. In fact, Guillemin already used this in his work on the existence of Fourier Integral Operator $\zeta$-functions and their residue trace \cite{guillemin-lagrangian}. In order to tap more of their potential, we will formally introduce and study gauged poly-$\log$-homogeneous distributions and their $\zeta$-functions first. These general considerations will be applied to gauged Schwartz kernels and gauged Lagrangian distributions, as Guillemin has studied in \cite{guillemin-lagrangian}, in section \ref{lagrangian}.

Consider integrals of the form
\begin{align*}
  \int_{\rn_{\ge1}\times M} \alpha(z)(\xi)d\vol_{\rn_{\ge1}\times M}(\xi)
\end{align*}
where $M$ is an orientable,\footnote{Replacing $\alpha(z)(r,\xi) d\vol_{\rn_{\ge1}\times M}(r,\xi)$ by some family $d\omega(z)(r,\xi)$ allows us to also treat non-orientable manifolds but we will not need this in the following and choose orientability for the sake of simplicity.} compact, finite dimensional manifold without boundary and $\alpha$ is a holomorphic family given by an expansion\footnote{This is not meant to be an asymptotic expansion but an actual identity. However, for a classical symbol $a$ with asymptotic expansion $\sum_{j\in\nn}a_j$ where $a_j$ is homogeneous of degree $m-j$ for some $m\in\cn$, it is possible to choose a finite set $I=\{0,1,\ldots,J\}$ and $\alpha_0$ will correspond to $a-\sum_{j=0}^Ja_{m-j}$. 

This is completely analogous to the Kontsevich-Vishik trace, i.e. splitting off finitely many terms with large degrees of homogeneity while the rest is integrable. The only difference is that those terms (that have been split off) might not regularize to zero anymore.}
\begin{align*}
  \alpha=\alpha_0+\sum_{\iota\in I}\alpha_\iota
\end{align*}
where $I\sse\nn$, $\alpha_0(z)\in L_1(\rn_{\ge1}\times M)$ in an open neighborhood of $\l\{z\in\cn;\ \Re(z)\le 0\r\}$ and each of the $\alpha_\iota(z)$ is $\log$-homogeneous with degree of homogeneity $d_\iota+z\in\cn$ and logarithmic order $l_\iota\in\nn_0$, that is,
\begin{align*}
  \ex\tilde\alpha_\iota\in\cn[M]\ \fa r\in\rn_{\ge1}\ \fa\nu\in M:\ \alpha_\iota(z)(r,\nu)=r^{d_\iota+z}(\ln r)^{l_\iota}\tilde\alpha_\iota(z)(\nu).
\end{align*}
We will furthermore assume the following.

\begin{itemize}
\item The family $(\Re(d_\iota))_{\iota\in I}$ is bounded from above. (Note, we do not require $\Re(d_\iota)\to-\infty$. $\fa\iota\in I:\ \Re(d_\iota)=42$ is entirely possible.)
\item The map $I\ni\iota\mapsto(d_\iota,l_\iota)$ is injective.
\item There are only finitely many $\iota$ satisfying $d_\iota=d$ for any given $d\in\cn$. 
\item The family $((d_\iota-\delta)^{-1})_{\iota\in I}$ is in $\ell_2(I)$ for any $\delta\in\cn\setminus\{d_\iota;\ \iota\in I\}$. 
\item Each $\sum_{\iota\in I}\tilde\alpha_{\iota}(z)$ converges unconditionally in $L_1(M)$. 
\end{itemize}
Any such family $\alpha$ will be called a gauged poly-$\log$-homogeneous distribution. Note that the generic case (that is, applications to Fourier Integral Operators with amplitudes of the form $a\sim\sum_{j\in\nn_0}a_{m-j}$) implies that $I$ is a finite set and all these conditions are, therefore, satisfied.

\begin{example*}
  Let $A(z)$ be a pseudo-differential operator on an $N$-dimensional manifold $X$ whose amplitude has an asymptotic expansion $a(z)\sim\sum_{j\in\nn}a_j(z)$ where each $a_j(z)$ is homogeneous of degree $m-j+z$. Then, we may want to evaluate the meromorphic extension of
  \begin{align*}
    \tr A(z)=&\int_X\int_{\rn[N]}a(z)(x,x,\xi)d\xi d\vol_X(x)\\
    =&\int_X\int_{\rn_{\ge1}\times\d B_{\rn[N]}}a(z)(x,x,\xi)d\xi d\vol_X(x)\\
    &+\int_X\int_{B_{\rn[N]}(0,1)}a(z)(x,x,\xi)d\xi d\vol_X(x)
  \end{align*}
  at zero. The poly-$\log$-homogeneous distribution here is
  \begin{align*}\tag{$*$}
    \int_X\int_{\rn_{\ge1}\times\d B_{\rn[N]}}a(z)(x,x,\xi)d\xi d\vol_X(x).
  \end{align*}
  At this point, we have many possibilities to write ($*$) in the form 
  \begin{align*}
    \int_{\rn_{\ge1}\times M} \alpha(z)(\xi)d\vol_{\rn_{\ge1}\times M}(\xi).
  \end{align*}
  The easiest choice is $M:=\d B_{\rn[N]}$ and $I:=\{j\in\nn;\ \Re(m)-j\ge-N\}$. This ensures that 
  \begin{align*}
    \int_Xa(z)(x,x,\xi)-\sum_{j\in I}a_j(z)(x,x,\xi)d\vol_X(x)
  \end{align*}
  is integrable in $\rn_{\ge1}\times\d B_{\rn[N]}$. Furthermore, having a finite $I$ ensures that all of the conditions above are satisfied and $\alpha$ can be defined by
  \begin{align*}
    \alpha_0(z)(r,\nu):=&\int_Xa(z)(x,x,r\nu)-\sum_{j\in I}a_j(z)(x,x,r\nu) d\vol_X(x)
  \end{align*}
  and 
  \begin{align*}
    \alpha_j(z)(r,\nu):=&\int_Xa_j(z)(x,x,r\nu) d\vol_X(x)
    =r^{m-j+z}\ubr{\int_Xa_j(z)(x,x,\nu) d\vol_X(x)}_{=:\tilde\alpha_j(z)(\nu)}
  \end{align*}
  for $j\in I$.

\end{example*}

\begin{remark*}
  Note that these distributions are strongly connected to traces of Fourier Integral Operators, as well. In fact, Guillemin's argument in \cite{guillemin-lagrangian} relies heavily on the fact that the dual pairing $\langle u(z),f\rangle$ at question are integrals of the form
  \begin{align*}
    \int_{\rn_{\ge1}\times \d B_{\rn[N]}} \alpha(z)(\xi)d\vol_{\rn_{\ge1}\times \d B_{\rn[N]}}(\xi)
  \end{align*}
  where $\alpha$ is a gauged polyhomogeneous distribution; cf. \cite{guillemin-lagrangian}*{equation (2.15)}.

\end{remark*}

If the conditions above are satisfied, we obtain formally
\begin{align*}
  \int_{\rn_{\ge1}\times M} \alpha(z)d\vol_{\rn_{\ge1}\times M}=&\ubr{\int_{\rn_{\ge1}\times M} \alpha_0(z)d\vol_{\rn_{\ge1}\times M}}_{=:\tau_0(z)\in\cn}+\sum_{\iota\in I}\int_{\rn_{\ge1}\times M} \alpha_\iota(z)d\vol_{\rn_{\ge1}\times M}\\
  =&\tau_0(z)+\sum_{\iota\in I}\int_{\rn_{\ge1}}\int_M \alpha_\iota(z)(\rho,\nu)\rho^{\dim M}d\vol_M(\nu)d\rho\\
  =&\tau_0(z)+\sum_{\iota\in I}\ubr{\int_{\rn_{\ge1}}\rho^{\dim M+d_\iota+z}\l(\ln\rho\r)^{l_\iota}d\rho}_{=:c_\iota(z)}\ubr{\int_M\tilde\alpha_\iota(z)d\vol_M}_{=:\res\alpha_\iota(z)\in\cn}\\
  =&\tau_0(z)+\sum_{\iota\in I}c_\iota(z)\res\alpha_\iota(z)
\end{align*}
which now needs to be justified.
\begin{lemma}\label{c-iota}
  $c_\iota(z)=(-1)^{l_\iota+1}l_\iota!\l(\dim M+d_\iota+z+1\r)^{-(l_\iota+1)}$
\end{lemma}
\begin{proof}
  Let $\Gamma_{ui}$ be the upper incomplete $\Gamma$-function given by the meromorphic extension of
  \begin{align*}
    \Gamma_{ui}(s,x):=\int_x^\infty t^{s-1}e^{-t}dt\qquad(\Re(s)>0,\ x\in\rn_{\ge0}).
  \end{align*}
  $\Gamma_{ui}$ satisfies $\Gamma_{ui}(s,0)=\Gamma(s)$ where $\Gamma$ denotes the (usual) $\Gamma$-function, $\Gamma(s,\infty)=0$, and $\d_2\Gamma_{ui}(s,x)=-x^{s-1}e^{-x}$. Then, we obtain
  \begin{align*}
    \l(\rn_{>0}\ni y\mapsto\frac{-\Gamma_{ui}(l+1,-(d+1)\ln y)}{(-(d+1))^{l+1}}\r)'(x)=&x^d(\ln x)^l.
  \end{align*}
  Hence, for $d<-1$,
  \begin{align*}
    \int_{\rn_{\ge1}}x^d(\ln x)^ldx=&\frac{(-1)^{l+1}l!}{(d+1)^{l+1}}
  \end{align*}
  which yields
  \begin{align*}
    c_\iota(z)=&\int_{\rn_{\ge1}}\rho^{\dim M+d_\iota+z}\l(\ln\rho\r)^{l_\iota}d\rho=\frac{(-1)^{l_\iota+1}l_\iota!}{\l(\dim M+d_\iota+z+1\r)^{l_\iota+1}}
  \end{align*}
  in a neighborhood of $\rn_{<-\dim M-d_\iota-1}$ (because any real analytic function can be extended locally to a holomorphic function) and, thence, by meromorphic extension everywhere in $\cn\setminus\l\{-\dim M-d_\iota-z-1\r\}$.

\end{proof}
Since the $\res\alpha_\iota$ are holomorphic functions, we now know that $\sum_{\iota\in I}c_\iota\res\alpha_\iota$ is a meromorphic function with isolated poles only (if it converges), because $((d_\iota+\delta)^{-1})_{\iota\in I}\in\ell_2(I)$ implies that there may be at most finitely many $d_\iota$ in any compact subset of $\cn$.
\begin{lemma}
  For every $z\in\cn\setminus\{-\dim M-d_\iota-1;\ \iota\in I\}$, $\sum_{\iota\in I}c_\iota(z)\res\alpha_\iota(z)$ converges absolutely.
\end{lemma}
\begin{proof}
  By assumption, $(c_\iota(z))_{\iota\in I}\in\ell_2(I)$ and $\sum_{\iota\in I}\tilde\alpha_{\iota}(z)$ converges unconditionally in $L_1(M)$. This allows us to utilize the following theorem.
  \begin{theorem*}[{{Theorem 4.2.1} in \cite{kadets}}]
    Let $p\in\rn_{\ge1}$, $q=
    \begin{cases}
      2&,\ p\in[1,2]\\
      p&,\ p\in\rn_{>2}
    \end{cases}$, and $\sum_{j\in\nn}x_j$ converges unconditionally in $L_p$. Then, $\sum_{j\in\nn}\norm{x_j}_{L_p}^q$ converges.
  \end{theorem*}
  Hence,
  \begin{align*}
    \sum_{\iota\in I}\betr{c_\iota(z)\res\alpha_\iota(z)}\le&\sum_{\iota\in I}\betr{c_\iota(z)}\norm{\tilde\alpha_\iota(z)}_{L_1(M)}\\
    =&\norm{\l(\betr{c_\iota(z)}\norm{\tilde\alpha_\iota(z)}_{L_1(M)}\r)_{\iota\in I}}_{\ell_1(I)}\\
    \le&\norm{\l(\betr{c_\iota(z)}\r)_{\iota\in I}}_{\ell_2(I)}\norm{\l(\norm{\tilde\alpha_\iota(z)}_{L_1(M)}\r)_{\iota\in I}}_{\ell_2(I)}\\
    =&\norm{\l(c_\iota(z)\r)_{\iota\in I}}_{\ell_2(I)}\sqrt{\sum_{\iota\in I}\norm{\tilde\alpha_\iota(z)}_{L_1(M)}^2}<\infty.
  \end{align*}

\end{proof}
\begin{definition}
  Let $\alpha$ be a gauged poly-$\log$-homogeneous distribution. Then, we define the $\zeta$-function of $\alpha$ to be the meromorphic extension of
  \begin{align*}
    \zeta(\alpha)(z):=\int_{\rn_{\ge1}\times M} \alpha(z)d\vol_{\rn_{\ge1}\times M},
  \end{align*}
  i.e.
  \begin{align*}
    \zeta(\alpha)(z)=\tau_0(z)+\sum_{\iota\in I}\frac{(-1)^{l_\iota+1}l_\iota!\res\alpha_\iota(z)}{\l(\dim M+d_\iota+z+1\r)^{l_\iota+1}}.
  \end{align*}
\end{definition}
Now that we know $\zeta(\alpha)$ exists as a meromorphic function, we will compute its Laurent expansion.
\begin{definition}
  Let $f(z)=\sum_{n\in\zn}a_n(z-z_0)^n$ be a meromorphic function defined by its Laurent expansion at $z_0\in\cn$ without essential singularity at $z_0$, that is, $\ex N\in\zn\ \fa n\in\zn_{\le N}:\ a_n=0$. Then, we define the order of the initial Laurent coefficient $\oilc_{z_0}(f)$ of $f$ at $z_0$ to be
  \begin{align*}
    \oilc_{z_0}(f):=\min\{n\in\zn;\ a_n\ne0\}
  \end{align*}
  and the initial Laurent coefficient $\ilc_{z_0}(f)$ of $f$ at $z_0$
  \begin{align*}
    \ilc_{z_0}(f):=a_{\oilc_{z_0}(f)}.
  \end{align*}
\end{definition}

\begin{lemma}\label{gauge-independence-structural}
  Let $\alpha=\alpha_0+\sum_{\iota\in I}\alpha_\iota$ and $\beta=\beta_0+\sum_{\iota\in I'}\beta_\iota$ be two gauged poly-$\log$-homogeneous distributions with $\alpha(0)=\beta(0)$ and $\res\alpha_j(0)\ne0$ if $l_j$ is the maximal logarithmic order with $d_j=-\dim M-1$. Then, $\oilc_0(\zeta(\alpha))=\oilc_0(\zeta(\beta))$ and $\ilc_0(\zeta(\alpha))=\ilc_0(\zeta(\beta))$.

  In other words, $\oilc_0(\zeta(\alpha))$ and $\ilc_0(\zeta(\alpha))$ depend on $\alpha(0)$ only and are, thus, independent of the gauge.
\end{lemma}
\begin{proof}
  Since $\alpha(0)=\beta(0)$, we obtain that $z\mapsto\gamma(z):=\frac{\alpha(z)-\beta(z)}{z}$ is a gauged poly-$\log$-homogeneous distribution again. Furthermore, 
  \begin{align*}
    \oilc_0(\zeta(\gamma))\ge\min\{\oilc_0(\zeta(\alpha)),\oilc_0(\zeta(\beta))\}=:-l=-l_j-1
  \end{align*}
  holds because each pair $(d_\iota,l_\iota)$ in the expansion of $\gamma$ appears in at least one of the expansions of $\alpha$ or $\beta$. This implies that $z\mapsto z^l\zeta(\gamma)(z)=z^{l-1}\l(\zeta(\alpha)(z)-\zeta(\beta)(z)\r)$ is holomorphic at zero (equality holds for $\Re(z)$ sufficiently small and, thence, in general by meromorphic extension). Hence, the highest order poles of $\zeta(\alpha)$ and $\zeta(\beta)$ at zero must cancel out which directly implies $\oilc_0(\zeta(\alpha))=\oilc_0(\zeta(\beta))$ and $\ilc_0(\zeta(\alpha))=\ilc_0(\zeta(\beta))$.

\end{proof}

\begin{lemma}\label{gauge-independence-regular}
  Let $\alpha=\alpha_0+\sum_{\iota\in I}\alpha_\iota$ and $\beta=\beta_0+\sum_{\iota\in I'}\beta_\iota$ be two gauged poly-$\log$-homogeneous distributions with $\alpha(0)=\beta(0)$ and $\fa\iota\in I\cup I':\ d_\iota\ne-\dim M-1$. Then, $\zeta(\alpha)(0)=\zeta(\beta)(0)$.
\end{lemma}
\begin{proof}
  Again, since $\alpha(0)=\beta(0)$, we obtain that $z\mapsto\gamma(z):=\frac{\alpha(z)-\beta(z)}{z}$ is a gauged poly-$\log$-homogeneous distribution and $\oilc_0(\zeta(\gamma))\ge0$. Hence
  \begin{align*}
    \zeta(\alpha)(0)-\zeta(\beta)(0)=&\res_0\l(z\mapsto\frac{\zeta(\alpha)(z)-\zeta(\beta)(z)}{z}\r)=\res_0\zeta(\gamma)=0
  \end{align*}
  where $\res_0$ denotes the residue of a meromorphic function at zero.

\end{proof}

\begin{definition}
  Let $\alpha=\alpha_0+\sum_{\iota\in I}\alpha_\iota$ be a gauged poly-$\log$-homogeneous distribution and $I_{z_0}:=\{\iota\in I;\ d_\iota=-\dim M-1-z_0\}$. Then, we define
  \begin{align*}
    \fp_{z_0}(\alpha):=\alpha-\sum_{\iota\in I_{z_0}}\alpha_\iota=\alpha_0+\sum_{\iota\in I\setminus I_{z_0}}\alpha_\iota.
  \end{align*}
\end{definition}

\begin{corollary}\label{KV-gauge-independence}
  $\zeta(\fp_0\alpha)(0)$ is independent of the chosen gauge.
\end{corollary}

\begin{definition}
  Let $\alpha=\alpha_0+\sum_{\iota\in I}\alpha_\iota$ be a gauged poly-$\log$-homogeneous distribution and $\res\alpha_\iota\ne0$ for some $\iota\in I_0$. Then, we say $\zeta(\alpha)$ has a structural singularity at zero.
\end{definition}

\begin{remark*}
  Note that the pole structure of $\zeta(\alpha)$ does not only depend on the $\res\alpha_\iota$ but also on derivatives of $\alpha$. A structural singularity is a property of $\alpha(0)$ in the sense that it cannot be removed under change of gauge. More precisely, choosing $\beta$ such that $\alpha(0)=\beta(0)$ does not imply that the principal part of the Laurent expansion of $\zeta(\alpha)-\zeta(\beta)$ vanishes. However, if all $\res\alpha_\iota$ vanish ($\iota\in I_0$), then there exists a $\beta$ with $\alpha(0)=\beta(0)$ such that $\zeta(\beta)$ is holomorphic in a neighborhood of zero (e.g. $\beta$ being $\Mp$-gauged; see below). Having a non-vanishing $\res\alpha_\iota$ for some $\iota\in I_0$, on the other hand, implies that every $\zeta(\beta)$ with $\alpha(0)=\beta(0)$ has a pole at zero.

\end{remark*}

\begin{definition}
  Let $\alpha=\alpha_0+\sum_{\iota\in I}\alpha_\iota$ be a gauged poly-$\log$-homogeneous distribution. If all $\tilde\alpha_\iota$ are independent of the complex argument, i.e. $\alpha_\iota(z)(r,\nu)=r^{d_\iota+z}(\ln r)^{l_\iota}\tilde\alpha_\iota(0)(\nu)=r^z\alpha_\iota(0)(r,\nu)$, then we call this choice of gauge an $\Mp$-gauge (or Mellin-gauge).
\end{definition}

\begin{remark*}
  The $\Mp$-gauge for Fourier Integral Operators can always be chosen locally.

\end{remark*}

\begin{remark*}
  Suppose we have a gauged distribution $\alpha$ such that 
  \begin{align*}
    \fa z\in\cn\ \fa (r,\xi)\in\rn_{\ge1}\times M:\ \alpha(z)(r,\xi)=r^z\alpha(0)(r,\xi)
  \end{align*}
  is satisfied and we artificially continue $\alpha$ by zero to $\rn_{>0}\times M$. Then,
  \begin{align*}
    \int_{\rn_{>0}\times M}\alpha(z)(r,\xi)d\vol_{\rn_{>0}\times M}(r,\xi)=&\int_{\rn_{>0}}r^{\dim M+z}\ubr{\int_M\alpha(0)(r,\xi)d\vol_M(\xi)}_{=:A(r)}dr\\
    =&\Mp(A)(\dim M+z+1)
  \end{align*}
  holds where $\Mp$ denotes the Mellin transform
  \begin{align*}
    \Mp f(z)=\int_{\rn_{>0}}t^{z-1}f(t)dt
  \end{align*}
  for $f:\ \rn_{>0}\to\rn$ measurable whenever the integral exists. Hence, the name ``$\Mp$-gauge''.

\end{remark*}

\begin{corollary}
  Let $\alpha=\alpha_0+\sum_{\iota\in I}\alpha_\iota$ be a gauged poly-$\log$-homogeneous distribution.
  \begin{enumerate}
  \item[(i)] If $\alpha$ is $\Mp$-gauged, then all $\res\alpha_\iota$ are constants.
  \item[(ii)] If $\res\alpha_\iota(0)=0$ for $\iota\in I$, then the corresponding pole in $\zeta(\alpha)$ can be removed by re-gauging.
  \item[(iii)] If $\res\alpha_\iota(0)\ne0$ for $\iota\in I_0$, then the corresponding pole in $\zeta(\alpha)$ in independent from the gauge. In particular, $\res\alpha_\iota(0)$ does not depend on the gauge.
  \end{enumerate}
\end{corollary}
\begin{proof}
  \begin{enumerate}
  \item[(i)] trivial.
  \item[(ii)] The corresponding pole contributes the term $\frac{(-1)^{l_\iota+1}l_\iota!\res\alpha_\iota(z)}{\l(\dim M+d_\iota+z+1\r)^{l_\iota+1}}$ to the expansion of $\zeta(\alpha)$. Choosing an $\Mp$-gauge yields 
    \begin{align*}
      \frac{(-1)^{l_\iota+1}l_\iota!\res\alpha_\iota(z)}{\l(\dim M+d_\iota+z+1\r)^{l_\iota+1}}=\frac{(-1)^{l_\iota+1}l_\iota!\res\alpha_\iota(0)}{\l(\dim M+d_\iota+z+1\r)^{l_\iota+1}}=0
    \end{align*}
    by holomorphic extension.
  \item[(iii)] Lemma \ref{gauge-independence-structural} shows that $\oilc_0\zeta(\alpha_\iota)$ and $\ilc_0(\zeta(\alpha_\iota))$ are independent of the gauge. Since, $\res\alpha_\iota(0)\ne0$, we obtain $\oilc_0\zeta(\alpha_\iota)=-l_\iota-1$ and 
    \begin{align*}
      \res\alpha_\iota(0)=\frac{\ilc_0\zeta(\alpha_\iota)}{(-1)^{l_\iota+1}l_\iota!}.
    \end{align*}
  \end{enumerate}

\end{proof}

\begin{prop}[Laurent expansion of $\zeta(\fp_0\alpha)$]\label{Laurent-finite-part}
  Let $\alpha=\alpha_0+\sum_{\iota\in I}\alpha_\iota$ be a gauged poly-$\log$-homogeneous distribution with $I_0=\emptyset$. Then,
  \begin{align*}
    \zeta(\alpha)(z)=\sum_{n\in\nn_0}\frac{\zeta(\d^n\alpha)(0)}{n!}z^n
  \end{align*}
  holds in a sufficiently small neighborhood of zero.

  Let $\beta=\beta_0+\sum_{\iota\in I'}\beta_\iota$ be a gauged poly-$\log$-homogeneous distribution without structural singularities at zero, i.e. $\fa\iota\in I'_0:\ \res\beta_\iota=0$. Then, there exists a gauge $\hat\beta$ such that
  \begin{align*}
    \zeta\l(\hat\beta\r)(z)=\sum_{n\in\nn_0}\frac{\zeta(\d^n\fp_0\beta)(0)}{n!}z^n
  \end{align*}
  holds in a sufficiently small neighborhood of zero.
\end{prop}
\begin{proof}
  The first assertion is a direct consequence of the facts that the $n$\textsuperscript{th} Laurent coefficient of a holomorphic function $f$ is given by $\frac{\d^nf(0)}{n!}$ and
  \begin{align*}
    \d^n\zeta(\alpha)=\d^n\int_{\rn_{\ge1}\times M} \alpha\ d\vol_{\rn_{\ge1}\times M}=\int_{\rn_{\ge1}\times M} \d^n\alpha\ d\vol_{\rn_{\ge1}\times M}=\zeta(\d^n\alpha).
  \end{align*}
  Now 
  \begin{align*}
    \zeta\l(\hat\beta\r)(z)=\sum_{n\in\nn_0}\frac{\zeta(\d^n\fp_0\beta)(0)}{n!}z^n
  \end{align*}
  follows from the fact that we may choose an $\Mp$-gauge for $\beta_\iota$ with $\iota\in I'_0$ which yields $\zeta\l(\hat\beta\r)=\zeta(\fp_0\beta)$.

\end{proof}

$\Mp$-gauging will, furthermore, yield the following theorem which can be very handy with respect to actual computations. In particular, the fact that we can remove the influence of higher order derivatives of $\alpha_\iota$ with critical degree of homogeneity will imply that the generalized Kontsevich-Vishik density (which we will define in section \ref{sec:structural singularities}) is globally defined, i.e. for $\Mp$-gauged families with polyhomogeneous amplitudes the residue trace density and the generalized Kontsevich-Vishik density both exist globally (provided the kernel patches together).

\begin{theorem}
  Let $\alpha=\alpha_0+\sum_{\iota\in I}\alpha_\iota$ be a gauged poly-$\log$-homogeneous distribution. Then, there exists a gauge $\hat\alpha$ such that
  \begin{align*}
    \zeta\l(\hat\alpha\r)(z)=\sum_{\iota\in I_0}\frac{(-1)^{l_\iota+1}l_\iota!\res\alpha_\iota(0)}{z^{l_\iota+1}}+\sum_{n\in\nn_0}\frac{\zeta(\d^n\fp_0\alpha)(0)}{n!}z^n
  \end{align*}
  holds in a sufficiently small neighborhood of zero.
\end{theorem}
\begin{proof}
  This follows directly from Proposition \ref{Laurent-finite-part} using an $\Mp$-gauge for $\alpha_\iota$ with $\iota\in I_0$.

\end{proof}

\begin{remark*}
  In general, there will be correction terms arising from the Laurent expansion of $\res\alpha_\iota$. Incorporating these yields
  \begin{align*}
    \zeta(\alpha)(z)=&\sum_{\iota\in I_0}\l(\frac{(-1)^{l_\iota+1}l_\iota!\res\alpha_\iota(0)}{z^{l_\iota+1}}+\sum_{n=1}^{l_\iota}\frac{(-1)^{l_\iota+1}l_\iota!\d^n\res\alpha_\iota(0)}{n!}z^{n-l_\iota-1}\r)\\
    &+\sum_{n\in\nn_0}\l(\frac{\zeta(\d^n\fp_0\alpha)(0)}{n!}+\sum_{\iota\in I_0}\frac{(-1)^{l_\iota+1}l_\iota!\d^{n+l_\iota+1}\res\alpha_\iota(0)}{(n+l_\iota+1)!}\r)z^n.
  \end{align*}

\end{remark*}

\begin{corollary}\label{corollary-difference-zeta}
  Let $\alpha=\alpha_0+\sum_{\iota\in I}\alpha_\iota$ and $\beta=\beta_0+\sum_{\iota\in I}\beta_\iota$ be two gauged poly-$\log$-homogeneous distributions with $\alpha(0)=\beta(0)$ and such that the degrees of homogeneity and logarithmic orders of $\alpha_\iota$ and $\beta_\iota$ coincide. Then,
  \begin{align*}
    \zeta(\alpha)(z)-\zeta(\beta)(z)=&\sum_{\iota\in I_0}\sum_{n=1}^{l_\iota}\frac{(-1)^{l_\iota+1}l_\iota!\d^n\res\l(\alpha_\iota-\beta_\iota\r)(0)}{n!}z^{n-l_\iota-1}\\
    &+\sum_{n\in\nn_0}\frac{\zeta(\d^n\fp_0\l(\alpha-\beta\r))(0)}{n!}z^n\\
    &+\sum_{n\in\nn_0}\sum_{\iota\in I_0}\frac{(-1)^{l_\iota+1}l_\iota!\d^{n+l_\iota+1}\res\l(\alpha_\iota-\beta_\iota\r)(0)}{(n+l_\iota+1)!}z^n
  \end{align*}
  holds in a sufficiently small neighborhood of zero.
\end{corollary}
In section \ref{lagrangian}, we will see that Corollary \ref{corollary-difference-zeta} applied to pseudo-differential operators implies many well-known formulae, e.g. \cite{kontsevich-vishik}*{equation (2.21)}, \cite{paycha}*{equation (9)}, and \cite{paycha-scott}*{equation (2.20)}.

\begin{example*}
  Let $\alpha=\alpha_0+\sum_{\iota\in I}\alpha_\iota$ and $\beta=\beta_0+\sum_{\iota\in I}\beta_\iota$ be two gauged poly-homogeneous distributions with $\alpha(0)=\beta(0)$ and such that the degrees of homogeneity of $\alpha_\iota$ and $\beta_\iota$ coincide. Then, $\#I_0\le1$ and (because) all $l_\iota$ are zero. Hence,
  \begin{align*}
    \zeta(\alpha)(z)=&\sum_{\iota\in I_0}\frac{-\res\alpha_\iota(0)}{z}+\sum_{n\in\nn_0}\l(\frac{\zeta(\d^n\fp_0\alpha)(0)}{n!}-\sum_{\iota\in I_0}\frac{\d^{n+1}\res\alpha_\iota(0)}{(n+1)!}\r)z^n
  \end{align*}
  and
  \begin{align*}
    \zeta(\alpha)(z)-\zeta(\beta)(z)=&\sum_{n\in\nn_0}\l(\frac{\zeta(\d^n\fp_0\l(\alpha-\beta\r))(0)}{n!}-\sum_{\iota\in I_0}\frac{\d^{n+1}\res\l(\alpha_\iota-\beta_\iota\r)(0)}{(n+1)!}\r)z^n
  \end{align*}
  holds in a sufficiently small neighborhood of zero. This shows that the residue trace $-\sum_{\iota\in I_0}\res\alpha_\iota(0)$ is well-defined and independent of the gauge for poly-homogeneous distributions. Higher orders of the Laurent expansion depend on the gauge. 

  Furthermore, $\zeta(\alpha)-\zeta(\beta)$ is holomorphic in a neighborhood of zero and
  \begin{align*}
    \l(\zeta(\alpha)-\zeta(\beta)\r)(0)=&\zeta(\fp_0\l(\alpha-\beta\r))(0)-\sum_{\iota\in I_0}\d\res\l(\alpha_\iota-\beta_\iota\r)(0)\\
    =&\ubr{\zeta(\fp_0\alpha)(0)-\zeta(\fp_0\beta)(0)}_{=0}-\sum_{\iota\in I_0}\d\res\l(\alpha_\iota-\beta_\iota\r)(0)\\
    =&-\sum_{\iota\in I_0}\d\res\l(\alpha_\iota-\beta_\iota\r)(0).
  \end{align*}
  Defining $\gamma_\iota(z):=\frac{\alpha_\iota(z)-\beta_\iota(z)}{z}$ and $\gamma(z):=\frac{\alpha(z)-\beta(z)}{z}$ we, thus, obtain
  \begin{align*}
    \l(\zeta(\alpha)-\zeta(\beta)\r)(0)=&-\sum_{\iota\in I_0}\d\res\l(\alpha_\iota-\beta_\iota\r)(0)
    =-\sum_{\iota\in I_0}\res\gamma_\iota(0)
    =\res_0\zeta(\gamma).
  \end{align*}
  Since $\res\gamma_\iota(0)\ne0$ implies that it is independent of gauge, we obtain that $\res_0\zeta(\gamma)$ is independent of gauge which directly yields
  \begin{align*}
    \l(\zeta(\alpha)-\zeta(\beta)\r)(0)=\res_0\zeta(\gamma)=\res_0\zeta\l(\d(\alpha-\beta)\r).
  \end{align*}
  In other words, $\l(\zeta(\alpha)-\zeta(\beta)\r)(0)$ is a trace residue.

\end{example*}

\begin{theorem}[Laurent expansion of $\zeta(\alpha)$] \label{Laurent-distrib}
  Let $\alpha=\alpha_0+\sum_{\iota\in I}\alpha_\iota$ be a gauged poly-$\log$-homogeneous distribution. Then,
  \begin{align*}
    \zeta(\alpha)(z)=&\sum_{\iota\in I_0}\sum_{n=0}^{l_\iota}\frac{(-1)^{l_\iota+1}l_\iota!\int_M\d^n\tilde\alpha_\iota(0)d\vol_M}{n!\ z^{l_\iota+1-n}}
    +\sum_{n\in\nn_0}\frac{\int_{\rn_{\ge1}\times M}\d^n\alpha_0(0)d\vol_{\rn_{\ge1}\times M}}{n!}z^n\\
    &+\sum_{n\in\nn_0}\sum_{\iota\in I\setminus I_0}\sum_{j=0}^n\frac{(-1)^{l_\iota+j+1}(l_\iota+j)!\int_M\d^{n-j}\tilde\alpha_\iota(0)d\vol_M}{n!(\dim M+d_\iota+1)^{l_\iota+j+1}}z^n\\
    &+\sum_{n\in\nn_0}\sum_{\iota\in I_0}\frac{(-1)^{l_\iota+1}l_\iota!\int_M\d^{n+l_\iota+1}\tilde\alpha_\iota(0)d\vol_M}{(n+l_\iota+1)!}z^n
  \end{align*}
  holds in a sufficiently small neighborhood of zero.

  In particular, if $\alpha$ is poly-homogeneous, we obtain
  \begin{align*}
    \zeta(\alpha)(z)=&\sum_{\iota\in I_0}\frac{-\int_M\alpha_\iota(0)d\vol_M}{z}
    +\sum_{n\in\nn_0}\frac{\int_{\rn_{\ge1}\times M}\d^n\alpha_0(0)d\vol_{\rn_{\ge1}\times M}}{n!}z^n\\
    &+\sum_{n\in\nn_0}\sum_{\iota\in I\setminus I_0}\sum_{j=0}^n\frac{(-1)^{j+1}j!\int_M\d^{n-j}\alpha_\iota(0)d\vol_M}{n!(\dim M+d_\iota+1)^{j+1}}z^n\\
    &+\sum_{n\in\nn_0}\sum_{\iota\in I_0}\frac{-\int_M\d^{n+1}\alpha_\iota(0)d\vol_M}{(n+1)!}z^n
  \end{align*}
  in a sufficiently small neighborhood of zero.
\end{theorem}
\begin{proof}
  Note that having a gauged $\log$-homogeneous distribution 
  \begin{align*}
    \beta(z)(r,\xi)=r^{d+z}(\ln r)^l\tilde\beta(z)(\xi)
  \end{align*}
  the residue $\res\beta=\int_M\tilde\beta\ d\vol_M$ does not depend on the logarithmic order. Hence, we may assume without loss of generality that $l=0$ and we had a gauged homogeneous distribution in the first place, i.e. replace $\beta$ by
  \begin{align*}
    \hat\beta(z)(r,\xi)=r^{d+z}\tilde\beta(z)(\xi)
  \end{align*}
  Then, we observe
  \begin{align*}
    \d^n\beta(z)(r,\xi)=&\sum_{j=0}^n{n\choose j}r^{d+z}(\ln r)^{l+j}\d^{n-j}\tilde\beta(z)(\xi)
  \end{align*}
  and
  \begin{align*}
    \d^n\tilde\beta(z)(\xi)=&\d^n\l(x\mapsto r^{-d-x}\hat\beta(x)(\xi)\r)(z)=\sum_{j=0}^n{n\choose j}r^{-d-z}(-\ln r)^j\d^{n-j}\hat\beta(z)(r,\xi)
  \end{align*}
  for every $n\in\nn_0$, $r\in\rn_{\ge1}$, and $\xi\in M$. In particular, for $r=1$, we deduce
  \begin{align*}
    \d^n\tilde\beta(z)=&\d^{n}\hat\beta(z)|_M,
  \end{align*}
  i.e.
  \begin{align*}
    \d^n\res\beta=\d^n\int_M\tilde\beta\ d\vol_M=\int_M\d^n\tilde\beta\ d\vol_M=\int_M\d^n\hat\beta\ d\vol_M.
  \end{align*}
  Especially, for $\beta$ homogeneous, we have $\hat\beta=\beta$ and, therefore,
  \begin{align*}
    \d^n\res\beta=\int_M\d^n\tilde\beta\ d\vol_M=\int_M\d^n\hat\beta\ d\vol_M=\int_M\d^n\beta\ d\vol_M.
  \end{align*}
  Hence,
  \begin{align*}
    \begin{aligned}
      &\zeta(\d^n\fp_0\alpha)(z)\\
      =&\int_{\rn_{\ge1}\times M}\d^n\alpha_0(z)d\vol_{\rn_{\ge1}\times M}
      +\sum_{\iota\in I\setminus I_0}\sum_{j=0}^n\frac{(-1)^{l_\iota+j+1}(l_\iota+j)!\int_M\d^{n-j}\tilde\alpha_\iota(z)d\vol_M}{(\dim M+d_\iota+z+1)^{l_\iota+j+1}}.
    \end{aligned}
  \end{align*}
  This directly yields
  \begin{align*}
    \zeta(\alpha)(z)=&\sum_{\iota\in I_0}\l(\frac{(-1)^{l_\iota+1}l_\iota!\int_M\tilde\alpha_\iota(0)d\vol_M}{z^{l_\iota+1}}+\sum_{n=1}^{l_\iota}\frac{(-1)^{l_\iota+1}l_\iota!\int_M\d^n\tilde\alpha_\iota(0)d\vol_M}{n!\ z^{l_\iota+1-n}}\r)\\
    &+\sum_{n\in\nn_0}\l(\frac{\zeta(\d^n\fp_0\alpha)(0)}{n!}+\sum_{\iota\in I_0}\frac{(-1)^{l_\iota+1}l_\iota!\int_M\d^{n+l_\iota+1}\tilde\alpha_\iota(0)d\vol_M}{(n+l_\iota+1)!}\r)z^n\\
    =&\sum_{\iota\in I_0}\sum_{n=0}^{l_\iota}\frac{(-1)^{l_\iota+1}l_\iota!\int_M\d^n\tilde\alpha_\iota(0)d\vol_M}{n!\ z^{l_\iota+1-n}}
    +\sum_{n\in\nn_0}\frac{\int_{\rn_{\ge1}\times M}\d^n\alpha_0(0)d\vol_{\rn_{\ge1}\times M}}{n!}z^n\\
    &+\sum_{n\in\nn_0}\sum_{\iota\in I\setminus I_0}\sum_{j=0}^n\frac{(-1)^{l_\iota+j+1}(l_\iota+j)!\int_M\d^{n-j}\tilde\alpha_\iota(0)d\vol_M}{n!(\dim M+d_\iota+1)^{l_\iota+j+1}}z^n\\
    &+\sum_{n\in\nn_0}\sum_{\iota\in I_0}\frac{(-1)^{l_\iota+1}l_\iota!\int_M\d^{n+l_\iota+1}\tilde\alpha_\iota(0)d\vol_M}{(n+l_\iota+1)!}z^n\\
    =&\sum_{\iota\in I_0}\sum_{n=0}^{l_\iota}\frac{(-1)^{l_\iota+1}l_\iota!\int_M\d^n\tilde\alpha_\iota(0)d\vol_M}{n!\ z^{l_\iota+1-n}}
    +\sum_{n\in\nn_0}\frac{\int_{\rn_{\ge1}\times M}\d^n\alpha_0(0)d\vol_{\rn_{\ge1}\times M}}{n!}z^n\\
    &+\sum_{n\in\nn_0}\sum_{\iota\in I\setminus I_0}\sum_{j=0}^n\frac{(-1)^{l_\iota+j+1}(l_\iota+j)!\int_M\d^{n-j}\tilde\alpha_\iota(0)d\vol_M}{n!(\dim M+d_\iota+1)^{l_\iota+j+1}}z^n\\
    &+\sum_{n\in\nn_0}\sum_{\iota\in I_0}\frac{(-1)^{l_\iota+1}l_\iota!\int_M\d^{n+l_\iota+1}\tilde\alpha_\iota(0)d\vol_M}{(n+l_\iota+1)!}z^n.
  \end{align*}
\end{proof}

\begin{remark*}
  Closely related to the notion of $\zeta$-regularized traces are $\zeta$-determinants. Let $\alpha=\alpha_0+\sum_{\iota\in I}\alpha_\iota$ be a gauged poly-$\log$-homogeneous distribution such that $\zeta(\alpha)$ is holomorphic in a neighborhood of zero. Then, we define the generalized $\zeta$-determinant 
  \begin{align*}
    \det_\zeta(\alpha):=\exp\l(\zeta(\alpha)'(0)\r).
  \end{align*}
  This generalized $\zeta$-determinant reduces to the $\zeta$-determinants as studied by Kontsevich and Vishik in \cite{kontsevich-vishik,kontsevich-vishik-geometry}. In other words, we do not expect it to be multiplicative if $\alpha$ corresponds to a general Fourier Integral Operator. Though an interesting question, we will not study classes of families of Fourier Integral Operators satisfying the multiplicative property, here.

\end{remark*}

\subsection{Remark on more general gauged poly-$\log$-homogeneous distributions}\label{more general distributions}
In many applications, considering gauges $d_\iota(z)$ which are not of the form $d_\iota(0)+z$ is important. However, it can be shown that such gauges and the corresponding $\zeta$-functions are germ equivalent to the case $d_\iota(0)+z$ provided that $d_\iota(z)$ is not germ equivalent to $-\dim M-1$ (in that case, the $\zeta$-function won't exist at all). Thus, we assume without loss of generality that we are working with gauged poly-$\log$-homogeneous distributions of the form above. For more detail, please refer to \cite{hartung-phd}*{chapter 3} or \cite{hartung-scott}*{chapter 2}.

\section{Application to gauged Lagrangian distributions}\label{lagrangian}
The objective for this section is to apply the Laurent expansion of gauged poly-$\log$-homogeneous distributions from section \ref{gauged-poly-log-homogeneous-distributions} to gauged Lagrangian distributions and Fourier Integral Operator traces; thus, extending Guillemin's work \cite{guillemin-lagrangian,guillemin-residue-traces} on the residue trace of Fourier Integral Operators. 

If we consider a dual pair $\langle u(z),f\rangle$ where $u:\ \cn\to I(X\times X,\Lambda)$ is a gauged Lagrangian distribution and $f\in I(X\times X,\hat\Lambda)$ (cf. \cite{guillemin-lagrangian} and \cite{hoermander-books}*{chapter 25}) such that $\Lambda$ and $\hat\Lambda$ intersect cleanly at $\gamma$, then \cite{hoermander-books}*{Theorem 21.2.10} yields homogeneous symplectic coordinates $(x,\xi)$ near $\gamma$ such that $\gamma=(1,0,\ldots,0)$, $\Lambda=\{(0,\xi)\}$, and $\hat\Lambda=\{(0,\hat x,\check\xi,0)\}$ where $x=(\check x,\hat x)$, $\check x=(x_1,\ldots,x_k)$, $\hat x=(x_{k+1},\ldots,x_{\dim X})$, $\xi=(\check \xi,\hat \xi)$, $\check \xi=(\xi_1,\ldots,\xi_k)$, $\hat \xi=(\xi_{k+1},\ldots,\xi_{\dim X})$, and $k=\dim \Lambda\cap\hat\Lambda$.

Since $f$ can be written as $f=P^t\delta_0$ for some pseudo-differential operator $P$, we obtain $\langle u(z),f\rangle=\langle Pu(z),\delta_0\rangle$ and, using the coordinates above, $Pu(z)$ is an oscillatory integral of the form
\begin{align}\label{eq:FIO-trace-diagonal}
  \int_{\rn[k]}e^{i\sum_{j=1}^k x_j\xi_j}a(z)\l(x_{k+1},\ldots,x_{\dim X},\xi_1,\ldots,\xi_k\r)d(\xi_1,\ldots,\xi_k),
\end{align}
i.e.
\begin{align*}
  \langle u(z),f\rangle=\int_{\rn[k]}a(z)\l(0,\xi\r)d\xi.
\end{align*}
As pointed out by Guillemin in the proof of \cite{guillemin-lagrangian}*{Theorem 2.1}, this is a gauged poly-$\log$-homogeneous distribution, i.e. the formalism developed above is applicable.

In order to treat
\begin{align*}
  \langle u(z),\delta_{\diag}\rangle=\int_X\int_{\rn[N]}e^{i\theta(x,\xi)}a(z)(x,\xi)\ d\xi\ d\vol_X=\int_{\rn[k]}\alpha(z)(\xi)d\xi,
\end{align*}
we will split off the integral
\begin{align*}
  \tilde\tau_0(z):=\int_{B_{\rn[k]}(0,1)}\alpha(z)(\xi)\ d\xi
\end{align*}
which defines a holomorphic function and we are left with
\begin{align*}
  \int_{\rn_{\ge1}\times\d B_{\rn[k]}}\alpha(z)(\xi)d\vol_{\rn_{\ge1}\times\d B_{\rn[k]}}(\xi)
\end{align*}
which is a distribution as considered in section \ref{gauged-poly-log-homogeneous-distributions}. In other words, if $A$ is a gauged Fourier Integral Operator with phase function $\theta$ and amplitude $a$ on $X$, then
\begin{align*}
  \zeta(A)(z)=&\ubr{\int_X\int_{B_{\rn[N]}(0,1)}e^{i\theta(x,x,\xi)}a(z)(x,x,\xi)\ d\xi\ d\vol_X(x)}_{=:\tau_0(A)(z)}\\
  &+\int_X\int_{\rn_{\ge1}\times\d B_{\rn[N]}}e^{i\theta(x,x,\xi)}a(z)(x,x,\xi)\ d\vol_{\rn_{\ge1}\times\d B_{\rn[N]}}(\xi)\ d\vol_X(x)
\end{align*}
exists and inherits all properties described in section \ref{gauged-poly-log-homogeneous-distributions}. 
\begin{theorem}\label{Laurent-zeta-FIO}
  If $a$ is poly-$\log$-homogeneous and $A_\iota$ the gauged Fourier Integral Operator with phase $\theta$ and amplitude $a_\iota$ then
  \begin{align*}
    \res A_\iota(z)=\int_{\d B_{\rn[N]}}\int_Xe^{i\theta(x,x,\xi)}\tilde a_\iota(z)(x,x,\xi)\ d\vol_X(x)\ d\vol_{\d B_{\rn[N]}}(\xi)
  \end{align*}
  and 
  \begin{align*}
    \begin{aligned}
      &\zeta(A)(z)\\
      =&\sum_{n\in\nn_0}\frac{\int_X\int_{B_{\rn[N]}(0,1)}e^{i\theta(x,x,\xi)}\d^na(0)(x,x,\xi)\ d\xi\ d\vol_X(x)}{n!}z^n\\
      &+\sum_{\iota\in I_0}\sum_{n=0}^{l_\iota}\frac{(-1)^{l_\iota+1}l_\iota!\int_{\Delta(X)\times\d B_{\rn[N]}}e^{i\theta}\d^n\tilde a_\iota(0)\ d\vol_{\Delta(X)\times\d B_{\rn[N]}}}{n!}z^{n-l_\iota-1}\\
      &+\sum_{n\in\nn_0}\frac{\int_{\rn_{\ge1}\times\d B_{\rn[N]}}\int_Xe^{i\theta(x,x,\xi)}\d^na_0(0)(x,x,\xi)\ d\vol_X(x)\ d\vol_{\rn_{\ge1}\times\d B_{\rn[N]}}(\xi)}{n!}z^n\\
      &+\sum_{n\in\nn_0}\sum_{\iota\in I\setminus I_0}\sum_{j=0}^n\frac{(-1)^{l_\iota+j+1}(l_\iota+j)!\int_{\Delta(X)\times\d B_{\rn[N]}}e^{i\theta}\d^{n-j}\tilde a_\iota(0)\ d\vol_{\Delta(X)\times\d B_{\rn[N]}}}{n!(N+d_\iota)^{l_\iota+j+1}}z^n\\
      &+\sum_{n\in\nn_0}\sum_{\iota\in I_0}\frac{(-1)^{l_\iota+1}l_\iota!\int_{\Delta(X)\times\d B_{\rn[N]}}e^{i\theta}\d^{n+l_\iota+1}\tilde a_\iota(0)\ d\vol_{\Delta(X)\times\d B_{\rn[N]}}}{(n+l_\iota+1)!}z^n
    \end{aligned}
  \end{align*}
  in a neighborhood of zero where $\Delta(X):=\{(x,y)\in X^2;\ x=y\}$. 
\end{theorem}

\begin{remark*}
  Appendix \ref{examples} contains examples applying Theorem \ref{Laurent-zeta-FIO} to the heat trace on a flat torus, as well as, $\zeta$-functions of fractional Laplacians and shifted fractional Laplacians on $\rn/_{2\pi\Zn}$.
\end{remark*}

For a poly-homogeneous $a$ this reduces to
\begin{align*}
  \zeta(A)(z)
  =&\sum_{n\in\nn_0}\frac{\int_X\int_{B_{\rn[N]}(0,1)}e^{i\theta(x,x,\xi)}\d^na(0)(x,x,\xi)\ d\xi\ d\vol_X(x)}{n!}z^n\\
  &-\sum_{\iota\in I_0}\int_{\Delta(X)\times\d B_{\rn[N]}}e^{i\theta}a_\iota(0)\ d\vol_{\Delta(X)\times\d B_{\rn[N]}}z^{-1}\\
  &+\sum_{n\in\nn_0}\frac{\int_{\Delta(X)\times(\rn_{\ge1}\times\d B_{\rn[N]})}e^{i\theta}\d^na_0(0)\ d\vol_{\Delta(X)\times(\rn_{\ge1}\times\d B_{\rn[N]})}}{n!}z^n\\
  &+\sum_{n\in\nn_0}\sum_{\iota\in I\setminus I_0}\sum_{j=0}^n\frac{(-1)^{j+1}j!\int_{\Delta(X)\times\d B_{\rn[N]}}e^{i\theta}\d^{n-j}a_\iota(0)\ d\vol_{\Delta(X)\times\d B_{\rn[N]}}}{n!(N+d_\iota)^{j+1}}z^n\\
  &+\sum_{n\in\nn_0}\sum_{\iota\in I_0}\frac{-\int_{\Delta(X)\times\d B_{\rn[N]}}e^{i\theta}\d^{n+1}a_\iota(0)\ d\vol_{\Delta(X)\times\d B_{\rn[N]}}}{(n+1)!}z^n,
\end{align*}
i.e.
\begin{align*}
  \zeta(A)(z)
  =&-\sum_{\iota\in I_0}\res A_\iota(0)z^{-1}-\sum_{n\in\nn_0}\sum_{\iota\in I_0}\frac{\res\d^{n+1}A_\iota(0)}{(n+1)!}z^n\\
  &+\sum_{n\in\nn_0}\frac{\int_X\int_{B_{\rn[N]}(0,1)}e^{i\theta(x,x,\xi)}\d^na(0)(x,x,\xi)\ d\xi\ d\vol_X(x)}{n!}z^n\\
  &+\sum_{n\in\nn_0}\frac{\int_{\Delta(X)\times(\rn_{\ge1}\times\d B_{\rn[N]})}e^{i\theta}\d^na_0(0)\ d\vol_{\Delta(X)\times(\rn_{\ge1}\times\d B_{\rn[N]})}}{n!}z^n\\
  &+\sum_{n\in\nn_0}\sum_{\iota\in I\setminus I_0}\sum_{j=0}^n\frac{(-1)^{j+1}j!\res\d^{n-j}A_\iota(0)}{n!(N+d_\iota)^{j+1}}z^n
\end{align*}
where $\d^n A_\iota$ is the gauged Fourier Integral Operator with phase $\theta$ and amplitude $\d^na_\iota$.

From this last formula and the knowledge that $\res A_\iota(0)$ is independent of the gauge we obtain the following well-known result (cf. \cite{guillemin-lagrangian}).

\begin{theorem}
  Let $A$ and $B$ be poly-homogeneous Fourier Integral Operators. Let $G_1$ and $G_2$ be gauged Fourier Integral Operators with $G_1(0)=AB$ and $G_2(0)=BA$. Then, 
  \begin{align*}
    \res_0\zeta(G_1)=\res_0\zeta(G_2),
  \end{align*}
  i.e. the residue of the $\zeta$-function is tracial and $A\mapsto\res_0\zeta\l(\hat A\r)$ is a well-defined trace where $\hat A$ is any choice of gauge for $A$.
\end{theorem}
\begin{proof}
  This is a direct consequence of the following two facts.
  \begin{enumerate}
  \item[(i)] $\res_0\zeta(G_1)=-\sum_{\iota\in I_0}\res (G_1)_\iota(0)$ is independent of the gauge.
  \item[(ii)] $\zeta\l(\hat AB\r)=\zeta\l(B\hat A\r)$ for any gauge $\hat A$ of $A$ because it is true for $\Re(z)$ sufficiently small.
  \end{enumerate}

\end{proof}
Similarly, for $I_0(AB)=\emptyset$, we obtain that $\zeta(AB)(0)=\zeta(BA)(0)$ where we used that $\zeta(\fp_0\alpha)(0)$ is independent of gauge. In other words, we may generalize the Kontsevich-Vishik trace to $\zeta(\fp_0A)(0)$ where $\fp_0A$ is the gauged Fourier Integral Operator with phase $\theta$ and amplitude $a-\sum_{\iota\in I_0}a_\iota$. In particular, we may also consider the regularized generalized determinant
\begin{align*}
  \det_{\fp}(A):=\exp{\zeta(\fp_0A)'(0)}
\end{align*}
where
\begin{align*}
  &\zeta(\fp_0A)(z)\\
  =&\sum_{n\in\nn_0}\frac{\int_X\int_{B_{\rn[N]}(0,1)}e^{i\theta(x,x,\xi)}\d^na(0)(x,x,\xi)\ d\xi\ d\vol_X(x)}{n!}z^n\\
  &+\sum_{n\in\nn_0}\frac{\int_{\Delta(X)\times\l(\rn_{\ge1}\times\d B_{\rn[N]}\r)}e^{i\theta}\d^na_0(0)\ d\vol_{\Delta(X)\times\l(\rn_{\ge1}\times\d B_{\rn[N]}\r)}}{n!}z^n\\
  &+\sum_{n\in\nn_0}\sum_{\iota\in I\setminus I_0}\sum_{j=0}^n\frac{(-1)^{l_\iota+j+1}(l_\iota+j)!\int_{\Delta(X)\times\d B_{\rn[N]}}e^{i\theta}\d^{n-j}\tilde a_\iota(0)\ d\vol_{\Delta(X)\times\d B_{\rn[N]}}}{n!(N+d_\iota)^{l_\iota+j+1}}z^n
\end{align*}
though we will not study this determinant, here. 

An important class of gauges are multiplicative gauges.

\begin{definition}
  Let $A$ be a Fourier Integral Operator and $G$ a gauged Fourier Integral Operator with $G(0)=1$ such that each $G(z)$ and all derivatives are composable with $A$. Then, we call $AG(\cdot)$ a multiplicative gauge of $A$.
\end{definition}

\begin{remark*}
  If we consider a canonical relation $\Gamma$ and the corresponding algebra of Fourier Integral Operators $\Ap_\Gamma$, then we may be inclined to search for multiplicative gauges in $\Ap_\Gamma$. Unfortunately, the identity will not be an element of $\Ap_\Gamma$, in general (otherwise, $\Gamma$ would need to contain the (graph of the) identity on $T^*X\setminus0$ which would imply that all pseudo-differential operators are in $\Ap_\Gamma$, as well). An appropriate candidate of an algebra to consider if looking for a multiplicative gauge should, therefore, be the unitalization $\Ap_\Gamma\oplus\cn$ of $\Ap_\Gamma$. Since derivatives should exists within the algebra and we might be interested in using a functional calculus, it may be necessary to also include an $L(L_2(X))$ closure of $\Ap_\Gamma\oplus\cn$.

We may, however, gauge with properly supported pseudo-differential operators $G(z)$ (cf. \cite{shubin}*{section 18.4}). 
\end{remark*}
Let $P$ be a gauged pseudo-differential operator. Then, we may also consider
\begin{align*}
  \langle P(z)u,f\rangle
\end{align*}
as a gauge. This is due to \cite{hoermander-books}*{Theorems 18.2.7 and 18.2.8}. In particular, if $f$ is a Lagrangian distribution, then it can be represented in the form $\int e^{i\langle x,\xi\rangle}a_f(x,\xi)d\xi$ which is nothing other than $P_f\delta_0$ where $P_f$ is the pseudo-differential operator with amplitude $a_f$. Hence, 
\begin{align*}
  \langle P(z)u,f\rangle=\langle P_f'P(z)u,\delta_0\rangle.
\end{align*}
For traces, though, a multiplicative gauge yields
\begin{align*}
  \zeta(A)(z)=\langle \gf(z)\circ k_A,\delta_{\diag}\rangle
\end{align*}
where $\gf(z)\circ k_A$ is the kernel of $G(z)A$ and $\fa\phi\in C(X):\ \delta_{\diag}(\phi)=\int_X\phi(x,x)dx$ (i.e. $\delta_{\diag}$ is the kernel of the identity).
\begin{example*}
  Suppose $u$ is an $\Mp$-gauged $\log$-homogeneous distribution. We, thus, obtain
  \begin{align*}
    u(0)(x)=&\tau_0(u(0))(x)+\int_{\rn[N]\setminus B_{\rn[N]}}e^{i\langle x,\xi\rangle}v(0)(\xi)=\tilde\tau_0(u(0))(x)+(P_u\delta_0)(x)
  \end{align*}
  where $P_u$ is a pseudo-differential operator with amplitude $p_u(x,\xi)=v(\xi)$ for $\xi\in\rn[n]\setminus B_{\rn[N]}$. Furthermore, the complex power $H^{z}$ with $H:=\sqrt{\betr\Delta}$ has the amplitude $p_z(x,\xi)=(2\pi)^{-N}\norm\xi_{\ell_2(N)}^z$ where $\betr\Delta$ is the (non-negative) Dirichlet Laplacian because $\betr\Delta^{-1}=\Fp[-1]\norm m_{\ell_2(N)}^{-2}\Fp$ where $m$ is the maximal multiplication operator with the argument on $L_2(\rn[N])$
  \begin{align*}
    D(m):&=\l\{f\in L_2(\rn[N]);\ \l(\rn[n]\ni\xi\mapsto \xi f(\xi)\in\cn[N]\r)\in L_2(\rn[N];\cn[N])\r\},\\
    m:&\ D(m)\sse L_2(\rn[N])\to L_2(\rn[N];\cn[N]);\ f\mapsto\l(\xi\mapsto \xi f(\xi)\r).
  \end{align*}
  $(-\Delta)^{-1}$ is well-known to be a compact operator. 

  Hence, let $r-1$ be its spectral radius. Then, the holomorphic functional calculus yields
  \begin{align*}
    H^z=&\l(\betr\Delta^{-1}\r)^{-\frac z2}=\frac{1}{2\pi i}\int_{r\d B_{\cn}}\lambda^{-\frac z2}\l(\lambda-(-\Delta)^{-1}\r)^{-1}d\lambda=\Fp[-1]\norm m_{\ell_2(N)}^z\Fp.
  \end{align*}
  The composition formula for pseudo-differential operators implies that $(2\pi)^{N}H^zP_u$ has the amplitude
  \begin{align*}
    \sum_{\alpha\in\nn[n]_0}\frac{1}{\alpha!}\d_2^\alpha \l((2\pi)^{N}p_z\r)(x,\xi)\ubr{(-i\d_1)^\alpha p_u(x,\xi)}_{=0\ \neht\ \alpha\ne0}=\norm\xi_{\ell_2(N)}^zv(0)(\xi)=v(z)(\xi).
  \end{align*}
  In other words, $u(z)\equiv (2\pi)^{N}H^{z}u(0)$ modulo whatever happens on $B_{\rn[N]}$.
 
\end{example*}
\begin{example*}
  Let $A$ be a poly-$\log$-homogeneous Fourier Integral Operator and $u$ a poly-$\log$-homogeneous distribution with $I_0(A)=I_0(u)=\emptyset$. Suppose $G$ and $P$ are exponential multiplicative gauges, that is,
  \begin{align*}
    G'(z)=G(z)G_0\qquad\text{and}\qquad P'(z)=P(z)P_0,
  \end{align*}
  for $A$ and $u$, respectively. Then
  \begin{align*}
    \zeta(GA)(z)=&\sum_{n\in\nn_0}\frac{\d^n\zeta(GA)(0)}{n!}z^n=\sum_{n\in\nn_0}\frac{\zeta(\d^nGA)(0)}{n!}z^n=\sum_{n\in\nn_0}\frac{\zeta(GG_0^nA)(0)}{n!}z^n
  \end{align*}
  and
  \begin{align*}
    \zeta(Pu)(z)=&\sum_{n\in\nn_0}\frac{\d^n\zeta(Pu)(0)}{n!}z^n=\sum_{n\in\nn_0}\frac{\zeta(\d^nPu)(0)}{n!}z^n=\sum_{n\in\nn_0}\frac{\zeta(PP_0^nu)(0)}{n!}z^n
  \end{align*}
  hold in sufficiently small neighborhoods of zero. 
\end{example*}

\section{Mollification of singular amplitudes}\label{sec:mollification}
In this section, we will address the fact that many applications consider amplitudes which are homogeneous on $\rn[N]\setminus\{0\}$ (our results up to now assume continuity of the amplitude in zero). In particular for pseudo-differential operators, this does not add too many problems because we can use a cut-off function near zero and extend the symbol as a distribution to $\rn[N]$ (which is uniquely possible up to certain critical degrees of homogeneity which are related to the residues). Then, we are left with a Fourier transform of a compactly supported distribution, i.e. the corresponding kernel is continuous and we can take the trace. In the general Fourier Integral Operator case, the situation is more complicated. Hence, in this section, we will show that the Laurent expansion holds for such amplitudes, as well. We will prove this result by showing that we can always find a sequence of ``nice'' families of operators such that their $\zeta$-functions converge compactly.

In appendix \ref{examples}, our calculations of $\zeta\l(s\mapsto H^s H^\alpha\r)$ with $H:=\sqrt{\betr{\Delta}}$, where $\Delta$ is the Laplacian on $\rn/_{2\pi\zn}$, are currently pushing the boundaries of Theorem \ref{Laurent-zeta-FIO} in the sense that the Laurent expansion of Fourier Integral Operators assumes integrability of all amplitudes $a(z)$ on $B_{\rn[N]}$. This is obviously not true for $a(z)(x,y,\xi)=\betr\xi^{z+\alpha}$ (at least not for all $z\in\cn$). Hence, we would have to consider the Laurent expansion in a more general version where we also allowed
\begin{align*}
  z\mapsto \int_X\int_{B_{\rn[N]}}e^{i\theta(x,x,\xi)}a(z)(x,x,\xi)\ d\xi\ d\vol_X(x)
\end{align*}
to have a non-vanishing principal part. 

However, we may use $\zeta\l(s\mapsto G^sG^\alpha\r)$ with $G:=h+H$ for $h\in(0,1)$ to justify the calculations as they are by taking the limit $h\searrow0$ in $\zeta\l(s\mapsto G^sG^\alpha\r)$. In fact, it is possible to show
\begin{align*}
  \lim_{h\searrow0}\zeta\l(s\mapsto G^sG^\alpha\r)=&\zeta\l(s\mapsto H^sH^\alpha\r)\qquad\text{compactly}.
\end{align*}

Here, we regularized the kernel $a(z)(x,y,\xi)=\betr\xi^z$ by adding an $h\in(0,1)$ yielding a perturbed amplitude $a_h(z)(x,y,\xi)=(h+\betr\xi)^z$ which has no singularities. Showing that the limit $h\searrow0$ exists, then, justifies our calculations. Using Vitali's theorem (cf. e.g. \cite{jentzsch}*{chapter 1}) we can largely generalize this approach.
\begin{theorem}[Vitali]
    Let $\Omega\sse_{\open,\connected}\cn$, $f\in C^\infty(\Omega)^{\nn}$ locally bounded, and let 
    \begin{align*}
      \{z\in\Omega;\ (f_n(z))_{n\in\nn}\text{ converges}\}
    \end{align*}
    have an accumulation point in $\Omega$. Then, $f$ is compactly convergent.
\end{theorem}
In general, we will use the following terminology.

\begin{definition}
  Let $\alpha$ be a gauged poly-$\log$-homogeneous distribution on $\rn_{>0}\times M$. We say $\zeta(\alpha)$ can be mollified if and only if there exists a sequence $\l(\alpha_n\r)_{n\in\nn}$ of holomorphic families $\alpha_n$ in $L_{1,\loc}(\rn_{>0}\times M)$ such that each $\alpha_n$ restricts to a gauged poly-$\log$-homogeneous distribution on $\rn_{>1}\times M$ and $\l(\zeta(\alpha_n)\r)_{n\in\nn}$ converges compactly. 
\end{definition}

Let $(A_n)_{n\in\nn}$ be a sequence of gauged Fourier Integral Operators with $C^\infty$-amplitudes and $A$ a gauged Fourier Integral Operator whose amplitudes may contain singularities. Furthermore, let $A_n(z)\to A(z)$ for every $z$ in the gap topology (cf. \cite{kato}*{Chapter IV}). Let $d\in\rn$ such that $\fa z\in\cn:\ (\Re(z)<d\ \then\ A(z)$ is of trace-class$)$ and $\Omega:=\cn_{\Re(\cdot)<d-1}$. Then, for every $z\in\Omega$, $(A_n(z))_{n\in\nn}$ is eventually a sequence of bounded operators and $A_n|_\Omega\to A|_\Omega$ converges pointwise in norm. Furthermore, let $\l(\lambda_k(z)\r)_{k\in\nn}$ be the sequence of eigenvalues of $A(z)$ counting multiplicities and $\l(\lambda_k(z)+h_k^n(z)\r)_{k\in\nn}$ be the sequence of eigenvalues of $A_n(z)$ counting multiplicities. Suppose that $h^n(z):=\sum_{k\in\nn}\betr{h_k^n(z)}$ exists and converges to zero for $z\in\Omega$. Then,
\begin{align*}
  \betr{\zeta(A_n)(z)-\zeta(A)(z)}=&\betr{\sum_{k\in\nn}\l(\lambda_k(z)+h_k^n(z)\r)-\sum_{k\in\nn}\lambda_k(z)}=\betr{\sum_{k\in\nn}h_k^n(z)}\le h^n(z)\to0
\end{align*}
for $z\in\Omega$ shows
\begin{align*}
  \l\{z\in\Omega;\ \l(\zeta(A_n)(z)\r)_{n\in\nn}\text{ converges}\r\}=\Omega.
\end{align*}
Let $\tilde\Omega\sse\cn$ be open and connected with $\Omega\sse\tilde\Omega$ such that all $\zeta(A_n)|_{\tilde\Omega}$ are holomorphic and $\l\{\zeta(A_n)|_{\tilde\Omega};\ n\in\nn\r\}$ is locally bounded. Then, 
\begin{align*}
  \lim_{n\to\infty}\zeta(A_n)|_{\tilde\Omega}=\zeta(A)|_{\tilde\Omega}.
\end{align*}
In particular, if $h^n$ admits an analytic continuation to $\tilde\Omega$, then $\lim_{n\to\infty}\zeta(A_n)|_{\tilde\Omega}=\zeta(A)|_{\tilde\Omega}$.
\begin{remark*}
  Note that $A_n(z)\to A(z)$ in the gap topology implies that the $h_k^n(z)$ exist and for every $k$ and $z$ we have $\lim_{n\to\infty}h_k^n(z)\to0$. However, in general, we will not have any uniform bound on them, let alone find an $h^n(z)$; cf. \cite{kato}*{Section IV.3.5}.

\end{remark*}

\begin{definition}
  Let $A$ be an operator with purely discrete spectrum. For every $\lambda\in\sigma(A)$ let $\mu_\lambda$ be the multiplicity of $\lambda$. Then, we define the spectral $\zeta$-function $\zeta_\sigma(A)$ to be the meromorphic extension of
  \begin{align*}
    \zeta_\sigma(A)(s):=&\sum_{\lambda\in\sigma(A)\setminus\{0\}}\mu_\lambda\lambda^{-s}
  \end{align*}
  and the spectral $\Theta$-function $\Theta_\sigma(A)$ 
  \begin{align*}
    \fa t\in\rn_{>0}:\ \Theta_\sigma(A)(t):=&\sum_{\lambda\in\sigma(A)}\mu_\lambda\exp\l(-t\lambda\r)
  \end{align*}
  if they exist.
\end{definition}

\begin{definition}
  Let $T\in\rn_{>0}$ and $\phi\in C(\rn_{>0})$. We define the upper Mellin transform as
  \begin{align*}
    \Mp[T](\phi)(s):=\int_{(0,T)}\phi(t)t^{s-1}dt
  \end{align*}
  and the lower Mellin transform
  \begin{align*}
    \Mp_T(\phi)(s):=\int_{\rn_{\ge T}}\phi(t)t^{s-1}dt
  \end{align*}
  (if the integrals exist). If both integrals exist and with non-empty intersection $\Omega$ of domains of holomorphy (that is, the maximal connected and open subset admitting an analytic continuation of the function), then we define the generalized Mellin transform of $\phi$ to be the meromorphic extension of
  \begin{align*}
    \Mp(\phi):=\Mp[T](\phi)|_\Omega+\Mp_T(\phi)|_\Omega.
  \end{align*}
\end{definition}

\begin{example*}
  Let $\phi(t):=t^\alpha$ for some $\alpha\in\cn$. Then
  \begin{align*}
    \Mp[T](\phi)(s)=&\int_{(0,T)}t^{s+\alpha-1}dt=\frac{T^{s+\alpha}}{s+\alpha}
  \end{align*}
  for $\Re(s)>\alpha$ extending to $\cn\setminus\{-\alpha\}$ and
  \begin{align*}
    \Mp_T(\phi)(s)=&\int_{\rn_{\ge T}}t^{s+\alpha-1}dt=-\frac{T^{s+\alpha}}{s+\alpha}
  \end{align*}
  for $\Re(s)<\alpha$ extending to $\cn\setminus\{-\alpha\}$. Hence, $\Mp(\phi)(s)=\frac{T^{s+\alpha}}{s+\alpha}-\frac{T^{s+\alpha}}{s+\alpha}=0$ exists on $\cn\setminus\{-\alpha\}$, i.e. $\Mp(\phi)=0$.

\end{example*}

\begin{example*}
  Let $\lambda\in\rn_{>0}$ and $s\in\cn$ with $\Re(s)>0$. Then
  \begin{align*}
    \int_{\rn_{>0}}e^{-\lambda t}t^{s-1}dt=\int_{\rn_{>0}}e^{-\tau}\tau^{s-1}\lambda^{-s}dt=\lambda^{-s}\Gamma(s)
  \end{align*}
  shows that $\lambda\mapsto\int_{\rn_{>0}}e^{-\lambda t}t^{s-1}dt$ extends analytically to $\cn\setminus\rn_{\le0}$.

\end{example*}

\begin{example*}
  Let $A$ be an operator with purely discrete spectrum. For every $\lambda\in\sigma(A)$ let $\mu_\lambda$ be the multiplicity of $\lambda$ and $\Re(\lambda)\ge0$. $\Mp(1)=0$, then, implies
  \begin{align*}
    \Mp\l(\Theta_\sigma(A)\r)(s)=&\sum_{\lambda\in\sigma(A)}\mu_\lambda\Mp\l(t\mapsto\exp(-t\lambda)\r)(s)\\
    =&\sum_{\lambda\in\sigma(A)\setminus\{0\}}\mu_\lambda\Mp\l(t\mapsto\exp(-t\lambda)\r)(s)\\
    =&\sum_{\lambda\in\sigma(A)\setminus\{0\}}\mu_\lambda\lambda^{-s}\Gamma(s)\\
    =&\zeta_\sigma(A)(s)\Gamma(s).
  \end{align*}

\end{example*}

\begin{lemma}\label{compact-convergence-mellin}
  $\lim_{h\searrow0}\Mp\l(t\mapsto\exp(-th)\r)=\Mp(1)=0$ compactly.
\end{lemma}

\begin{proof}
  For $\Re(s)>1$, we obtain
  \begin{align*}
    \frac{1}{\Gamma(s)}\Mp\l(t\mapsto\exp(-th)\r)(s)=&\frac{1}{\Gamma(s)}\int_{\rn_{>0}}e^{-th}t^{s-1}dt\\
    =&h^{-s}\\
    =&\sum_{k\in\nn_0}(k+h)^{-s}-\sum_{k\in\nn_0}(k+1+h)^{-s}\\
    =&\zeta_H(s;h)-\zeta_H(s;1+h).
  \end{align*}
  Hence,
  \begin{align*}
    \Mp\l(t\mapsto\exp(-th)\r)(s)=&\Gamma(s)\zeta_H(s;h)-\Gamma(s)\zeta_H(s;1+h)
  \end{align*}
  holds on $\cn\setminus\zn_{\le1}$. Furthermore, $\Gamma(s)\zeta_H(s;h)-\Gamma(s)\zeta_H(s;1+h)$ is locally bounded on $\cn\setminus\zn_{\le1}$ for $h\searrow0$ which implies
  \begin{align*}
    \lim_{h\searrow0}\Mp\l(t\mapsto\exp(-th)\r)(s)=&\lim_{h\searrow0}\l(\Gamma(s)\zeta_H(s;h)-\Gamma(s)\zeta_H(s;1+h)\r)=0,
  \end{align*}
  i.e. $\lim_{h\searrow0}\Mp\l(t\mapsto\exp(-th)\r)$ exists and vanishes on $\cn\setminus\zn_{\le1}$. Vitali's theorem, thence, proves the assertion.

\end{proof}

\begin{corollary}\label{corollary-spectral-mollification}
  Let $A$ and $A_h$ be operators with spectral $\zeta$-functions. Let $\zeta_\sigma(A)$ be the meromorphic extension of $\sum_{k\in N}\lambda_k^{-s}$ for some $N\sse\nn$ and $\zeta_\sigma(A_h)$ the meromorphic extension of $\sum_{j=1}^n\tilde h^{-s}_j+\sum_{k\in N}(\lambda_k+h_k)^{-s}$ where all $\tilde h_j\in\rn_{>0}$. Suppose $A_h$ converges to $A$ in the gap topology, the meromorphic extension $f_h$ of $\sum_{k\in N}(\lambda_k+h_k)^{-s}$ is locally bounded, and converges to $\zeta_\sigma(A)$ pointwise. 

  Then, $\zeta_\sigma(A_h)$ converges to $\zeta_\sigma(A)$ compactly.
\end{corollary}

\begin{proof}
  The assertion is a direct consequence of $\sum_{j=1}^n\tilde h^{-s}_j\to0$ compactly (Lemma \ref{compact-convergence-mellin}) and $f_h\to\zeta_\sigma(A)$ compactly (Vitali's theorem).

\end{proof}

The ideas leading Corollary \ref{corollary-spectral-mollification} and $\zeta\l(s\mapsto(h+H)^{s+\alpha}\r)\to\zeta\l(s\mapsto H^{s+\alpha}\r)$ can, then, be used to prove the following theorem.

\begin{theorem}\label{mollification}
  Let $\alpha=\alpha_0+\sum_{\iota\in I}\alpha_\iota$ be a gauged poly-$\log$-homogeneous distribution on $\rn_{>0}\times M$ with $I\sse\nn$, $\alpha_0$ regular on $(0,1)\times M$, 
  \begin{align*}
    \alpha_\iota(z)(r,\xi)=&r^{d_\iota+z}(\ln r)^{l_\iota}\tilde\alpha_\iota(z)(\xi),
  \end{align*}
  where $\l(\Re(d_\iota)\r)_{\iota\in I}$ is bounded from above, each $\l(\frac{1}{\dim M+d_\iota+z+1}\r)_{\iota\in I}\in\ell_2(I)$, and each of the $\sum_{\iota\in I}\tilde\alpha_\iota(z)$ converges unconditionally in $L_1(M)$. Then, $\zeta(\alpha)$ can be mollified.

  In particular, 
  \begin{align*}
    \zeta(\alpha)(z)=&\int_{\rn_{>0}\times M}\alpha_0(z)d\vol_{\rn_{>0}\times M}+\sum_{\iota\in I}\int_{\rn_{\ge1}\times M}\alpha_\iota(z)d\vol_{\rn_{\ge1}\times M}\\
    &+\sum_{\iota\in I}\int_{(0,1)}r^{\dim M+d_\iota+z}(\ln r)^{l_\iota}dr\res\alpha_\iota(z)
  \end{align*}
  is the compact limit of 
  \begin{align*}
    \zeta(\alpha_h)(z)=&\int_{\rn_{>0}\times M}\alpha_0(z)d\vol_{\rn_{>0}\times M}+\sum_{\iota\in I}\int_{\rn_{\ge1}\times M}\alpha_\iota(z)d\vol_{\rn_{\ge1}\times M}\\
    &+\sum_{\iota\in I}\int_{(0,1)}(h_\iota+r)^{\dim M+d_\iota+z}(\ln(h_\iota+r))^{l_\iota}dr\res\alpha_\iota(z)
  \end{align*}
  for $h:=\l(h_\iota\r)_{\iota\in I}\in\ell_\infty(I;\rn_{>0})$ and $h\searrow0$ in $\ell_\infty(I)$ such that 
  \begin{align*}
    Z_\iota(z):=l_\iota\sum_{j=0}^{l_\iota}\betr{\zeta_H(l_\iota-j-d_\iota-z;h_\iota)-\zeta_H(l_\iota-j-d_\iota-z;1+h_\iota)}
  \end{align*}
  is bounded on an exhausting family of compacta as $h\searrow0$.
\end{theorem}

The proof of this theorem hinges on the ``construction'' (axiom of dependent choice) of such sequences $h$ and many pages of estimation aimed to prove that boundedness of $Z_\iota$ on an exhausting family of compacta as $h\searrow0$ implies local boundedness of the sequence of $\zeta$-functions (then, we can use Vitali's theorem). The detailed proof can be found in \cite{hartung-phd}*{chapter 6} or in \cite{hartung-scott}*{chapter 5}.

\section{On structural singularities, the residue trace of Fourier Integral Operators, and the generalized Kontsevich-Vishik trace}\label{sec:structural singularities}

In this section, we will discuss the integrals appearing in the Laurent coefficients. Most importantly, this will yield the generalized Kontsevich-Vishik density 
\begin{align}\label{KV-density}
  \begin{aligned}
    \rho_{\mathrm{KV}}:=&\int_{B_{\rn[N]}(0,1)}e^{i\theta(x,x,\xi)}a(0)(x,x,\xi)\ d\xi\ d\vol_X(x)\\
    &+\int_{\rn_{\ge1}\times\d B_{\rn[N]}}e^{i\theta(x,x,\xi)}a_0(0)(x,x,\xi)\ d\vol_{\rn_{\ge1}\times\d B_{\rn[N]}}(\xi)\ d\vol_X(x)\\
    &+\sum_{\iota\in I\setminus I_0}\frac{(-1)^{l_\iota+1}l_\iota!\int_{\d B_{\rn[N]}}e^{i\theta(x,x,\xi)}\tilde a_\iota(0)(x,x,\xi)\ d\vol_{\d B_{\rn[N]}}(\xi)}{(N+d_\iota)^{l_\iota+j+1}}\ d\vol_X(x)
  \end{aligned}
\end{align}
which we obtain by removing the terms with critical degree of homogeneity from the result in Theorem \ref{Laurent-zeta-FIO}, as well as the fact that this density is globally defined if $I_0=\emptyset$, that is in the absence of terms with critical degree of homogeneity. Furthermore, we will study abstract properties of the integrals 
\begin{align*}
  \int_X\int_{\d B_{\rn[N]}}e^{i\theta(x,x,\xi)}a(x,x,\xi)d\vol_{\d B_{\rn[N]}}(\xi) d\vol_X(x)
\end{align*}
in order to decide whether they vanish. A more in-depth analysis of these integrals will use stationary phase approximation and is subject of section \ref{sec:stationary-phase}. An interesting example, reproducing some of Boutet de Monvel's findings in \cite{boutet-de-monvel} and extending them through computation of the generalized Kontsevich-Vishik trace (which turns out to be form equivalent to the pseudo-differential Kontsevich-Vishik trace - mutatis mutandis), will follow Theorem \ref{kernel-stationary-phase}.

Considering classical pseudo-differential operators it is common to start with the Kontsevich-Vishik trace which is constructed by removing those terms from the asymptotic expansion which have degree of homogeneity with real part greater than or equal to $-\dim X$ where $X$ denotes the underlying manifold, i.e. if $k$ is the kernel of the pseudo-differential operator, then the regularized kernel is given by
\begin{align*}
  k_{KV}(x):=\l(k-\sum_{j=0}^Nk_{d-j}\r)(x,x)
\end{align*}
where $d-j\in\cn\setminus\zn_{\ge-\dim X}$ is the degree of homogeneity of the corresponding term in the expansion of the amplitude $a\sim\sum_{j\in\nn_0}a_{d-j}$ and $N$ sufficiently large. Then, $k_{KV}\in C(X)$, i.e. $\int_Xk_{KV}(x)d\vol_X(x)$ is well-defined. In other words, $k_{KV}$ and $\alpha_0$ play the same role and we would like to interpret $\zeta(\alpha_0)(0)$ as a generalized version of the Kontsevich-Vishik trace. The term $\sum_{j=0}^N\int_Xk_{d-j}(x,x)d\vol_X(x)$ would, hence, be analogous to spinning off $\sum_{\iota\in I}\zeta(\alpha_\iota)(0)$. Unfortunately, we have to issue a couple of caveats.
\begin{enumerate}
\item[(i)] The observation above is fine if we are in local coordinates. However, when patching things together, some of the terms in our Laurent expansion will not patch to global densities on $X$. This is no problem for Fourier Integral Operators, per se, as they are simply defined as a sum of local representations and the Laurent expansion holds in each of these representations. Generally, however, we will want to work with globally defined operators and require local terms to patch together defining densities globally. In this sense, all references to global definedness will assume that the kernels of the operators have been globally defined in the first place.
\item[(ii)] Since $\Fp(a_{d-j}(x,y,\cdot))(z)$ is homogeneous of degree $-\dim X-d+j$ (where $\Fp$ denotes the Fourier transform), we obtain $\Fp(a_{d-j}(x,y,\cdot))(0)=0$ for $d-j<-\dim X$, i.e. $k_{d-j}(x,x)=\lim_{y\to x}k_{d-j}(x,y)=\lim_{y\to x}\Fp(a_{d-j}(x,y,\cdot))(y-x)=\Fp(a_{d-j}(x,x,\cdot))(0)$. Thus, $k_{KV}(x,x)$ is independent of $N$. 

  However, this property does not extend to $\alpha_0$ as we can easily construct a counter-example. Let $a(x,y,\xi)$ be homogeneous of degree $d<-n$ in the third argument and the phase function $\theta(x,y,\xi)=-\langle\Theta(x,y),\xi\rangle_{\ell_2(n)}$ such that $\Theta(x,x)$ has no zeros. Then,
  \begin{align*}
    k(x,y)=&\int_{\rn[n]}e^{-i\langle\Theta(x,y),\xi\rangle_{\ell_2(n)}}a(x,y,\xi)d\xi=\Fp(a(x,y,\cdot))(\Theta(x,y))
  \end{align*}
  shows that $k(x,x)$ is well-defined and continuous. Furthermore, since $\Fp(a(x,y,\cdot))$ is homogeneous, $k(x,x)$ vanishing implies $\Fp(a(x,y,\cdot))=0$ on $\{r\Theta(x,x);\ r\in \rn_{>0}\}$.
\end{enumerate}
On the other hand, for pseudo-differential operators the terms $a_{d-j}$ with $d-j=-\dim X$ define a global density on the manifold giving rise to the residue trace. If this extends to poly-$\log$-homogeneous distributions, then we obtain the residue trace globally from $\sum_{\iota\in I_0}\alpha_\iota$. Furthermore, this would imply that 
\begin{align*}
  \fp_0\alpha=\alpha-\sum_{\iota\in I_0}\alpha_\iota
\end{align*}
induces a global density through meromorphic continuation, if $\alpha$ does and the contributions of the $\alpha_\iota$ for $\iota\in I_0$ to the constant term Laurent coefficient vanish (in particular in the non-critical case $I_0=\emptyset$), which allows us to interpret $\zeta(\fp_0\alpha)(0)$ as the generalization of the (non-critical) Kontsevich-Vishik trace. In the critical case, the derivatives of $\alpha_\iota$ terms with $\iota\in I_0$ have to be considered, as well (cf. section \ref{sec:crit-KV}).

This, of course, needs to be interpreted in a gauged sense, that is, after performing the regularization through meromorphic extension. $\zeta(\fp_0\alpha)(0)$ corresponds\footnote{Recall, for pseudo-differential operators with symbol $\sigma$, we can understand $\alpha$ to be given locally by $\alpha(\xi)=\int_X\sigma(x,x,\xi)dx$. Hence, $\alpha$ depends on the choice of local coordinates but the regularized result $\sum_{\mathrm{patches}}\zeta(\alpha)(0)$ is independent of the choice of local coordinates since summing over all local $\zeta(\alpha)(z)$ yields the trace of the corresponding operator for $\Re(z)$ sufficiently small.} to the kernel $k(x,y)-k_{d-j}(x,y)$ where $d-j=-\dim X$. Hence, all terms $k_{d-j}$ with $j\in\nn_{0,<d+\dim X}$ still appear in $\fp_0\alpha$ but not in $k_{KV}$. Since $\zeta(\fp_0\alpha)(0)$ is but constructed by gauging and already regularized, we should do the same for $k_{d-j}$, i.e. consider $k_{d-j+z}$ which is continuous for $\Re(z)$ sufficiently small and vanishes along the diagonal (for pseudo-differential operators). Therefore,
\begin{align*}
  \zeta(\fp_0\alpha)(0)=&\l.\int_X\l(k(z)-k_{z-\dim X}\r)(x,x)d\vol_X(x)\r|_{\mathrm{mer.},z=0}\\
  =&\int_X\l(k-\sum_{j=0}^Nk_{d-j}\r)(x,x)d\vol_X(x)
\end{align*}
holds in the regularized sense; in particular since Corollary \ref{KV-gauge-independence} guarantees that $\zeta(\fp_0\alpha)(0)$ is independent of the gauge. 

Returning to Fourier Integral Operators, $\zeta(\fp_0\alpha)(0)$ is, thus, the best candidate for a generalized Kontsevich-Vishik trace. Since $\rho_{\mathrm{KV}}$ in equation \eqref{KV-density} is the density induced by $\zeta(\fp_0\alpha)(0)$, existence of $\rho_{\mathrm{KV}}$ as a globally defined density is equivalent to the question whether the critical terms $\sum_{\iota\in I_0}\frac{(-1)^{l_\iota+1}\int_M\d^{l_\iota+1}\tilde\alpha_\iota(0)d\vol_M}{l_\iota+1}$, that we removed from the constant Laurent coefficient, and the constant Laurent coefficient induce globally defined densities. In other words (since all principal part Laurent coefficients contain only critical degrees of homogeneity), the objective is to show that 
\begin{align*}
  \sum_{\chi}\res\alpha^\chi(0)=&\sum_{\chi}\l\langle\int_{\d B_{\rn[N]}}\tilde\alpha^\chi(0)d\vol_{\d B_{\rn[N]}},\delta_{\diag}\r\rangle\\
  =&\sum_{\chi}\l\langle P\int_{\d B_{\rn[N]}}\tilde\alpha^\chi(0)d\vol_{\d B_{\rn[N]}},\delta_0\r\rangle\\
  =&\sum_{\chi}\l\langle\int_{\d B_{\rn[N]}}e^{i\theta(x,y,\xi)}\tilde a^\chi(0)(x,y,\xi)\ d\vol_{\d B_{\rn[N]}}(\xi),\delta_0\r\rangle
\end{align*}
is globally well-defined ($\sum_{\chi}$ denotes a partition of unity and $P$ is a suitable pseudo-differential operator) if the $a^\chi$ are $\log$-homogeneous with degree of homogeneity $-N$. Then, all Laurent coefficients in the principal part induce globally defined densities and we may remove the principal part from the $\zeta$-function to obtain the constant Laurent coefficient. 

\begin{remark*}
  A major theme we will discuss in this section is the question whether the generalized Kontsevich-Vishik density $\rho_{\mathrm{KV}}$ as defined in equation \eqref{KV-density}, i.e. the density induced by $\zeta(\fp_0\alpha)(0)$, is globally defined. This question is only relevant if we assume that our Fourier Integral Operators are given by globally defined densities rather than merely a finite sum of localizations $A=\sum_{j=1}^JA_j$ without any assumptions on ``patching properties'' of the kernel. Hence, any statements, that refer to something being globally defined, implicitly assume that the trace integrals, i.e. the densities induced by the $\zeta$-functions, are globally defined in the first place. 

  Locally, $\zeta(\fp_0\alpha)(0)$ is always defined and can be used as the definition of the generalized Kontsevich-Vishik trace. However, $\rho_{\mathrm{KV}}$ will not exist as a global density in that case.

\end{remark*}

Unfortunately, the critical terms $\sum_{\iota\in I_0}\frac{(-1)^{l_\iota+1}\int_M\d^{l_\iota+1}\tilde\alpha_\iota(0)d\vol_M}{l_\iota+1}$ do not have to induce a globally defined density, that is, they are an ``obstruction'' to the existence of the generalized Kontsevich-Vishik trace (a fact well-known from the pseudo-differential theory). We will have a closer look at this obstruction in section \ref{sec:crit-KV} and, in particular, equation~\eqref{eq:prop-coeffs} of Proposition \ref{prop-coeffs} where we compare the constant Laurent coefficients with respect to two different multiplicative gauges. 

\begin{lemma}\label{spherical-integrals}
  Let $a\in C\l(\rn[n]\setminus\{0\}\r)$ be homogeneous of degree $d$, $k\in\nn_0$, $z\in\cn$, and $T\in GL(\rn[n])$. Then
  \begin{align*}
    &\int_{\d B_{\rn[n]}}a(T\xi)\norm{T\xi}^z\l(\ln\norm{T\xi}\r)^kd\vol_{\d B_{\rn[n]}}(\xi)\\
    =&\frac{(-1)^k}{\betr{\det T}}\int_{\d B_{\rn[n]}}a(\xi)\norm{T^{-1}\xi}^{-n-d-z}\l(\ln\norm{T^{-1}\xi}\r)^kd\vol_{\d B_{\rn[n]}}(\xi).
  \end{align*}
\end{lemma}
This lemma (cf. e.g. \cite{lesch}*{equation (2.13)} or \cite{paycha-scott}*{Lemma 2.20} with minimal changes to the proof; \cite{hartung-phd}*{Lemma 7.1}), equation \ref{eq:FIO-trace-diagonal} (chapter \ref{lagrangian}), and \cite{hoermander-FIO-I}*{Proposition 2.4.1} (warranting the existence of $\Theta(x)\in GL(N)$) yield (for a suitable $U\sse_\open\rn[N]$, a diffeomorphism $\chi:\ U\to\chi[U]$, and a $\phi\in C_c^\infty(\chi[U])$)
\begin{align*}
  &\int_U\int_{\d B_{\rn[N]}}\tilde a(x,\xi)\phi(\chi(x))d\vol_{\d B_{\rn[N]}}(\xi)dx\\
  =&\int_U\int_{\d B_{\rn[N]}}\tilde a^\chi(\chi(x),\Theta(x)^{-1}\xi)\betr{\det\Theta(x)^{-1}}\betr{\det\chi'(x)}\phi(\chi(x))d\vol_{\d B_{\rn[N]}}(\xi)dx\\
  =&\int_{U}\int_{\d B_{\rn[N]}}\tilde a^\chi(\chi(x),\Theta(x)^{-1}\xi)\betr{\det\Theta(x)^{-1}}\betr{\det\chi'(x)}\phi(\chi(x))d\vol_{\d B_{\rn[N]}}(\xi)dx\\
  =&\int_{U}\betr{\det\Theta(x)^{-1}}\int_{\d B_{\rn[N]}}\tilde a^\chi(\chi(x),\Theta(x)^{-1}\xi)d\vol_{\d B_{\rn[N]}}(\xi)\betr{\det\chi'(x)}\phi(\chi(x))dx\\
  =&\int_{U}\int_{\d B_{\rn[N]}}\tilde a^\chi(\chi(x),\xi)d\vol_{\d B_{\rn[N]}}(\xi)\betr{\det\chi'(x)}\phi(\chi(x))dx\\
  =&\int_{\chi[U]}\int_{\d B_{\rn[N]}}\tilde a^\chi(x,\xi)\phi(x)d\vol_{\d B_{\rn[N]}}(\xi)dx,
\end{align*}
i.e. the following theorem.

\begin{theorem}\label{res-form-invariant}
  The residue $\res\langle u,f\rangle=\res\alpha(0)=\int_{\d B_{\rn[N]}}\tilde\alpha(0)d\vol_{\d B_{\rn[N]}}$ is form-invariant under change of coordinates if $\alpha(0)$ has degree of homogeneity $-N$.

  In particular, $\sum_\chi\sum_{\iota\in I_0^\chi}\res\alpha_\iota^\chi(0)$ induces a global density and $\sum_\chi\zeta\l(\fp_0\alpha^\chi\r)(0)$ induces a globally defined density provided $\sum_\chi\sum_{\iota\in I_0^\chi}\d\res\alpha_\iota^\chi(0)$ vanishes.
\end{theorem}

\begin{remark*}
  Note that this means that if $a$ is polyhomogeneous and $\iota_0$ is the index such that $a_{\iota_0}$ is homogeneous of degree $-N$, then
  \begin{align*}
    &\sum_{\iota\in I_0}\int_X\int_{\d B_{\rn[N]}}e^{i\theta(x,x,\xi)}a_\iota(x,x,\xi)d\vol_{\d B_{\rn[N]}}(\xi)d\vol_X(x)\\
    =&\int_X\int_{\d B_{\rn[N]}}e^{i\theta(x,x,\xi)}a_{\iota_0}(x,x,\xi)d\vol_{\d B_{\rn[N]}}(\xi)d\vol_X(x).
  \end{align*}
  This, of course, extends to higher order residues
  \begin{align*}
    \int_X\int_{\d B_{\rn[N]}}e^{i\theta(x,x,\xi)}a_{\iota}(x,x,\xi)d\vol_{\d B_{\rn[N]}}(\xi)d\vol_X(x).
  \end{align*}
  with $\iota\in I_0$ and $l_\iota>0$ generalizing \cite{lesch}*{Corollary 4.8} on the residue traces for $\log$-polyhomogeneous pseudo-differential operators.

\end{remark*}
\begin{theorem}\label{KV-trace}
  Let $A_0$ be a Fourier Integral Operator with amplitude $a=a_0+\sum_{\iota\in I}a_\iota$ such that $I_0=\emptyset$. Then, 
  \begin{align*}
    \rho_{\mathrm{KV}}=&\int_{B_{\rn[N]}(0,1)}e^{i\theta(x,x,\xi)}a(0)(x,x,\xi)\ d\xi\ d\vol_X(x)\\
    &+\int_{\rn_{\ge1}\times\d B_{\rn[N]}}e^{i\theta(x,x,\xi)}a_0(0)(x,x,\xi)\ d\vol_{\rn_{\ge1}\times\d B_{\rn[N]}}(\xi)\ d\vol_X(x)\\
    &+\sum_{\iota\in I\setminus I_0}\frac{(-1)^{l_\iota+1}l_\iota!\int_{\d B_{\rn[N]}}e^{i\theta(x,x,\xi)}\tilde a_\iota(0)(x,x,\xi)\ d\vol_{\d B_{\rn[N]}}(\xi)}{(N+d_\iota)^{l_\iota+j+1}}\ d\vol_X(x)
  \end{align*}
  is globally defined and
  \begin{align*}
    \zeta(A)(0)=\int_X\rho_{\mathrm{KV}}
  \end{align*}
  holds for every gauged Fourier Integral Operator $A$ with $A(0)=A_0$.
  
  Furthermore, if $A_0$ is a commutator, then 
  \begin{align*}
    \int_X\rho_{\mathrm{KV}}=0.
  \end{align*}
\end{theorem}
\begin{proof}
  The fact, that $\rho_{\mathrm{KV}}$ is globally defined is a direct corollary of Theorem \ref{res-form-invariant} and $\zeta(A)(0)=\int_X\rho_{\mathrm{KV}}$ follows from the definition of $\rho_{\mathrm{KV}}$ as well as gauge invariance of $\zeta(A)(0)$.

  Let $A_0=[B,C(0)]$ where $B$ is a Fourier Integral Operator and $C$ a gauged Fourier Integral Operator. Then, $[B,C(z)]$ is of trace-class for $\Re(z)$ sufficiently small, i.e. $\tr[B,C(z)]=0$ implies $\zeta([B,C])=0$. In particular, gauge invariance (Corollary \ref{KV-gauge-independence}) implies
  \begin{align*}
    \int_X\rho_{\mathrm{KV}}=\zeta([B,C])(0)=0.
  \end{align*}

\end{proof}
In particular, we choose the following definition for the generalized Kontsevich-Vishik trace.
\begin{definition}
  Let $A_0$ be a Fourier Integral Operator whose amplitude has no critical degrees of homogeneity, i.e. $I_0=\emptyset$, and $A$ a gauged Fourier Integral Operator with $A(0)=A_0$. Then, we define the generalized Kontsevich-Vishik trace of $A$ as
  \begin{align*}
    \tr_{\mathrm{KV}}A_0:=\zeta(A)(0)=\int_X\rho_{\mathrm{KV}}.
  \end{align*}
\end{definition}
Uniqueness of the residue trace, then, directly implies the following proposition.
\begin{prop}\label{unique-patching}
  Let $a\sim \sum_{j\in\nn_0}a_{m-j}$ be the amplitude of a Fourier Integral Operator where $m\in\cn$ and $a_{m-j}$ is homogeneous of degree $m-j$. If the residue trace is the (projectively) unique non-trivial continuous trace, then none of the 
  \begin{align*}
    \int_{\d B_{\rn[N]}}e^{i\theta(x,\xi)}a_{m-j}(x,\xi)d\vol_{\d B_{\rn[N]}}(\xi)
  \end{align*}
  with $m-j\ne-N$ can define a global density, in general, unless they are trivial (i.e. vanish constantly). 

  In particular, removing non-trivial terms from $\zeta(\fp_0\alpha)$ will, in general, destroy global well-definedness of the induced density.
\end{prop}

\subsection{Extending the generalized Kontsevich-Vishik trace to the critical case}\label{sec:crit-KV}
In the critical case $I_0\ne0$, derivatives of the $\alpha_\iota$ with $\iota\in I_0$ appear in the constant Laurent coefficient and, thus, are an obstruction to $\rho_{\mathrm{KV}}$ being globally defined. In this section, we will study this case for an important class of multiplicative gauges.

If we consider a multiplicatively gauged $A(z)=BQ^{z}$ where $Q$ may be non-invertible but is an element of an admissible algebra of Fourier Integral Operators with holomorphic functional calculus, e.g. a pseudo-differential operator of order $1$ (order $q>0$ can be obtained using the results of section \ref{more general distributions}) and spectral cut (the following is to be interpreted in this setting), then $Q^0=1-1_{\{0\}}(Q)$ where
\begin{align*}
  1_{\{0\}}(Q):=\frac{1}{2\pi i}\int_{\d B(0,\eps)}\l(\lambda-Q\r)^{-1}d\lambda
\end{align*}
with $\eps$ sufficiently small such that $B(0,\eps)\cap\sigma(Q)=\{0\}$. Thus, assuming $I_0=\emptyset$ (that is, the Kontsevich-Vishik trace is well-defined and coincides with $\zeta(A)(0)$), we obtain (abusing the notation $\tr$ because $\zeta$ is gauge invariant)
\begin{align*}
  \zeta(A)(0)=&\tr\l(BQ^0\r)=\tr\l(B\r)-\tr\l(B1_{\{0\}}(Q)\r)
\end{align*}
and
\begin{align*}
  \fa k\in\nn:\ \zeta(\d^kA)(0)=&\tr\l(B(\ln Q)^kQ^0\r)=\tr\l(B(\ln Q)^k\r)-\tr\l(B(\ln Q)^k1_{\{0\}}(Q)\r)
\end{align*}
where we note that there still is a dependence on the spectral cut used to define the operators $Q^z$ and $\ln Q$. These generalize the \cite{paycha-scott}*{equations (0.17) and (0.18)} (note that the factors $(-1)^k$ are due to sign convention $Q^z$ vs. $Q^{-z}$).

\begin{prop}\label{prop-coeffs}
  Let $A(z)=BQ^z$ be poly-homogeneous, $\fp\zeta$ the finite part of $\zeta$, and $\tr_\fp$ the finite part of the trace integral (cf. \cite{kontsevich-vishik}, \cite{kontsevich-vishik-geometry}, \cite{lesch}, and \cite{paycha-scott}). Furthermore, let $c_k$ be the coefficient of $\frac{z^k}{k!}$ in the Laurent coefficient with $k\in\nn_0$.

  Then, we obtain
  \begin{align*}
    c_k=&\zeta\l(\d^k\fp_0A\r)(0)+\sum_{\iota\in I_0}\int_X\int_{B_{\rn[N]}(0,1)}e^{i\theta(x,x,\xi)}\d^ka_\iota(0)(x,x,\xi)\ d\xi\ d\vol_X(x)\\
    &-\sum_{\iota\in I_0}\frac{1}{k+1}\res\l(\d^{k+1}A_\iota\r)(0)\\
    =&\fp\zeta\l(\d^kA\r)(0)-\frac{1}{k+1}\res\l(\d^{k+1}A\r)(0)\\
    =&\tr_\fp\l(B(\ln Q)^kQ^0\r)-\frac{1}{k+1}\res\l(B(\ln Q)^{k+1}Q^0\r).
  \end{align*}
  In particular,
  \begin{align*}
    c_0=\tr_\fp\l(B\r)-\res\l(B\ln Q\r)-\tr_\fp\l(B1_{\{0\}}(Q)\r)
  \end{align*}
  and
  \begin{align*}
    \fa k\in\nn:\ c_k=&\tr_\fp\l(B(\ln Q)^k\r)-\frac{1}{k+1}\res\l(B(\ln Q)^{k+1}\r)-\tr_\fp\l(B(\ln Q)^k1_{\{0\}}(Q)\r)
  \end{align*}
  generalize \cite{paycha-scott}*{equations (0.12) and (0.14)} (keeping in mind the factors $(-1)^k$ due to sign convention).

  If $Q$ is invertible, then $1_{\{0\}}(Q)=0$, and for another admissible and invertible operator $Q'$, we obtain
  \begin{align*}\tag{$*$}\label{eq:prop-coeffs}
    c_0(Q)-c_0(Q')=-\res\l(B\l(\ln Q-\ln Q'\r)\r)
  \end{align*}
  which is a generalization of \cite{kontsevich-vishik}*{equation (2.21)} and \cite{paycha}*{equation (9)}. Furthermore, we obtain for $A(z)=[B,CQ^z]$ with invertible $Q$, that $\zeta(A)=0$, i.e. $c_0=0$ and 
  \begin{align*}
    \l.\tr_\fp([B,CQ^z])\r|_{z=0}=\res\l(\l[B,C\ln Q\r]\r)
  \end{align*}
  a generalization of \cite{paycha-scott}*{equation (2.20)}.
\end{prop}

\subsection{The residues}
Now we may ask when the residues vanish. As a first result we obtain the well-known fact that the residue trace vanishes for odd-class operators on odd-dimensional manifolds.
\begin{obs}
  Let $\alpha(-\xi)=-\alpha(\xi)$. Then, $\res\alpha=\int_{\d B_{\rn[N]}}\alpha(\xi)d\vol_{\d B_{\rn[N]}}(\xi)=0$.
\end{obs}
Note that the property $\alpha(-\xi)=-\alpha(\xi)$ is invariant under change of linear phase functions with the same ``$N$''. Choosing non-linear phase functions or changing $N$ might destroy this property.

On the other hand, if $N=1$, then 
\begin{align*}
  \int_{\d B_{\rn}}\alpha(\xi)d\vol_{\d B_{\rn}}(\xi)=\alpha(1)+\alpha(-1)
\end{align*}
shows that $\res\alpha$ vanishes if and only if $\alpha$ is odd. Equivalently, we obtain
\begin{align*}
  \int_{\d B_{\rn}}e^{i\theta(x,\xi)}a(x,\xi)d\vol_{\d B_{\rn}}(\xi)=e^{i\theta(x,1)}a(x,1)+e^{i\theta(x,-1)}a(x,-1).
\end{align*}
Note, this implies there are two residue traces for $N=1$; namely, $\alpha_{-1}(1)$ and $\alpha_{-1}(-1)$. 

For $N>1$, the de Rham co-homology of $\d B_{\rn[N]}$ is given by
\begin{align*}
  \fa k\in\nn_0:\ H^k_\dR\l(\d B_{\rn[N]}\r)\cong
  \begin{cases}
    \rn&,\ k\in\{0,N-1\}\\
    0&,\ k\in\nn\setminus\{N-1\}
  \end{cases}.
\end{align*}
In other words, there exists $\omega_0\in\Omega^{N-1}\l(\d B_{\rn[N]},\cn\r)$ such that $\int_{\d B_{\rn[N]}}\omega_0=1$ and
\begin{align*}
  \fa\omega\in\Omega^{N-1}\l(\d B_{\rn[N]},\cn\r)\ \ex c\in\cn\ \ex\tilde\omega\in\Omega^{N-2}\l(\d B_{\rn[N]},\cn\r):\ \omega=c\omega_0+d\tilde\omega.
\end{align*}
Thus, $\int_{\d B_{\rn[N]}}e^{i\theta(x,\xi)}a(x,\xi)d\vol_{\d B_{\rn[N]}}(\xi)=0$ if and only if $e^{i\theta(x,\cdot)}a(x,\cdot)d\vol_{\d B_{\rn[N]}}$ is an exact differential form.
\begin{remark*}
  Since we are integrating $\dim M$-forms over a manifold $M$, we assume that all manifolds are orientable as we can only integrate pseudo-$\dim M$-forms if $M$ is non-orientable. So far everything can be re-formulated for pseudo-forms and, thus, on non-orientable manifolds. From this point onwards, though, statements will need orientability; in particular with respect to uniqueness of residue traces and the commutator structure since
\begin{align*}
  H^{\dim M}_\dR(M)
  \cong
  \begin{cases}
    \rn&,\ M\text{ orientable, connected}\\
    0&,\ M\text{ non-orientable, connected}
  \end{cases}.
\end{align*}

\end{remark*}
\begin{definition}
  Let $A$ be a poly-homogeneous Fourier Integral Operator on a compact manifold $X$ and $\res_0\zeta(A)$ be locally given by
  \begin{align*}
    \int_X\int_{\d B_{\rn[N]}}e^{i\theta(x,\xi)}a(x,\xi)\ d\vol_{\d B_{\rn [N]}}(\xi)\ d\vol_X(x).
  \end{align*}
  Then, we call the $(N-1+\dim X)$-form $\rho(A)$ on $X\times\d B_{\rn [N]}$, locally defined as
  \begin{align*}
    \rho(A):=\exp\circ(i\theta)\cdot a\ d\vol_{X\times\d B_{\rn [N]}},
  \end{align*}
  the residue form of $A$ (in other words, $*\rho(A)=e^{i\theta}a$ where $*$ denotes the Hodge-$*$-operator).
\end{definition}
\begin{prop}\label{res-form-closed}
  Let $Y\sse X$ be a connected component. Then, $\int_{Y\times\d B_{\rn[N]}}\rho(A)=0$ if and only if $\rho(A)$ is exact on $Y\times\d B_{\rn[N]}$.

  More precisely, let $X=Y_1\dcup\ldots\dcup Y_k$ be composed of finitely many connected components ($\dcup$ denotes the disjoint union) and let $\rho(A)|_{Y_j\times\d B_{\rn[N]}}=c_j\omega_j+d\tilde\omega_j$ be the corresponding decompositions of $\rho(A)$ with $\omega_j=\vol_{Y_j\times\d B_{\rn[N]}}(Y_j\times\d B_{\rn[N]})^{-1}d\vol_{Y_j\times\d B_{\rn[N]}}$. Then,
  \begin{align*}
    \int_{X\times\d B_{\rn[N]}}\rho(A)=\sum_{j=1}^kc_j.
  \end{align*}
\end{prop}
Using the Hodge-$*$-operator $*$, the co-derivative $d^*:=\l(-1\r)^{N_X\l(N_X-1\r)+1}*d*$ with $N_X:=N+\dim X-1$, as well as
\begin{align*}
  \rho(A)=d\omega\ \iff\ e^{i\theta}a=&*d\omega=d^*\l((-1)^{N_X^2}*\omega\r),
\end{align*}
and the divergence $\div F=*d*F^\flat=(-1)^{N_X\l(N_X-1\r)+1}d^*F^\flat$ with the musical isomorphism
\begin{align*}
  \cdot^\flat:\ T\l(X\times\d B_{\rn[N]}\r)\to T^*\l(X\times\d B_{\rn[N]}\r);\ \sum_iF_i\d_i\mapsto \sum_iF_idx_i,
\end{align*}
we can re-formulate Proposition \ref{res-form-closed}.
\begin{theorem}\label{vanishing-spheric-integral}
  Let $X$ be connected. Then, the following are equivalent.
  \begin{enumerate}
  \item[(i)] $\int_X\int_{\d B_{\rn[N]}}e^{i\theta(x,\xi)}a(x,\xi)\ d\vol_{\d B_{\rn [N]}}(\xi)\ d\vol_X(x)=0$.
  \item[(ii)] There exists an $(N+\dim X-2)$-form $\omega$ on $X\times\d B_{\rn[N]}$ such that $d\omega=e^{i\theta}a\ d\vol_{X\times\d B_{\rn[N]}}$ locally.
  \item[(iii)] There exists a $1$-form $\omega$ on $X\times\d B_{\rn[N]}$ such that $d^*\omega=e^{i\theta}a$ locally.
  \item[(iv)] There exists a vector field $F$ on $X\times\d B_{\rn[N]}$ such that $\div F=e^{i\theta}a$ locally.
  \end{enumerate}
\end{theorem}

\begin{remark*}
  These results hold if we replace $\d B_{\rn[N]}$ by any other connected manifold $M$ and consider the residue terms $\res\alpha=\int_M\hat\alpha d\vol_M$ for poly-$\log$-homogeneous distributions. In particular, we obtain $\res\alpha=0$ if and only if there exists a vector field $F$ on $M$ such that $\hat\alpha=\div F$.

\end{remark*}

\begin{remark*}
  \emph{Condition (iv) can be extended to $X\times\l(\rn[N]\setminus\{0\}\r)$.} Let $M:=X\times\d B_{\rn[N]}$, $(g_i)_i$ the local frame in which $e^{i\theta}a$ is given by $\alpha$, and $(g^i)_i$ the dual frame. Let $\tilde M:=\rn_{>0}\times M\cong X\times\l(\rn[N]\setminus\{0\}\r)$ and the metric tensor is of the form
  \begin{align*}
    \tilde g(r,\xi)=
    \begin{pmatrix}
      1&0\\0&r^{2\dim M}g(\xi)
    \end{pmatrix},
  \end{align*}
  i.e. $d\vol_{\tilde M}(r,\xi)=\sqrt{\det\tilde g(r,\xi)}dr\wedge d\xi=r^{\dim M}\sqrt{\det g(\xi)}dr\wedge d\xi=r^{\dim M}dr\wedge d\vol_M(\xi)$. Let $F$ be a vector field on $M$ and $\tilde F$ be a vector field on $\tilde M$. Then,
  \begin{align*}
    \div F(\xi)=&\tr\grad F(\xi)=\sum_{j=1}^{\dim M}\sum_{i=1}^{\dim M}\d_jF_i(\xi)g^{ji}(\xi)
  \end{align*}
  and
  \begin{align*}
    \div\tilde F(r,\xi)=&\d_0\tilde F_0(r,\xi)+r^{2\dim M}\sum_{j=1}^{\dim M}\sum_{i=1}^{\dim M}\d_j\tilde F_i(r,\xi)g^{ji}(\xi).
  \end{align*}
  In other words, we obtain $\div\tilde F(1,\xi)=\div F(\xi)$ if $\d_0\tilde F_0(1,\xi)=0$ and $\d_j\tilde F_i(1,\xi)=\d_jF_i(\xi)$. On the other hand, we want $\div F(\xi)=\tilde \alpha(\xi)$ and $\div\tilde F(r,\xi)=f(r)\tilde\alpha(\xi)$ with $f(1)=1$. Choosing $\tilde F_0=0$ and $\tilde F_i(r,\xi)=f(r)F_i(\xi)$ implies $\div\tilde F(r,\xi)=f(r)\tilde\alpha(\xi)$ and $\div\tilde F(1,\xi)=\div F(\xi)$. 

  Thus, knowing (iv) we can construct a vector field $\tilde F$ such that $e^{i\theta}=\div\tilde F$ on $X\times\l(\rn[N]\setminus\{0\}\r)$ and $\tilde F$ satisfies the conditions above. Conversely, if $\tilde F$ has the described properties, then $\tilde F|_{X\times\d B_{\rn[N]}}$ satisfies (iv).

\end{remark*}
At this point, using the framework of gauged poly-$\log$-homogeneous distributions, we can follow the lines of \cite{guillemin-lagrangian}*{Theorem 1.1} (using gauged poly-$\log$-homogeneous distributions) to obtain the following theorem (\cite{guillemin-lagrangian}*{Theorem 1.2}, \cite{hartung-phd}*{Theorem 7.9}) which we state here for completeness.

\begin{theorem}\label{commutator-structure}
  Let $\Ap_\Gamma$ be an algebra of classical Fourier Integral Operators associated with the canonical relation $\Gamma$ such that the twisted relation $\Gamma'$ ($A\in\Ap_\Gamma\ \iff\ k_A\in I(X^2,\Gamma')$) has clean and connected intersection with the co-normal bundle of diagonal in $X^2$. Then, the residue-trace of $A\in\Ap_\Gamma$ vanishes if and only if $A$ is a smoothing operator plus a sum of commutators $[P_i,A_i]$ where the $P_i$ are pseudo-differential operators and the $A_i\in\Ap_\Gamma$.
\end{theorem}
Guillemin also proved the following (more general) version of Theorem \ref{commutator-structure} (cf. \cite{guillemin-residue-traces}*{Proposition 4.11}).
\begin{prop}
  Let $\Gamma$ be connected. Then, the commutator of $\Ap_\Gamma$ is of co-dimension one in $\Ap_\Gamma$ modulo smoothing operators.
\end{prop}
Hence, $\res_0\circ\zeta$ is either zero or the unique trace on $\Ap_\Gamma$ up to a constant factor. Regarding the trace of smoothing operators, \cite{guillemin-residue-traces}*{Theorems A.1 and A.2} yield the commutator structure of smoothing operators (the following two definitions, the theorem, and the remark can all be found in the appendix of \cite{guillemin-residue-traces}).

\begin{definition}
  Let $H$ be a separable Hilbert space and $e:=(e_i)_{i\in\nn}$ an orthonormal basis of $H$. An operator $A\in L(H)$ is called smoothing with respect to $e$ if and only if
  \begin{align*}
    \fa n\in\nn\ \ex c\in\rn:\ \betr{\langle Ae_i,e_j\rangle_H}\le c(i+j)^{-n}.
  \end{align*}
\end{definition}

\begin{definition}
  Let $H$ be a separable Hilbert space, $e$ an orthonormal basis, $\Omega\sse_\open\Kn[n]$ with $\Kn\in\{\rn,\cn\}$ and $A\in L(H)^\Omega$ such that each $A(s)$ is smoothing with respect to $e$. Then, $A$ is said to be smooth/holomorphic if and only if all $s\mapsto\langle A(s)e_i,e_j\rangle_H$ are $C^\infty(\Omega)$.
\end{definition}

\begin{theorem}
  \begin{enumerate}
  \item[(i)] If $A$ is smoothing with respect to $e$ and $\tr A=0$, then $A$ can be written as a finite sum of commutators $[B_i,C_i]$ where the $B_i$ and $C_i$ are smoothing with respect to $e$.
  \item[(ii)] If a family $A\in L(H)^\Omega$ of smoothing operators is smooth/holomorphic, then $A$ can be written as a finite sum of commutators $s\mapsto[B_i(s),C_i]$ on every compact $K\sse\Omega$ where the $B_i(s)$ and $C_i$ are smoothing, and the $B_i$ are smooth/holomorphic.
  \end{enumerate}
\end{theorem}

\begin{remark*}
  \begin{enumerate}
  \item[(i)] Let $X$ be a compact Riemannian manifold, $H=L_2(X)$, and $e$ the family of eigenfunctions of the Laplacian on $X$. An operator $A\in L\l(L_2(X)\r)$ is smoothing with respect to $e$ if it is smoothing with respect to the Sobolev norms.
  \item[(ii)] Let $H=L_2(\rn[n])$ and $e$ the family of Hermite functions. An operator $A\in L(H)$ is $e$-smoothing if it is smoothing with respect to the Schwartz semi-norms.
  \end{enumerate}

\end{remark*}
These theorems yield the following table assuming that the (unique) residue trace $\res_0\circ\zeta$ is non-trivial and $\Ap_\Gamma=\langle\mathfrak{A}\rangle+\l\langle\l[\Ap_\Gamma,\Ap_\Gamma\r]\r\rangle+\{\text{smoothing operators}\}$ for some $\mathfrak{A}\in\Ap_\Gamma$ with $\res_0\zeta(\mathfrak{A})\ne0$.
\begin{center}
  \begin{tabular}{|c|c|c|c|}
    \hline
    \multicolumn{2}{|c|}{$I_0\ne\emptyset$}&\multicolumn{2}{c|}{$I_0=\emptyset$}\\\hline
    $\res_0\zeta(A)\ne0$&$\res_0\zeta(A)=0$&$\zeta(A)(0)\ne0$&$\zeta(A)(0)=0$\\\hline
    $
    \begin{array}{l}
      A=\alpha\mathfrak{A}+S+\sum_{i=1}^kC_i\\
      C_i\in\l[\Ap_\Gamma,\Ap_\Gamma\r]\\
      \alpha=\l(\res_0\zeta(\mathfrak{A})\r)^{-1}\res_0\zeta(A)\\
      S\text{ smoothing}
    \end{array}$&\multicolumn{2}{c|}{
    $
    \begin{array}{l}
      A=S+\sum_{i=1}^kC_i\\
      C_i\in\l[\Ap_\Gamma,\Ap_\Gamma\r]\\
      S\text{ smoothing}
    \end{array}$}&
    $
    \begin{array}{l}
      A=\sum_{i=1}^kC_i\\
      C_i\text{ commutators}
    \end{array}$\\\hline
  \end{tabular}
\end{center}

\begin{remark*}
  Note that the obstruction to the generalized Kontsevich-Vishik trace is given by the derivatives of the $a_\iota$ for $\iota\in I_0$. Using the example above Theorem \ref{Laurent-distrib}, we obtain that these are residue traces themselves if the operator is poly-homogeneous. These residues are explicitly computed for gauged families $A(z)=BQ^z$ in Proposition \ref{prop-coeffs}.

\end{remark*}

\section{Stationary phase approximation}\label{sec:stationary-phase}
In this section, we would like to get to know a little more about the singularity structure of
\begin{align*}
  k(x,y)=\int_{\rn[N]}e^{i\theta(x,y,\xi)}a(x,y,\xi)d\xi,
\end{align*}
primarily to compute the integrals
\begin{align*}
  \int_{\d B_{\rn[N]}}e^{i\theta(x,y,\xi)}a(x,y,\xi)d\vol_{\d B_{\rn[N]}}(\xi).
\end{align*}
We want to understand these integrals in more detail for two reasons. First, we need to know which of these integrals over $\rn[N]$ are regular in order to decide whether one such integral is an $\alpha_\iota$ or belongs to $\alpha_0$. Second, the integrals over the sphere $\d B_{\rn[N]}$ are essentially the contributions of the $\alpha_\iota$ to the Laurent coefficients. Hence, studying those integrals allows us to actually compute the residues and the generalized Kontsevich-Vishik trace.

We will skip many calculations in this chapter because they are very tedious and differ only slightly (if at all) from the calculations that can be found in any account on stationary phase approximation (e.g. \cite{hoermander-books}*{chapter 7.7}). 

For the remainder of the section, let $a$ be $\log$-homogeneous. Then, we obtain
\begin{align*}
  k(x,y)=&\int_{\rn_{>0}}r^{N+d-1}(\ln r)^l\ubr{\int_{\d B_{\rn[N]}}e^{ir\theta(x,y,\eta)}a(x,y,\eta)d\vol_{\d B_{\rn[N]}}(\eta)}_{=:I(x,y,r)}dr.
\end{align*}
Let $(x,y)$ be off the critical manifold, i.e. $\fa \xi\in\d B_{\rn[N]}:\ \d_3\theta(x,y,\xi)\ne0$. Then, we observe
\begin{align*}
  \fa n\in\nn:\ \betr{I(x,y,r)}=&\frac{1}{r}\betr{\int_{\d B_{\rn[N]}}e^{ir\theta(x,y,\xi)}\Dp a(x,y,\xi)d\vol_{\d B_{\rn[N]}}(\xi)}\\
  =&\frac{1}{r^n}\betr{\int_{\d B_{\rn[N]}}e^{ir\theta(x,y,\xi)}\Dp[n] a(x,y,\xi)d\vol_{\d B_{\rn[N]}}(\xi)}\\
  \le&\frac{1}{r^n}\norm{\Dp[n] a}_{L_\infty\l(X\times X\times\d B_{\rn[N]}\r)},
\end{align*}
where
\begin{align*}
  \Dp a(x,y,\xi):=\d_3^*\frac{a(x,y,\xi)\d_3\theta(x,y,\xi)}{\norm{\d_3\theta(x,y,\xi)}^2_{\ell_2(N)}},
\end{align*}
which proves the well-known fact that $k$ is $C^\infty$ away from the critical manifold. 

On the critical manifold, we will assume that 
\begin{align*}
  \d_3^2\l(\theta|_{X\times X\times\d B_{\rn[N]}}\r)(x,y,\xi)\in GL\l(\rn[N-1]\r)
\end{align*}
if $\d_3\theta(x,y,\xi)=0$ (note that this holds for pseudo-differential operators). Then, we are in a position to apply Morse' Lemma.

\begin{lemma}[Morse' Lemma]\label{morse-lemma}
  Let $(x_0,y_0,\xi_0)\in X\times X\times\d B_{\rn[N]}$ be stationary (in particular, $\d_{\d B}\theta(x_0,y_0,\xi_0)=0$) and $\d_{\d B}^2\theta(x_0,y_0,\xi_0)\in GL\l(\rn[N-1]\r)$ where $\d_{\d B}$ denotes the spherical derivative, i.e. the derivative in $\d B_{\rn[N]}$.

  Then, there are neighborhoods $U\sse_\open X\times X$ of $(x_0,y_0)$ and $V\sse_\open\d B_{\rn[N]}$ of $\xi_0$ and a function $\hat\xi\in C^\infty(U,V)$ such that 
  \begin{align*}
    \fa (x,y,\xi)\in U\times V:\ \d_{\d B}\theta(x,y,\xi)=0\ \iff\ \xi=\hat\xi(x,y).
  \end{align*}
  Furthermore, there is a function $\eta\in C^\infty\l(U\times V,\rn[N]\r)$ such that
  \begin{align*}
    \fa (x,y,\xi)\in U\times V:\ \eta(x,y,\xi)-\l(\xi-\hat\xi(x,y)\r)\in O\l(\norm{\xi-\hat\xi(x,y)}_{\ell_2(N)}^2\r)
  \end{align*}
  and 
  \begin{align*}
    \d_3\eta\l(x,y,\hat\xi(x,y)\r)=1.
  \end{align*}
\end{lemma}
\begin{corollary}
  Let $\theta$ be as in Morse' Lemma (Lemma \ref{morse-lemma}). Then, stationary points of $\theta(x,y,\cdot)$ are isolated in $\d B_{\rn[N]}$. In particular, there are only finitely many.
\end{corollary}
\begin{proof}
  For given stationary $(x,y,\xi)$ we can find a neighborhood $V\sse_\open\d B_{\rn[N]}$ such that $\xi=\hat\xi(x,y)$; thus, stationary points are locally unique. By compactness of $\d B_{\rn[N]}$ they are isolated and at most finitely many. 

\end{proof}
Hence, we may assume that
\begin{align*}
  k(x,y)=\sum_{s=0}^S\int_{\rn[N]}e^{i\theta(x,y,\xi)}a^s(x,y,\xi)d\xi
\end{align*}
where $a^0$ has no stationary points in its support and each of the $a^s$ has exactly one branch $(x,y,\hat\xi^s(x,y))$ of stationary points in its support. As we have already treated the $a^0$ case, we will assume, without loss of generality, that $a$ is of the form of one of the $a^s$. 

Using a stereographic projection $\sigma:\ \rn[N-1]\to\d B_{\rn[N]}$ with pole $-\hat\xi^s(x,y)$ (which is assumed to be outside of $\spt a^s(x,y,\cdot)$), we are in a position to use the standard set of techniques employed in applications of the stationary phase approximation (the detailed computation can be found in \cite{hartung-phd}*{chapter 8} or (slightly compressed) in \cite{hartung-scott}*{chapter 7}). 

\begin{theorem}\label{kernel-stationary-phase}
  Let $k(x,y)=\int_{\rn[N]}e^{i\theta(x,y,\xi)}a(x,y,\xi)d\xi$ be the kernel of a Fourier Integral Operator with poly-$\log$-homogeneous amplitude $a=a_0+\sum_{\iota\in I}a_\iota$. Let $\tilde I:=I\cup\{0\}$ and choose a decomposition $a=a^0+\sum_{s=1}^S a^s$ such that there is no stationary point in the support of $a^0(x,y,\cdot)$ and exactly one stationary point $\hat\xi^s(x,y)\in\d B_{\rn[N]}$ of $\theta(x,y,\cdot)$ in the support of each $a^s(x,y,\cdot)$. 

  Let $\hat\theta^s(x,y)=\theta\l(x,y,\hat\xi^s(x,y)\r)$, $\Theta^s(x,y)=\d_{\d B}^2\theta\l(x,y,\hat\xi^s(x,y)\r)$, $\sgn\Theta^s(x,y)$ the number of positive eigenvalues minus the number of negative eigenvalues of $\Theta^s(x,y)$, and $\Delta_{\d B,\Theta^s(x,y)}=\l\langle\Theta^s(x,y)^{-1}\d_{\d B},\d_{\d B}\r\rangle=-\div_{\d B_{\rn[N]}}\Theta^s(x,y)^{-1}\grad_{\d B_{\rn[N]}}$. Furthermore, let
  \begin{align*}
    h_{j,\iota}^s(x,y):=\frac{(2\pi)^{\frac{N-1}{2}}\betr{\det\Theta^s(x,y)}^{-\frac12}e^{\frac{i\pi}{4}\sgn\Theta^s(x,y)}}{j!(2i)^j}\Delta_{\d B,\Theta^s}^ja_\iota^s\l(x,y,\hat\xi^s(x,y)\r)
  \end{align*}
  and
  \begin{align*}
    g_{j,\iota}^s(x,y):=
    \begin{cases}
      \d^{l_\iota}\l(z\mapsto\Gamma\l(q_\iota+1+z\r)i^{q_\iota+1+z}\l(\hat\theta^s(x,y)+i0\r)^{-q_\iota-1-z}\r)(0)&,\ q_\iota\notin-\nn_0\\
      \d^{l_\iota}\l(z\mapsto\frac{-\Gamma(z+1)}{2\pi i\ (-q_\iota)!}\int_{c+i\rn}\frac{\l(-\sigma\r)^{-q_\iota}\l(c_{\ln}+\ln\sigma\r)}{\l(-i\hat\theta^s(x,y)+0-\sigma\r)^{z+1}}d\sigma\r)(0)&,\ q_\iota\in -\nn_0
    \end{cases}
  \end{align*}
  with $q_\iota:=d_\iota+\frac{N+1}{2}-j$, $c\in\rn_{>0}$, and some constant $c_{\ln}\in\cn$.

  Then,
  \begin{align*}
    k(x,y)=&\int_{\rn[N]}e^{i\theta(x,y,\xi)}a^0(x,y,\xi)d\xi+\sum_{\iota\in \tilde I}\sum_{s=1}^S\sum_{j\in\nn_0}h_{j,\iota}^s(x,y)g_{j,\iota}^s(x,y)
  \end{align*}
  holds in a neighborhood of the diagonal in $X^2$.
\end{theorem}
\begin{example*}
  Note that in the $N=1$ case everything collapses as there are no spherical derivatives. We will simply obtain
  \begin{align*}
    k_d(x,y)=&\int_{\rn_{>0}}r^de^{ir\theta(x,y,1)}a_d(x,y,1)dr+\int_{\rn_{>0}}r^de^{ir\theta(x,y,-1)}a_d(x,y,-1)dr
  \end{align*}
  and
  \begin{align*}
    &\int_{\rn_{>0}}r^de^{ir\theta(x,y,\pm1)}a_d(x,y,\pm1)dr\\
    =&
    \begin{cases}
      c_da_d(x,y,\pm1)\l(\theta(x,y,\pm1)+i0\r)^{-d-1}&,\ d\notin-\nn\\
      a_d(x,y,\pm1)\frac{\l(i\theta(x,y,\pm1)-0\r)^{-d-1}}{(-d-1)!}\l(c_{d}+\ln\l(-i\theta(x,y,\pm1)+0\r)\r)&,\ d\in-\nn
    \end{cases}
  \end{align*}
  with some constants $c_d$. Hence, for
  \begin{align*}
    k(x,y)\sim\sum_{j\in\nn_0}\int_{\rn}e^{i\theta(x,y,\xi)}a_{d-j}(x,y,\xi)d\xi
  \end{align*}
  with $d\in\zn$ and $a_{d-j}$ homogeneous of degree $d-j$, the coefficient of the logarithmic terms are
  \begin{align*}
    \sum_{j\in\nn_{\ge d+1}}a_{d-j}(x,y,\pm1)\frac{\l(i\theta(x,y,\pm1)-0\r)^{j-d-1}}{(j-d-1)!}.
  \end{align*}
  In particular, in the critical case $\theta(x,x,\pm1)=0$ (as studied by Boutet de Monvel \cite{boutet-de-monvel}), we are reduced to the fact that the densities of the residue traces at $x$ (that is, $a_{-1}(x,x,\pm1)$) coincide with the coefficients of the logarithmic terms (that is, $\ln\l(-i\theta(x,x,\pm1)+0\r)$) in the singularity structure of $k$ (cf. \cite{boutet-de-monvel}*{equations (3) and (4)}).

  Furthermore, we can compute the generalized Kontsevich-Vishik trace for $a=a_0+\sum_{\iota\in I}a_\iota$ if $\fa\iota\in I:\ d_\iota\in\rn\setminus\{-1\}$ and $l_\iota=0$. Then, the kernel $k$ satisfies (note $\theta(x,x,r)=0$ by assumption)
  \begin{align*}
    k(x,x)=&\int_{\rn_{>0}}a_0(x,x,r)dr+\sum_{\iota\in I}\int_{\rn_{>0}}a_\iota(x,x,r)dr.
  \end{align*}
  Since $1_{\rn_{>0}}a_\iota(x,x,\cdot)$ is homogeneous of degree $d_\iota$, we obtain that $\int_{\rn_{>0}}a_\iota(x,x,r)dr$ vanishes for $d_\iota<-1$ since the Fourier transform $\Fp(1_{\rn_{>0}}a_\iota(x,x,\cdot))$ over $\rn$ is a homogeneous distribution of degree $-1-d_\iota$. For $d_\iota>-1$, we obtain
  \begin{align*}
    \int_{\rn_{>0}}e^{i\theta(x,y,r)}a_\iota(x,x,r)dr=c_\iota a_\iota(x,y,1)\l(\theta(x,y,1)+i0\r)^{-d_\iota-1}
  \end{align*}
  which is precisely the other singular contribution to the kernel singularity (that is the $f(x,y)(\varphi+0)^{-N}$ term in \cite{boutet-de-monvel}*{equation (3)}). In other words, the difference of $k(x,y)$ and its singular part $k^{\mathrm{sing}}(x,y)$ satisfies
  \begin{align*}
    \l(k-k^{\mathrm{sing}}\r)(x,x)=&\int_{\rn_{>0}}a_0(x,x,r)dr.
  \end{align*}
  In order to use Theorem \ref{Laurent-zeta-FIO}, we will have to show that the regularized singular terms vanish. This follows directly from the Laurent expansion with mollification. For $d_\iota>-1$, we have the two terms
  \begin{align*}
    &\sum_{n\in\nn_0}\frac{\int_X\int_0^1e^{i\theta(x,x,\xi)}\d^na_\iota(0)(x,x,\xi)d\xi d\vol_X(x)}{n!}z^n\\
    &+\sum_{n\in\nn_0}\sum_{j=0}^n\frac{(-1)^{j+1}j!\int_Xe^{i\theta(x,x,1)}\d^na_\iota(0)(x,x,1) d\vol_X(x)}{n!(1+d_\iota)^{j+1}}z^n
  \end{align*}
  to evaluate at $z=0$, i.e.
  \begin{align*}
    \lim_{h\searrow0}\int_X\int_0^1(h+r)^{d_\iota}a_\iota(x,x,1)drd\vol_X(x)
    =&\int_X\frac{a_\iota(x,x,1)}{d_\iota+1}d\vol_X(x)
  \end{align*}
  and
  \begin{align*}
    \frac{-\int_Xa_\iota(x,x,1)d\vol_X(x)}{1+d_\iota}.
  \end{align*}
  Hence, the generalized Kontsevich-Vishik trace reduces to the pseudo-differential form. Let $a\sim\sum_{j\in\nn_0}a_{d-j}$ and $N$ be sufficiently large, then
  \begin{align*}
    \tr_{KV}A=\int_X\int_{\rn_{>0}}a(x,x,r)-\sum_{j=0}^Na_{d-j}(x,x,r)\ dr\ d\vol_X(x)
  \end{align*}
  is independent of $N$.
\end{example*}

The example above, i.e. operators with phase functions as considered by Boutet de Monvel in \cite{boutet-de-monvel}, are an important class of Fourier Integral Operators satisfying $\theta=0$ on the diagonal. This is the case ``closest to pseudo-differential operators'' for which we obtain the following theorem (for the proof, see \cite{hartung-phd}*{Theorem 8.5} or \cite{hartung-scott}*{Theorem 7.5}).
\begin{theorem}\label{psi-DO-KV-trace}
  Let $A$ be a Fourier Integral Operator with kernel
  \begin{align*}
    k(x,y)=\int_{\rn[N]}e^{i\theta(x,y,\xi)}a(x,y,\xi)d\xi
  \end{align*}
  whose phase function $\theta$ satisfies $\fa x\in X\ \fa\xi\in\rn[N]:\ \theta(x,x,\xi)=0$ and whose amplitude has an asymptotic expansion $a\sim\sum_{\iota\in \nn}a_\iota$ where each $a_\iota$ is $\log$-homogeneous with degree of homogeneity $d_\iota$ and logarithmic order $l_\iota$, and $\Re(d_\iota)\to-\infty$. Let $N_0\in\nn$ such that $\fa\iota\in\nn_{> N_0}:\ \Re(d_\iota)<-N$ and let 
  \begin{align*}
    k^{\mathrm{sing}}(x,y)=\int_{\rn[N]}e^{i\theta(x,y,\xi)}\sum_{\iota=1}^{N_0}a_\iota(x,y,\xi)d\xi
  \end{align*}
  denote the singular part of the kernel. 

  Then, the regularized kernel $k-k^{\mathrm{sing}}$ is continuous along the diagonal and independent of the particular choice of $N_0$ (along the diagonal). Furthermore, the generalized Kontsevich-Vishik density\footnote{Mind that this density is only locally defined. It only patches together (modulo pathologies) if we assume the kernel patched together in the first place and the derivatives of terms of critical dimension $d_\iota=-N$ regularize to zero, i.e. if $\zeta(\fp_0A)(0)$ is tracial and independent of gauge.} is given by
  \begin{align*}
    \l(k-k^{\mathrm{sing}}\r)(x,x)d\vol_X(x)
    =&\int_{\rn[N]}a(x,x,\xi)-\sum_{\iota=1}^{N_0}a_\iota(x,x,\xi)d\xi d\vol_X(x).
  \end{align*}
\end{theorem}

Finally, we will consider an example of linear phase functions which will be generalized to find algebras of Fourier Integral Operators which are Hilbert-Schmidt with continuous kernels. Let $\theta(x,y,\xi):=\l\langle\Theta(x,y),\xi\r\rangle_{\rn[N]}$ and $\Theta(x_0,y_0)\ne0$. Then,
\begin{align*}
  k(x,y)=&\int_{\rn[N]}e^{i\l\langle\Theta(x,y),\xi\r\rangle_{\rn[N]}}a(x,y,\xi)d\xi
  =\Fp\l(a(x,y,\cdot)\r)\l(-\Theta(x,y)\r)
\end{align*}
is continuous in a sufficiently small neighborhood of $(x_0,y_0)$ for homogeneous $a$ because $\Fp\l(a(x,y,\cdot)\r)$ is homogeneous and $\Theta(x,y)$ non-zero. Hence, if $\Theta$ does not vanish on the diagonal, then $X\ni x\mapsto k(x,x)\in\cn$ is continuous and, by compactness of $X$, $\int_Xk(x,x)d\vol_X(x)$ well-defined. 

The stationary phase approximation above generalizes this observation ($\hat\xi(x,y)=\pm\frac{\Theta(x,y)}{\norm{\Theta(x,y)}_{\ell_2(N)}}$, i.e. $\hat\theta^s(x,y)=(-1)^s\norm{\Theta(x,y)}_{\ell_2(N)}$ with $s\in\{0,1\}$).

\begin{theorem}\label{hilbert-schmidt-algebras}
  Let $A$ be a Fourier Integral Operator with phase function $\theta$ satisfying $\d_3^2\l(\theta|_{X\times X\times\d B_{\rn[N]}}\r)(x,y,\xi)\in GL\l(\rn[N-1]\r)$ whenever $\d_3\theta(x,y,\xi)=0$ (in a neighborhood of the diagonal) and $\l\{\hat\xi^s;\ s\in\nn_{\le n}\r\}$ the set of stationary points. Furthermore, let
  \begin{align*}
    \fa x\in X\ \fa s\in\nn_{\le n}:\ \theta\l(x,x,\hat\xi^s(x,x)\r)\ne0.
  \end{align*}
  Then,
  \begin{align*}
    \l(\ X\ni x\mapsto k(x,x)\in\cn\ \r)\in C(X)
  \end{align*}
  and 
  \begin{align*}
    \tr A=\int_Xk(x,x)d\vol_X(x)
  \end{align*}
  is well-defined, i.e. $A$ is a Hilbert-Schmidt operator. Furthermore, $\zeta$-functions of such operators have no poles (since the trace integral always exists).
\end{theorem}

An example for such operators occurs on quotient manifolds. Let $\Gamma$ be a co-compact discrete group on $M$ acting continuously\footnote{$\Gamma\times M/_\Gamma\ni (\gamma,x)\mapsto\gamma x\in M/_\Gamma$ is continuous} and freely\footnote{$\fa\gamma\in\Gamma:\ \l(\ex x\in M/_\Gamma:\gamma x=x\r)\ \then\ \gamma=\id$} on $M/_\Gamma$, $\tilde k$ a $\Gamma\times\Gamma$-invariant\footnote{$\fa \gamma\in\Gamma\ \fa x,y\in M:\ \tilde k(x,y)=\tilde k(\gamma x,\gamma y)$} Schwartz kernel on $M$, and $k$ its projection to $M/_\Gamma$. Then, $k(x,y)=\sum_{\gamma\in\Gamma}\tilde k(x,\gamma y)$. Suppose $\tilde k$ is pseudo-differential, i.e. 
\begin{align*}
  \tilde k(x,y)=\int_{\rn[N]}e^{i\langle x-y,\xi\rangle_{\rn[N]}}a(x,y,\xi)d\xi.
\end{align*}
Then,
\begin{align*}
  k(x,y)=\sum_{\gamma\in\Gamma}\int_{\rn[N]}e^{i\langle x-\gamma y,\xi\rangle_{\rn[N]}}a(x,\gamma y,\xi)d\xi.
\end{align*}
Hence, for $\gamma=\id$ we have a pseudo-differential part and for $\gamma\ne\id$ the phase function $\theta_\gamma(x,y,\xi)=\langle x-\gamma y,\xi\rangle_{\rn[N]}$ has stationary points $\pm\frac{x-\gamma y}{\norm{x-\gamma y}_{\ell_2(N)}}$, that is, $\theta_\gamma\l(x,y,\hat\xi^s(x,y)\r)=(-1)^s\norm{x-\gamma y}_{\ell_2(N)}$ does not vanish in a neighborhood of the diagonal.

\begin{example*}
  Let us consider manifolds $M$ with diagonal metric, that is, the metric tensor is given by
  \begin{align*}
    g^{ij}(x)=\gf(x)^2\delta^{ij}
  \end{align*}
  with some function $\gf$. An example of these are hyperbolic manifolds.\index{hyperbolic manifold} Let
  \begin{align*}
    \Hn[N]:=\{x\in\rn[N];\ x_N>0\}
  \end{align*}
  with the metric
  \begin{align*}
    g_{ij}(x)=\gf(x)^{-2}\delta_{ij}=x_N^{-2}\delta_{ij}.
  \end{align*}
  Then, $\sqrt{\betr{\det g(x)}}=\gf(x)^{-N}$. The Laplace-Beltrami operator on $M$ is given by
  \begin{align*}
    \Delta_{M}=\gf(x)^2\sum_{i=1}^n\d_i^2
  \end{align*}
  and the wave operator $\exp(it\sqrt{\betr{\Delta_{M}}})$ has the kernel
  \begin{align*}
    \kappa_{M}(x,y)=(2\pi)^{-N}\int_{\rn[N]}e^{i\langle x-y,\xi\rangle_{\rn[N]}}e^{it\gf(x)\norm\xi_{\ell_2(N)}}d\xi.
  \end{align*}
  Let $\Gamma$ be a co-compact, discrete, torsion-free sub-group of the isometries of $M$ such that $\Gamma$ is a lattice and $X:=M/_\Gamma$ can be identified with a fundamental domain in $M$ under action of $\Gamma$. If $M=\Hn[N]$, we call $X$ a hyperbolic manifold. Since $\Gamma$ is a subset of the isometries, the metric on $X$ is given by the metric on $M$ taking a representative of the orbit and the wave-operator $\exp(it\sqrt{\betr\Delta})$ factors through with the kernel
  \begin{align*}
    \kappa(x,y)=\sum_{\gamma\in\Gamma}(2\pi)^{-N}\int_{\rn[N]}e^{i\langle x-\gamma y,\xi\rangle_{\rn[N]}}e^{it\gf(x)\norm\xi_{\ell_2(N)}}d\xi.
  \end{align*}
  Let $A^t$ be a gauged Fourier Integral Operator with $A^t(0)=\exp(it\sqrt{\betr\Delta})$. Then, $A^t(0)=\sum_{\gamma\in\Gamma}A^t_\gamma(0)$ and each $A^t_\gamma(0)$ has the phase function
  \begin{align*}
    \theta_\gamma(x,y,\xi)=\langle x-\gamma y,\xi\rangle_{\rn[N]}+t\gf(x)\norm\xi_{\ell_2(N)}
  \end{align*}
  and amplitude $(x,y,\xi)\mapsto1$, i.e. each $\zeta(A^t_\gamma)$ is holomorphic in a neighborhood of zero. Thus, Lemma \ref{gauge-independence-regular} yields that $\zeta(A^t_\gamma)$ is independent of the gauge and we obtain
  \begin{align*}
    \zeta(A^t)(0)=&\sum_{\gamma\in\Gamma}\zeta(A^t_\gamma)(0)=\sum_{\gamma\in\Gamma}(2\pi)^{-N}\int_X\int_{\rn[N]}e^{i\langle x-\gamma x,\xi\rangle_{\rn[N]}}e^{it\gf(x)\norm\xi_{\ell_2(N)}}d\xi dx.
  \end{align*}

  For $\gamma=1$ (the identity) we will use the property
  \begin{align*}
    \fa q\in\cn_{\Re(\cdot)>-1}:\ \Lp\l(r\mapsto r^q\r)(s)=\int_{\rn_{>0}}r^qe^{-sr}dr=\Gamma(q+1) s^{-q-1}
  \end{align*}
  of the Laplace transform (where $\Gamma$ is the $\Gamma$-function) and obtain
  \begin{align*}
    \zeta(A_1^t)(0)=&\frac{(N-1)!\vol_{\d B_{\rn[N]}}\l(\d B_{\rn[N]}\r)\vol_X(X)}{(-2\pi it)^N}.
  \end{align*}

   For $\gamma\in\Gamma\setminus\{1\}$ we know $x-\gamma x\ne0$ and stationary points of $\theta_\gamma(x,x,\cdot)$ are $\xi_\pm^\gamma(x):=\pm\frac{x-\gamma x}{\norm{x-\gamma x}_{\ell_2(N)}}$ (since the term $t\gf(x)\norm\xi_{\ell_2(N)}$ vanishes taking derivatives with respect to $\xi\in\d B_{\rn[N]}$) with
  \begin{align*}
    \theta_\gamma\l(x,x,\xi_\pm^\gamma(x)\r)=&\l\langle x-\gamma x,\pm\frac{x-\gamma x}{\norm{x-\gamma x}_{\ell_2(N)}}\r\rangle_{\rn[N]}+t\gf(x)\norm{\pm\frac{x-\gamma x}{\norm{x-\gamma x}_{\ell_2(N)}}}_{\ell_2(N)}\\
    =&t\gf(x)\pm\norm{x-\gamma x}_{\ell_2(N)}.
  \end{align*}
  Since $\gf$ is a positive continuous function and $X$ compact, we obtain that $\gf$ is bounded away from zero and $x\mapsto\norm{x-\gamma x}_{\ell_2(N)}$ is bounded, i.e. $\theta_\gamma\l(x,x,\xi_\pm(x)\r)$ has no zeros for $t$ sufficiently large (similarly for $t$ sufficiently small). By Theorem \ref{hilbert-schmidt-algebras}, we obtain that each $\zeta(A^t_\gamma)(0)$ exists for sufficiently large $t$ (and sufficiently large $-t$, $\frac1t$, and $-\frac1t$, as well). 

  Hence, we want to evaluate
  \begin{align*}
    \zeta(A_\gamma^t)(0)=&(2\pi)^{-N}\int_X\int_{\rn[N]}e^{i\l\langle x-\gamma x,\xi\r\rangle_{\rn[N]}}e^{it\gf(x)\norm\xi_{\ell_2(N)}}d\xi dx\\
    =&(2\pi)^{-N}\int_X\int_{\rn_{>0}}r^{N-1}e^{it\gf(x)r}\int_{\d B_{\rn[N]}}e^{ir\l\langle x-\gamma x,\eta\r\rangle_{\rn[N]}}d\vol_{\d B_{\rn[N]}}(\eta)dr dx.
  \end{align*}
  $\int_{\d B_{\rn[N]}}e^{ir\l\langle x-\gamma x,\eta\r\rangle_{\rn[N]}}d\vol_{\d B_{\rn[N]}}(\eta)$ can be evaluated using stationary phase approximation. The stationary points are 
  \begin{align*}
    \eta_\pm(x):=\pm\frac{x-\gamma x}{\norm{x-\gamma x}_{\ell_2(N)}}
  \end{align*}
  and the corresponding phase function $\hat\theta(x,\eta):=r\l\langle x-\gamma x,\eta\r\rangle_{\rn[N]}$ satisfies
  \begin{align*}
    \hat\theta(x,\eta_\pm(x))=\pm r\norm{x-\gamma x}_{\ell_2(N)}.
  \end{align*}
  Since the amplitude is the constant function $1$, all higher order derivatives in the stationary phase approximation yield zero and we obtain
  \begin{align*}
    \int_{\d B_{\rn[N]}}e^{ir\l\langle x-\gamma x,\eta\r\rangle_{\rn[N]}}d\vol_{\d B_{\rn[N]}}(\eta)=&\norm{x-\gamma x}_{\ell_2(N)}^{-\frac{N-1}{2}}\l(\frac\pi2\r)^{\frac{N-1}{2}}e^{-\frac{i\pi}{4}(N-1)}e^{ir\norm{x-\gamma x}_{\ell_2(N)}}\\
    &+\norm{x-\gamma x}_{\ell_2(N)}^{-\frac{N-1}{2}}\l(\frac\pi2\r)^{\frac{N-1}{2}}e^{-\frac{i\pi}{4}(N-1)}e^{-ir\norm{x-\gamma x}_{\ell_2(N)}}
  \end{align*}
  which, in turn, yields
  \begin{align*}
    \zeta(A^t_\gamma)(0)
    =&\frac{\l(\frac\pi2\r)^{\frac{N-1}{2}}e^{-\frac{i\pi}{4}(N-1)}(N-1)!}{(-2\pi i)^N}\int_X\frac{\norm{x-\gamma x}_{\ell_2(N)}^{-\frac{N-1}{2}}}{\l(t\gf(x)+\norm{x-\gamma x}_{\ell_2(N)}\r)^N}dx\\
    &+\frac{\l(\frac\pi2\r)^{\frac{N-1}{2}}e^{-\frac{i\pi}{4}(N-1)}(N-1)!}{(-2\pi i)^N}\int_X\frac{\norm{x-\gamma x}_{\ell_2(N)}^{-\frac{N-1}{2}}}{\l(t\gf(x)-\norm{x-\gamma x}_{\ell_2(N)}\r)^N}dx.
  \end{align*}
  Let us consider the special case of a flat torus, that is, $\gf=1$ and $\gamma x=\gamma+x$. Then, the formula collapses to
  \begin{align*}
    \zeta(A^t_\gamma)(0)=&\sum_{\pm}\frac{\l(\frac\pi2\r)^{\frac{N-1}{2}}e^{-\frac{i\pi}{4}(N-1)}(N-1)!\vol_X(X)}{(-2\pi i)^N}\norm{\gamma}_{\ell_2(N)}^{-\frac{N-1}{2}}\l(t\pm\norm{\gamma}_{\ell_2(N)}\r)^{-N}.
  \end{align*}
  This shows the well-known result that poles of the $\zeta$-regularized wave trace can only occur if $t$ is equal to the length of a closed geodesic $\norm\gamma_{\ell_2(N)}$ \cite{duistermaat-guillemin} and for all other $t$, we obtain that $\zeta(A^t)(0)$ is given by
  \begin{align*}
      \frac{(N-1)!\vol_X(X)}{(-2\pi i)^N}
      \l(\frac{\vol_{\d B_{\rn[N]}}\l(\d B_{\rn[N]}\r)}{t^N}+\sum_{\gamma\in\Gamma}\sum_{\pm}\frac{\l(\frac\pi2\r)^{\frac{N-1}{2}}e^{-\frac{i\pi}{4}(N-1)}\norm{\gamma}_{\ell_2(N)}^{-\frac{N-1}{2}}}{\l(t\pm\norm{\gamma}_{\ell_2(N)}\r)^N}\r).
  \end{align*}

\end{example*}

\begin{example*}
  In light of the last example, we can even go a step further and consider manifolds where the Laplacian has the symbol $g^{ij}(x)\xi_i\xi_j=\langle\xi,G^{-1}(x)\xi\rangle_{\ell_2(N)}$, i.e.
  \begin{align*}
    \zeta(A^t)(0)=&\sum_{\gamma\in\Gamma}(2\pi)^{-N}\int_X\int_{\rn[N]}e^{it\norm{G^{-\frac12}(x)\xi}_{\ell_2(N)}}e^{i\langle x-\gamma x,\xi\rangle_{\rn[N]}}d\xi dx.
  \end{align*}
  Using Fubini's theorem
  \begin{theorem*}[Theorem (Fubini)]\index{Fubini's Theorem}
    Let $\Omega\sse\rn[n]$ be open, $\phi\in C_c(\Omega)$, $f\in C^1(\Omega,\rn)$, $\fa x\in\Omega:\ \grad f(x)\ne0$, and $M_r:=[\{r\}]f=\{x\in\Omega;\ f(x)=r\}$. Then, 
    \begin{align*}
      \int_\Omega\phi(x)dx=\int_{\rn}\int_{M_r}\phi(\xi)\norm{\grad f(\xi)}^{-1}_{\ell_2(n)}d\vol_{M_r}(\xi)dr.
    \end{align*}
  \end{theorem*}
  with $f(\xi)=\norm{G^{-\frac12}(x)\xi}_{\ell_2(N)}$ on $\rn[N]\setminus\{0\}$, i.e. $\grad f(\xi)=\frac{G^{-1}(x)\xi}{\norm{G^{-\frac12}(x)\xi}_{\ell_2(N)}}$, gives rise to the definition
  \begin{align*}
    \fa x\in X:\ M_x:=\l\{\frac{\xi}{\norm{G^{-\frac12}(x)\xi}_{\ell_2(N)}}\in\rn[N];\ \xi\in\d B_{\rn[N]}\r\}
  \end{align*}
  and, thus,
  \begin{align*}
    &(2\pi)^{-N}\int_X\int_{\rn[N]}e^{it\norm{G^{-\frac12}(x)\xi}_{\ell_2(N)}}e^{i\langle x-\gamma x,\xi\rangle_{\rn[N]}}d\xi dx\\
    =&\int_X\int_{\rn_{>0}}\int_{rM_x}\frac{e^{it\norm{G^{-\frac12}(x)\tilde\mu}_{\ell_2(N)}+i\langle x-\gamma x,\tilde\mu\rangle_{\rn[N]}}\norm{G^{-\frac{1}{2}}(x)\tilde\mu}_{\ell_2(N)}}{(2\pi)^{N}\norm{G^{-1}(x)\tilde\mu}_{\ell_2(N)}}d\vol_{rM_x}(\tilde\mu)dr dx\\
    =&(2\pi)^{-N}\int_X\int_{\rn_{>0}}\int_{M_x}e^{ir\l(t+\langle x-\gamma x,\mu\rangle_{\rn[N]}\r)}\frac{\norm{G^{-\frac{1}{2}}(x)\mu}_{\ell_2(N)}}{\norm{G^{-1}(x)\mu}_{\ell_2(N)}}r^{N-1}d\vol_{M_x}(\mu)dr dx\\
    =&(2\pi)^{-N}\int_X\int_{\rn_{>0}}e^{irt}r^{N-1}\int_{M_x}e^{ir\langle x-\gamma x,\mu\rangle_{\rn[N]}}\norm{G^{-1}(x)\mu}_{\ell_2(N)}^{-1}d\vol_{M_x}(\mu)dr dx.
  \end{align*}
  Note that integrals similar to $\int_{M_x}e^{ir\langle x-\gamma x,\mu\rangle_{\rn[N]}}\norm{G^{-1}(x)\mu}_{\ell_2(N)}^{-1}d\vol_{M_x}(\mu)$ also appear if we choose such a decomposition of $\rn[N]$ and want to compute the Laurent coefficients. Furthermore, note that we can re-write those integrals over $M_x$ into integrals over the sphere; namely,
  \begin{align*}
    \int_{M_x}fd\vol_{M_x}=&\int_{\d B_{\rn[N]}}f\circ\Psi_x\sqrt{\det\l(d\Psi_x^Td\Psi_x\r)}d\vol_{\d B_{\rn[N]}}
  \end{align*}
  with
  \begin{align*}
    \Psi_x(\xi):=\frac{\xi}{\norm{G^{-\frac12}(x)\xi}_{\ell_2(N)}}.
  \end{align*}

  For $\gamma=1$, these integrals simply reduce to
  \begin{align*}
    \frac{(N-1)!}{(-2\pi it)^N}\int_X\int_{\d B_{\rn[N]}}\frac{\norm{G^{-\frac12}(x)\xi}_{\ell_2(N)}}{\norm{G^{-1}(x)\xi}_{\ell_2(N)}}\eth\Psi_x(\xi)d\vol_{\d B_{\rn[N]}}(\xi)dx
  \end{align*}
  where $\eth\Psi_x(\xi):=\sqrt{\det\l(d\Psi_x(\xi)^Td\Psi_x(\xi)\r)}$.

  For $\gamma\ne1$, we want to evaluate
  \begin{align*}
    &(2\pi)^{-N}\int_{\rn_{>0}}e^{irt}r^{N-1}\int_X\int_{M_x}e^{ir\langle x-\gamma x,\mu\rangle_{\rn[N]}}\norm{G^{-1}(x)\mu}_{\ell_2(N)}^{-1}d\vol_{M_x}(\mu)dxdr.
  \end{align*}
  The stationary points are characterized by $x-\gamma x\perp T_\mu M_x$ and we can change coordinates in the $M_x$ integral to obtain
  \begin{align*}
    \int_X\int_{\d B_{\rn[N]}}e^{ir\l\langle x-\gamma x,\Psi_x(\xi)\r\rangle_{\rn[N]}}\norm{G^{-1}(x)\Psi_x(\xi)}_{\ell_2(N)}^{-1}\eth\Psi_x(\xi)d\vol_{\d B_{\rn[N]}}(\xi)dx.
  \end{align*}
  In particular, for the torus, we have $\gamma x=\gamma +x$ and
  \begin{align*}
    \int_X\int_{\d B_{\rn[N]}}e^{-ir\l\langle\gamma,\Psi_x(\xi)\r\rangle_{\rn[N]}}\norm{G^{-1}(x)\Psi_x(\xi)}_{\ell_2(N)}^{-1}\eth\Psi_x(\xi)d\vol_{\d B_{\rn[N]}}(\xi)dx
  \end{align*}
  can be evaluated applying the stationary phase approximation to
  \begin{align*}
    \int_{\d B_{\rn[N]}}e^{-ir\l\langle\gamma,\Psi_x(\xi)\r\rangle_{\rn[N]}}\norm{G^{-1}(x)\Psi_x(\xi)}_{\ell_2(N)}^{-1}\eth\Psi_x(\xi)d\vol_{\d B_{\rn[N]}}(\xi).
  \end{align*}

\end{example*}

\begin{remark*}
  Replacing $\d B_{\rn[N]}$ by $M_x$ becomes even more interesting if we want to compute residual integrals in the Laurent coefficients 
  \begin{align*}
    \int_{\Delta(X)\times\d B_{\rn[N]}}e^{i\theta(x,x,\xi)}\d^{n+l_\iota+1}\tilde a_\iota(0)(x,x,\xi)\ d\vol_{\Delta(X)\times\d B_{\rn[N]}}(x,\xi)
  \end{align*}
  which are now integrals 
  \begin{align*}
    \int_X\int_{M_x}e^{i\theta(x,x,\xi)}\d^{n+l_\iota+1}\tilde a_\iota(0)(x,x,\xi)\ d\vol_{M_x}(\xi)d\vol_X(x).
  \end{align*}
  In cases such as the example above, the integration over $M_x$ is without a phase function because $M_x\ni\xi\mapsto\theta(x,x,\xi)$ is a constant $\theta_x$, leaving us with integrals of the form
  \begin{align*}
    e^{i\theta_x}\int_{M_x}a_x(\xi)\ d\vol_{M_x}(\xi)
  \end{align*}
  where $a_x$ is homogeneous of some degree $d$. For $M_x=T_x\l[\d B_{\rn[n]}\r]$ with $T_x\in GL(\rn[n])$, this is equivalent to
  \begin{align*}
    e^{i\theta_x}\int_{M_x}a_x(\xi)\ d\vol_{M_x}(\xi)
    =&e^{i\theta_x}\int_{\d B_{\rn[n]}}a_x(\xi)\norm{T_x^{-1}\xi}^{-n-d}d\vol_{\d B_{\rn[n]}}(\xi).
  \end{align*}
  In particular, for the case of the residue trace, we have $d=-n$, i.e.
  \begin{align*}
    e^{i\theta_x}\int_{M_x}a_x(\xi)\ d\vol_{M_x}(\xi)=&e^{i\theta_x}\int_{\d B_{\rn[n]}}a_x(\xi)d\vol_{\d B_{\rn[n]}}(\xi),
  \end{align*}
  which shows that we have reduced the pointwise residue of the Fourier Integral Operator to the pointwise residue of a suitably chosen pseudo-differential operator and a rotation in the complex plane $\theta_x$. In fact, the symbol of that pseudo-differential operator can be chosen to be the amplitude of the Fourier Integral Operator itself.

\end{remark*}

\section{Conclusion}
Based on Guillemin’s work \cite{guillemin-lagrangian,guillemin-residue-traces} on the residue trace for Fourier Integral Operators, we have developed an extension of the theory of $\zeta$-functions for pseudo-differential operators to a large class of Fourier Integral Operators. By introducing the notion of gauged poly-log-homogeneous distributions explicitly (section \ref{gauged-poly-log-homogeneous-distributions}) and, thus, working in a generalized setting that shares the fundamental analytical structures of pseudo-differential operator $\zeta$-functions, we were able to study the Laurent expansion of Fourier Integral Operator $\zeta$-functions (Theorem \ref{Laurent-zeta-FIO}) and prove existence of a generalized Kontsevich-Vishik trace (Theorem \ref{KV-trace}). Most notably, many methods developed for pseudo-differential operator $\zeta$-functions are still applicable with only minor adjustments. 

In conjunction with stationary phase expansion results for the Laurent coefficients and the kernel singularity structure (Theorem \ref{kernel-stationary-phase}), we have extended many known formulae from the pseudo-differential operator case to various classes of Fourier Integral Operators (e.g. the trace defect formulae in Proposition \ref{prop-coeffs} and Fourier Integral Operators whose generalized Kontsevich-Vishik trace is form equivalent to the pseudo-differential Kontsevich-Vishik trace; Theorem \ref{psi-DO-KV-trace}). Furthermore, these considerations allowed us to identify non-trivial algebras of Fourier Integral Operators consisting purely of Hilbert-Schmidt operators with regular trace integrals (Theorem \ref{hilbert-schmidt-algebras}), as well as utilize our unified approach to independently verify many known results for special cases of Fourier Integral Operators (e.g. Theorem \ref{commutator-structure} or Examples following Theorems \ref{kernel-stationary-phase} (operators as considered by Boutet de Monvel \cite{boutet-de-monvel}) and \ref{hilbert-schmidt-algebras} (wave traces)).

\begin{appendix}
\setcounter{section}{0}
\renewcommand{\thesection}{\Alph{section}}
  
\section{The heat trace, fractional, and shifted fractional Laplacians on flat tori}\label{examples}
In this appendix, we will apply Theorem \ref{Laurent-zeta-FIO} to some examples which are well-known or can be easily checked through spectral considerations.

\begin{example*}[Example (the Heat Trace on the flat torus $\mathbb{R}^N/_\Gamma$)]
  Let $\Gamma\sse\rn[N]$ be a discrete group generated by a basis of $\rn[N]$, $\betr\Delta$ the Dirichlet Laplacian on $\rn[N]$, $\delta$ the Dirichlet Laplacian on $\rn[N]/_\Gamma$, and $T$ the semi-group generated by $-\delta$ on $\rn[N]/_\Gamma$. It is well-known that 
  \begin{align*}
    \tr T(t)=\frac{\vol_{\rn[N]/_\Gamma}\l(\rn[N]/_\Gamma\r)}{(4\pi t)^{\frac N2}}\sum_{\gamma\in\Gamma}\exp\l(-\frac{\norm\gamma_{\ell_2(N)}^2}{4t}\r)
  \end{align*}
  holds; cf. e.g. \cite{scott}*{equation 3.2.3.28}. Furthermore, the kernel $\kappa_\delta$ of $\delta$ is given by the kernel $\kappa_{\betr\Delta}$ via $\kappa_\delta(x,y)=\sum_{\gamma\in\Gamma}\kappa_{\betr\Delta}(x,y\gamma)$; cf. e.g. \cite{scott}*{section 3.2.2}. In other words,
  \begin{align*}
    \kappa_\delta(x,y)=\sum_{\gamma\in\Gamma}\int_{\rn[N]}e^{i\langle x-y-\gamma,\xi\rangle}(2\pi)^{-N}\norm\xi_{\ell_2(N)}^2d\xi.
  \end{align*}
  Hence, using functional calculus, we obtain
  \begin{align*}
    \kappa_{T(t)}(x,y)=\sum_{\gamma\in\Gamma}\int_{\rn[N]}e^{i\langle x-y-\gamma,\xi\rangle}(2\pi)^{-N}e^{-t\norm\xi_{\ell_2(N)}^2}d\xi.
  \end{align*}
  Considering some gauge of $T(t)$ we obtain from the Laurent expansion (Theorem \ref{Laurent-zeta-FIO})
  \begin{align*}
    &\zeta(T(t))(0)\\
    =&\int_{\rn[N]/_\Gamma\times B_{\rn[N]}}\sum_{\gamma\in\Gamma}e^{-i\langle\gamma,\xi\rangle}(2\pi)^{-N}e^{-t\norm\xi_{\ell_2(N)}^2}\ d\vol_{\rn[N]/_\Gamma\times B_{\rn[N]}}(x,\xi)\\
    &+\int_{\rn[N]/_\Gamma\times\l(\rn_{\ge1}\times\d B_{\rn[N]}\r)}\sum_{\gamma\in\Gamma}e^{-i\langle\gamma,\xi\rangle}(2\pi)^{-N}\l(e^{-t\norm\cdot_{\ell_2(N)}^2}\r)_0(\xi)\ d\vol_{\rn[N]/_\Gamma\times\l(\rn_{\ge1}\times\d B_{\rn[N]}\r)}(x,\xi)\\
    &+\sum_{\iota\in I}\frac{(-1)^{l_\iota+1}l_\iota!\res(T(t))_\iota}{(N+d_\iota)^{l_\iota+1}}.
  \end{align*}
  Since $\l(\xi\mapsto e^{-t\norm\xi_{\ell_2(N)}^2}\r)\in\Sp(\rn[N])$, we can choose $I=\emptyset$ and $\l(e^{-t\norm\cdot_{\ell_2(N)}^2}\r)_0=e^{-t\norm\cdot_{\ell_2(N)}^2}$ which yields
  \begin{align*}
    &\zeta(T(t))(0)\\
    =&\int_{\rn[N]/_\Gamma\times B_{\rn[N]}}\sum_{\gamma\in\Gamma}e^{-i\langle\gamma,\xi\rangle}(2\pi)^{-N}e^{-t\norm\xi_{\ell_2(N)}^2}\ d\vol_{\rn[N]/_\Gamma\times B_{\rn[N]}}(x,\xi)\\
    &+\int_{\rn[N]/_\Gamma\times\l(\rn_{\ge1}\times\d B_{\rn[N]}\r)}\sum_{\gamma\in\Gamma}e^{-i\langle\gamma,\xi\rangle}(2\pi)^{-N}e^{-t\norm\xi_{\ell_2(N)}^2}\ d\vol_{\rn[N]/_\Gamma\times\l(\rn_{\ge1}\times\d B_{\rn[N]}\r)}(x,\xi)\\
    =&\frac{\vol_{\rn[N]/_\Gamma}\l(\rn[N]/_\Gamma\r)}{(2\pi)^{N}}\int_{B_{\rn[N]}}\sum_{\gamma\in\Gamma}e^{-i\langle\gamma,\xi\rangle}e^{-t\norm\xi_{\ell_2(N)}^2}\ d\vol_{B_{\rn[N]}}(\xi)\\
    &+\frac{\vol_{\rn[N]/_\Gamma}\l(\rn[N]/_\Gamma\r)}{(2\pi)^{N}}\int_{\rn_{\ge1}\times\d B_{\rn[N]}}\sum_{\gamma\in\Gamma}e^{-i\langle\gamma,\xi\rangle}e^{-t\norm\xi_{\ell_2(N)}^2}\ d\vol_{\rn_{\ge1}\times\d B_{\rn[N]}}(\xi)\\
    =&\frac{\vol_{\rn[N]/_\Gamma}\l(\rn[N]/_\Gamma\r)}{(2\pi)^{N}}\sum_{\gamma\in\Gamma}\int_{\rn[N]}e^{-i\langle\gamma,\xi\rangle}e^{-t\norm\xi_{\ell_2(N)}^2}\ d\xi\\
    =&\frac{\vol_{\rn[N]/_\Gamma}\l(\rn[N]/_\Gamma\r)}{(4\pi^2)^{\frac N2}}\sum_{\gamma\in\Gamma}\pi^{\frac N2}t^{-\frac N2}e^{-\frac{\norm\gamma_{\ell_2(N)}^2}{4t}}\\
    =&\frac{\vol_{\rn[N]/_\Gamma}\l(\rn[N]/_\Gamma\r)}{(4\pi t)^{\frac N2}}\sum_{\gamma\in\Gamma}\exp\l(-\frac{\norm\gamma_{\ell_2(N)}^2}{4t}\r),
  \end{align*}
  i.e. precisely what we wanted to obtain.

\end{example*}

Please note that the following example exceeds the applicability of the $\zeta$-function Laurent expansion as stated in section \ref{lagrangian}. However, in section \ref{sec:mollification} we showed that the formula still holds.

\begin{example*}[Example (fractional Laplacians on $\rn/_{2\pi\Zn}$)]
  On $\Tn:=\rn/_{2\pi\Zn}$, let us consider the operator $H:=\sqrt{\betr{\Delta}}$ where $\betr\Delta$ denotes the (non-negative) Laplacian. It is well-known that the spectrum $\sigma(H)=\nn_0$ is discrete and each non-zero eigenvalue has multiplicity $2$. Furthermore, the symbol of $H^z$ has the kernel
  \begin{align*}
    \kappa_{H^z}(x,y)=\sum_{n\in\zn}\int_{\rn}e^{i(x-y-2\pi n)\xi}\frac{\betr\xi^z}{2\pi}d\xi.
  \end{align*}
  The singular part is given for $n=0$ and $\sum_{n\in\zn\setminus\{0\}}\int_{\rn}e^{i(x-y-2\pi n)\xi}\frac{\betr\xi^z}{2\pi}d\xi$ is regular.

  Let $\alpha\in(-1,0)$. Since $\zeta$ is the spectral $\zeta$-function, we obtain ($\mu_\lambda$ denoting the multiplicity of $\lambda$ and $\Re(z)<-1$)
  \begin{align*}
    \zeta\l(s\mapsto H^sH^\alpha\r)(z)=\sum_{\lambda\in\sigma\l(H\r)\setminus\{0\}}\mu_\lambda\lambda^{z+\alpha}=2\sum_{n\in\nn}n^{z+\alpha}=2\zeta_R(-z-\alpha)
  \end{align*}
  where $\zeta_R$ denotes Riemann's $\zeta$-function. In particular, 
  \begin{align*}
    \zeta\l(s\mapsto H^sH^\alpha\r)(0)=&2\zeta_R(-\alpha).
  \end{align*}

  On the other hand, we have the Laurent expansion (Theorem \ref{Laurent-zeta-FIO})
  \begin{align*}
    \zeta\l(s\mapsto H^sH^\alpha\r)(z)=&\sum_{k\in\nn_0}\frac{1}{k!}\l(\int_{\Delta(\Tn)\times B_{\rn}}e^{i\theta}\sigma\l(\l(\ln H\r)^kH^\alpha\r)\ d\vol_{\Delta(\Tn)\times\d B_{\rn}}\r.\\
    &\qquad+\int_{\Delta(\Tn)\times\l(\rn_{\ge1}\times\d B_{\rn}\r)}e^{i\theta}\sigma\l(\l(\ln H\r)^kH^\alpha\r)_0\ d\vol_{\Delta(\Tn)\times\l(\rn_{\ge1}\times\d B_{\rn}\r)}\\
    &\qquad+\l.\sum_{\iota\in I}\frac{(-1)^{l_\iota+1}l_\iota!\res\l(\l(\ln H\r)^kH^\alpha\r)_\iota}{(1+d_\iota)^{l_\iota+1}}\r)z^k,
  \end{align*}
  i.e.
  \begin{align*}
    \zeta\l(s\mapsto H^sH^\alpha\r)(0)=&\int_{\Delta(\Tn)\times B_{\rn}}e^{i\theta}\sigma\l(H^\alpha\r)\ d\vol_{\Delta(\Tn)\times\d B_{\rn}}\\
    &+\int_{\Delta(\Tn)\times\l(\rn_{\ge1}\times\d B_{\rn}\r)}e^{i\theta}\sigma\l(H^\alpha\r)_0\ d\vol_{\Delta(\Tn)\times\l(\rn_{\ge1}\times\d B_{\rn}\r)}\\
    &+\sum_{\iota\in I}\frac{(-1)^{l_\iota+1}l_\iota!\res\l(H^\alpha\r)_\iota}{(1+d_\iota)^{l_\iota+1}}.
  \end{align*}
  Plugging in our kernel yields
  \begin{align*}
    \zeta\l(s\mapsto H^sH^\alpha\r)(0)=&\sum_{n\in\zn}\int_0^{2\pi}\int_{-1}^1e^{-2\pi in\xi}\frac{\betr\xi^\alpha}{2\pi}\ d\xi\ dx\\
    &+\sum_{n\in\zn\setminus\{0\}}\int_0^{2\pi}\int_{\rn_{\le1}\cup\rn_{\ge1}}e^{-2\pi in\xi}\frac{\betr\xi^\alpha}{2\pi}\ d\xi\ dx\\
    &-\frac{1}{1+\alpha}\int_0^{2\pi}\int_{\d B_{\rn}}\frac{\betr\xi^\alpha}{2\pi}\ d\vol_{\d B_{\rn}}(\xi)\ dx\\
    =&\int_{-1}^1\betr\xi^\alpha d\xi+\sum_{n\in\zn\setminus\{0\}}\int_{\rn}e^{-2\pi in\xi}\betr\xi^\alpha d\xi\\
    &-\frac{1}{1+\alpha}\int_{\d B_{\rn}}\betr\xi^\alpha d\vol_{\d B_{\rn}}(\xi).
  \end{align*}
  Since $\alpha\in(-1,0)$ and $\vol_{\d B_{\rn}}$ is the sum of point measures $\delta_{-1}+\delta_1$, we obtain
  \begin{align*}
    \int_{-1}^1\betr\xi^\alpha d\xi=&2\int_0^1\xi^\alpha d\xi=\frac{2}{\alpha+1}=\frac{1}{1+\alpha}\int_{\d B_{\rn}}\betr\xi^\alpha d\vol_{\d B_{\rn}}(\xi),
  \end{align*}
  i.e.
  \begin{align*}
    \zeta\l(s\mapsto H^sH^\alpha\r)(0)=&\sum_{n\in\zn\setminus\{0\}}\int_{\rn}e^{-2\pi in\xi}\betr\xi^\alpha d\xi.
  \end{align*}
  Using that the Fourier transform of $\xi\mapsto\betr\xi^\alpha$ is 
  \begin{align*}
    \int_{\rn}e^{-2\pi ix\xi}\betr\xi^\alpha d\xi=\frac{2\sin\l(\frac{-\alpha\pi}{2}\r)\Gamma(\alpha+1)}{\betr{2\pi x}^{\alpha+1}}
  \end{align*}
  and Riemann's functional equation $\zeta_R(z)=2(2\pi)^{z-1}\sin\l(\frac{\pi z}{2}\r)\Gamma(1-z)\zeta_R(1-z)$, we obtain (in the sense of meromorphic extensions)
  \begin{align*}
    \zeta\l(s\mapsto H^sH^\alpha\r)(0)=&\sum_{n\in\zn\setminus\{0\}}\int_{\rn}e^{-2\pi in\xi}\betr\xi^\alpha d\xi
    =\sum_{n\in\zn\setminus\{0\}}\frac{2\sin\l(\frac{-\alpha\pi}{2}\r)\Gamma(\alpha+1)}{\betr{2\pi n}^{\alpha+1}}\\
    =&\frac{2\sin\l(\frac{-\alpha\pi}{2}\r)\Gamma(\alpha+1)}{(2\pi)^{\alpha+1}}\cdot 2\sum_{n\in\nn}\frac{1}{n^{\alpha+1}}\\
    =&2\ubr{2(2\pi)^{(-\alpha)-1}\sin\l(\frac{-\alpha\pi}{2}\r)\Gamma(1-(-\alpha))\zeta_R(1-(-\alpha))}_{=\zeta_R(-\alpha)}.
  \end{align*}

\end{example*}

\begin{remark*}
  Using identification via meromorphic extension of
  \begin{align*}
    \zeta_R(z)=\sum_{n\in\zn\setminus\{0\}}\frac{\sin\l(\frac{-z\pi}{2}\r)\Gamma(z+1)}{\betr{2\pi n}^{z+1}}
  \end{align*}
  and, therefore,
  \begin{align*}
    \fa z\in\cn\setminus\{-1\}:\ \sum_{n\in\zn\setminus\{0\}}\int_{\rn}e^{-2\pi in\xi}\betr\xi^z d\xi=2\zeta_R(-z)
  \end{align*}
  as well as 
  \begin{align*}
    \int_{-1}^1\betr\xi^z d\xi=&\frac{1}{1+z}\int_{\d B_{\rn}}\betr\xi^z d\vol_{\d B_{\rn}}(\xi),
  \end{align*}
  we can extend the example above to all $\alpha\in\cn\setminus\{-1\}$, i.e. 
  \begin{align*}
    \zeta_R=\l(\alpha\mapsto\frac12\zeta\l(s\mapsto H^sH^{-\alpha}\r)(0)\r).
  \end{align*}

\end{remark*}
\begin{example*}[Example ($\d^k\zeta\l(s\mapsto H^{s+\alpha}\r)(0)$ on $\rn/_{2\pi\Zn}$)]
  The spectral $\zeta$-function yields
  \begin{align*}
    \d^k\zeta\l(s\mapsto H^sH^\alpha\r)(0)=\d^k\l(z\mapsto 2\zeta_R(-z)\r)(\alpha)=(-1)^k\cdot2\d^k\zeta_R(-\alpha).
  \end{align*}
  From
  \begin{align*}
    \zeta\l(s\mapsto H^sH^\alpha\r)(z)=&\sum_{k\in\nn_0}\frac{1}{k!}\l(\int_{\Delta(\Tn)\times B_{\rn}}e^{i\theta}\sigma\l(\l(\ln H\r)^kH^\alpha\r)\ d\vol_{\Delta(\Tn)\times\d B_{\rn}}\r.\\
    &\qquad+\int_{\Delta(\Tn)\times\l(\rn_{\ge1}\times\d B_{\rn}\r)}e^{i\theta}\sigma\l(\l(\ln H\r)^kH^\alpha\r)_0\ d\vol_{\Delta(\Tn)\times\l(\rn_{\ge1}\times\d B_{\rn}\r)}\\
    &\qquad+\l.\sum_{\iota\in I}\frac{(-1)^{l_\iota+1}l_\iota!\res\l(\l(\ln H\r)^kH^\alpha\r)_\iota}{(1+d_\iota)^{l_\iota+1}}\r)z^k
  \end{align*}
  (Theorem \ref{Laurent-zeta-FIO}) we obtain
  \begin{align*}
    \d^k\zeta\l(s\mapsto H^sH^\alpha\r)(0)=&\int_0^{2\pi}\int_{-1}^1\sum_{n\in\zn}e^{-2\pi in\xi}\frac{\betr\xi^\alpha\l(\ln\betr\xi\r)^k}{2\pi}\ d\xi\ dx\\
    &+\int_0^{2\pi}\int_{\rn\setminus B_{\rn}}\sum_{n\in\zn\setminus\{0\}}e^{-2\pi in\xi}\frac{\betr\xi^\alpha\l(\ln\betr\xi\r)^k}{2\pi}\ d\xi\ dx\\
    &+\frac{(-1)^{k+1}k!}{(1+\alpha)^{k+1}}\int_0^{2\pi}\int_{\d B_{\rn}}\frac{\betr\xi^\alpha}{2\pi}\ d\vol_{\d B_{\rn}}(\xi)\ dx\\
    =&2\int_0^1\xi^\alpha\l(\ln\xi\r)^kd\xi+\sum_{n\in\zn\setminus\{0\}}\int_{\rn}e^{-2\pi in\xi}\betr\xi^\alpha\l(\ln\betr\xi\r)^kd\xi\\
    &-\frac{2\cdot(-1)^kk!}{(1+\alpha)^{k+1}}\\
    =&2\d^k\l(\beta\mapsto\int_0^1\xi^\beta d\xi\r)(\alpha)-\frac{2\cdot(-1)^kk!}{(1+\alpha)^{k+1}}\\
    &+\d^k\l(\beta\mapsto\sum_{n\in\zn\setminus\{0\}}\int_{\rn}e^{-2\pi in\xi}\betr\xi^\beta d\xi\r)(\alpha)\\
    =&2\ubr{\d^k\l(\beta\mapsto(1+\beta)^{-1}\r)(\alpha)}_{=(-1)^kk!(1+\alpha)^{-(k+1)}}-\frac{2\cdot(-1)^kk!}{(1+\alpha)^{k+1}}+\d^k\l(\beta\mapsto2\zeta_R(-\beta)\r)(\alpha)\\
    =&(-1)^k\cdot2\d^k\zeta_R(-\alpha).
  \end{align*}

\end{example*}
Finally, let us compute the residue of $\zeta\l(s\mapsto H^sH^{-1}\r)$.
\begin{example*}[Example ($\res_0\zeta\l(s\mapsto H^sH^{-1}\r)$ on $\rn/_{2\pi\Zn}$)]
  $\zeta\l(s\mapsto H^sH^{-1}\r)(z)=2\zeta_R(1-z)$ shows that $\res_0\zeta\l(s\mapsto H^sH^{-1}\r)=-2\res_1\zeta_R=-2.$ Also, using the Laurent expansion (Theorem \ref{Laurent-zeta-FIO}) of $\zeta(A)$ for $A=\l(s\mapsto H^sH^{-1}\r)$, we obtain
  \begin{align*}
    \res_0\zeta\l(s\mapsto H^sH^{-1}\r)=&-\int_0^{2\pi}\int_{\d B_{\rn}}\frac{\betr\xi^{-1}}{2\pi}\ d\vol_{\d B_{\rn}}\ dx=-2.
  \end{align*}

\end{example*}
Furthermore, we can consider shifted fractional Laplacians which do not have singular amplitudes, that is, these are actually covered by the theory we have developed so far. They will also lead to the crucial observation that helped incorporate the case of singular amplitudes through mollification and, thus, justify the example of fractional Laplacians.
\begin{example*}[Example (shifted fractional Laplacians on $\rn/_{2\pi\Zn}$)]
  Again, let $H:=\sqrt{\betr\Delta}$ on $\rn/_{2\pi\Zn}$, $h\in(0,1]$, and $G:=h+H$. Then,
  \begin{align*}
    \zeta\l(s\mapsto G^{s+\alpha}\r)(z)=&\sum_{n\in\zn}\l(h+\betr n\r)^{z+\alpha}=2\sum_{n\in\nn_0}{\l(h+n\r)^{z+\alpha}}-h^{z+\alpha}
    =2\zeta_H(-z-\alpha;h)-h^{z+\alpha}
  \end{align*}
  where $\zeta_H(z;h)$ denotes the Riemann-Hurwitz-$\zeta$-function. In order to use our formalism above (Theorem \ref{Laurent-zeta-FIO}), we will need to write $\xi\mapsto(h+\betr\xi)^\alpha$ as a series of poly-homogeneous functions. Using
  \begin{align*}
    (h+\betr\xi)^\alpha=\sum_{k\in\nn_0}{\alpha\choose k}\betr\xi^{\alpha-k}h^k
  \end{align*}
  for $\betr\xi\ge1$ yields that the kernel of $G^{z+\alpha}$
  \begin{align*}
    k_{G^{z+\alpha}}(x,y)=\sum_{n\in\zn}\int_{\rn}e^{i(x-y-2\pi n)\xi}\frac{1}{2\pi}(h+\betr\xi)^{z+\alpha} d\xi
  \end{align*}
  is, in fact, poly-$\log$-homogeneous. For $\alpha=-1$, the critical term in zero is given by the $k=0$ term of $\sum_{k\in\nn_0}{\alpha\choose k}\betr\xi^{\alpha-k}h^k$, i.e.
  \begin{align*}
    \res_0\zeta\l(s\mapsto G^{s-1}\r)=&-\int_{\d B_{\rn}}\betr\xi^{-1}d\vol_{\d B_{\rn}}(\xi)=-2.
  \end{align*}
  On the other hand,
  \begin{align*}
    \res_0\zeta\l(s\mapsto G^{s-1}\r)=&\res_0\l(z\mapsto 2\zeta_H(-z+1;h)-h^{z+\alpha}\r)
    =2\res_0\l(z\mapsto\zeta_H(-z+1;h)\r)\\
    =&-2\res_0\l(z\mapsto\zeta_H(z-1;h)\r)
    =-2\res_1\zeta_H(\cdot;h)
    =-2.
  \end{align*}
  For $\alpha\ne-1$ and $\betr\xi\ge1$,
  \begin{align*}
    (h+\betr\xi)^\alpha=\sum_{k\in\nn_0}{\alpha\choose k}h^k\betr\xi^{\alpha-k}
  \end{align*}
  implies $\alpha-k\in I_0$ if and only if $k=\alpha+1\in\nn_0$. However, since ${\alpha\choose\alpha+1}=0$ for $\alpha\in\nn_0$, we obtain $I_0=\emptyset$. 

  We will skip the computation of the Laurent coefficients at this point since they are rather long without giving much insight. The detailed computation can be found in \cite{hartung-phd}*{chapter 5} or \cite{hartung-scott}*{chapter 4}.
\end{example*}

\end{appendix}

\end{document}